\documentclass[11pt,a4paper]{amsart}

\usepackage{amsfonts,amssymb,amsmath,amsthm}
\usepackage{mathtools,esint}
\usepackage{lmodern}
\usepackage[noadjust]{cite}
\usepackage{enumitem}
\usepackage{xcolor}
\usepackage[plainpages=false,pdfpagelabels,backref=page]{hyperref}
\hypersetup{
 colorlinks=true,
 linkcolor={cyan!90!black},
 citecolor={magenta},
 urlcolor={green!40!black},
 pdftitle={Uniform coarse ellipticity, A2 weights, and boundary regularity},
 pdfauthor={Kaj Nystr\"om},
 pdfsubject={Degenerate elliptic equations, Dirichlet problems and boundary Harnack inequalities}
}

\usepackage{graphicx}

\newtheorem{thm}{Theorem}[section]
\newtheorem{lem}[thm]{Lemma}
\newtheorem{prop}[thm]{Proposition}
\newtheorem{cor}[thm]{Corollary}

\theoremstyle{definition}
\newtheorem{defn}[thm]{Definition}
\newtheorem{rem}[thm]{Remark}
\numberwithin{equation}{section}
\DeclareMathOperator*{\essosc}{ess\,osc}

\newcommand{\R}{\mathbb R}

\renewcommand{\d}{\,\mathrm{d}}
\newcommand{\esssup}{\mathop{\mathrm{ess\,sup}}}
\newcommand{\essinf}{\mathop{\mathrm{ess\,inf}}}
\newcommand{\supp}{\operatorname{supp}}
\newcommand{\avgint}{\fint}
\newcommand{\Om}{\Omega}
\newcommand{\dist}{\operatorname{dist}}
\newcommand{\osc}{\operatorname{osc}}
\newcommand{\cE}{\mathcal E}
\newcommand{\cF}{\mathcal F}
\newcommand{\cT}{\mathcal T}
\newcommand{\cS}{\mathcal S}
\newcommand{\Ha}{H^1_a}
\newcommand{\Haz}{H^1_{a,0}}
\newcommand{\Vzero}{\mathcal V_0}
\newcommand{\La}{\mathcal L_a}
\newcommand{\tr}{\operatorname{tr}}
\newcommand{\diam}{\operatorname{diam}}
\newcommand{\Lip}{\operatorname{Lip}}
\newcommand{\Id}{\mathrm I}
\newcommand{\loc}{\mathrm{loc}}
\newcommand{\Qunit}{Q^{\circ}}
\allowdisplaybreaks
\makeatletter
\renewcommand{\l@section}{%
  \@tocline{1}{3pt}{0pt}{}{\normalfont\small}}
\makeatother

\begin{document}

\title[Uniform coarse ellipticity and boundary regularity]
{Uniform coarse ellipticity, $A_2$ weights,\\
and boundary regularity}
\author{Kaj Nystr\"om}
\address{Department of Mathematics, Uppsala University, Box 480\\
SE-751 06 Uppsala, Sweden}
\email{kaj.nystrom@math.uu.se}
\subjclass[2020]{35J70, 35B45, 35B65, 31B25, 31C25, 42B25}
\renewcommand{\subjclassname}{\textup{2020} Mathematics Subject Classification}
\keywords{Degenerate elliptic equations, coarse ellipticity, NTA domains, continuous Dirichlet problem, boundary Harnack inequality, H\"older continuity, Muckenhoupt weights}

\begin{abstract}
We study symmetric divergence form elliptic equations with measurable
coefficients that may be degenerate and unbounded, under a uniform bound
on the coarse ellipticity ratio over all ambient cubes and scales.
We prove that uniform coarse ellipticity and classical $A_2$
weighted ellipticity define incomparable coefficient classes.
For every admissible choice of coarse parameters, we construct a scalar
weight with a logarithmically corrected, locally integrable singularity
along a product Cantor set. Its codimension is chosen according to the
upper coarse discount exponent, and the resulting coefficient satisfies
the uniform coarse condition but is comparable to no $A_2$ weight.
Conversely, an explicit $A_2$ weight fails the coarse condition for every
admissible parameter choice.  Under uniform coarse ellipticity we give a direct PDE treatment of boundary
regularity in the approximable energy realization. Using coarse interior
estimates, we prove boundary decay and solvability of the continuous
Dirichlet problem in bounded NTA domains. An almost-positivity argument,
with energy approximation and boundary testing justified, yields
boundary Harnack and H\"older continuity of quotients of positive
solutions with common zero boundary values. The argument requires no
additional energy Sobolev inequality and avoids the construction of
Green functions. Further examples distinguish uniform coarse control
from finite coefficient moments and from coercivity on
$\mathrm L^2(\mathrm{d} x)$. We also record interior H\"older estimates
and continuity criteria under nonuniform control along nested scales.
\end{abstract}

\maketitle

\section{Introduction and statement of main results}\label{Section_0}

Boundary regularity for elliptic equations concerns the interaction between
coefficient degeneracy, domain geometry, and the behaviour of solutions
near the boundary. In this paper we consider
\begin{equation}\label{eq:equation}
 \La u=-\operatorname{div}(a\nabla u)=0,
\end{equation}
where $a$ is measurable and symmetric and may be both degenerate and
unbounded. We give a direct treatment of boundary regularity under
uniform coarse ellipticity in nontangentially accessible domains and
prove that this coefficient class is incomparable with classical
$A_2$ weighted ellipticity.

The notion of coarse-grained ellipticity was introduced by Armstrong
and Kuusi \cite{AK}. We use the formulation and interior estimates of
Armstrong, Avelin, Kuusi and Turpeinen \cite{AAKT}, in which variational
matrices measure effective energies on simplices and their spatial
moments are discounted and summed across finer scales. Our assumption
bounds the resulting ellipticity ratio uniformly in location and scale,
providing uniform interior estimates for the boundary decay and
comparison arguments. We work in the approximable energy realization
and justify the zero extensions and boundary tests needed to apply
these estimates near the boundary.

The paper has two main components. Theorem~\ref{thm:coefficientclasses}
establishes incomparability of uniform coarse ellipticity with
classical $A_2$ weighted ellipticity.
For every admissible tuple of coarse parameters, we construct a scalar
coefficient satisfying the uniform coarse condition but comparable to
no $A_2$ weight. Conversely, a single explicit $A_2$ weight fails the
coarse condition for every admissible tuple.

Under uniform coarse ellipticity, Theorems~\ref{thm:dirichletmain} and \ref{thm:BHImain} give the continuous
Dirichlet solution, boundary Harnack, and H\"older continuity of quotients
and their logarithms. The energy and boundary-value identifications in
Subsection~\ref{sec:energy} also allow boundary comparison on
uniform domains to be deduced from the abstract Dirichlet-form
theorem of \cite{ChenBHP}. Here we give a direct PDE proof in
the stated local NTA setting, based on the almost-positivity
method of \cite{DSS}.  These boundary results apply,
in particular, to the coefficients outside the classical weighted
class constructed here. Subsection~\ref{intro:literature} details the
relation to the abstract theory and earlier coefficient examples.

\subsection{Assumptions on the coefficients}\label{intro:coefficients}

Throughout the paper, $n\geq3$ and
$a\colon\R^n\to\R^{n\times n}$ is measurable, symmetric and positive
definite almost everywhere. Our basic qualitative assumption is
\begin{equation}\label{eq:basic}
 \tr a,\ \tr a^{-1}\in\mathrm L^1_{\loc}(\R^n).
\end{equation}
Here and throughout, matrix norms are operator norms. For $R>0$, write
$Q_R(x)=x+(-R,R)^n$. All cubes used below have sides parallel to the
coordinate axes.
The coefficient is defined on the ambient space. This is relevant at the
boundary, where weak Harnack will be applied to zero extensions on cubes
meeting both sides of the domain. An extension satisfying the hypotheses
is part of the assumptions when the coefficient is initially given only
inside the domain.

Armstrong, Avelin, Kuusi and Turpeinen \cite{AAKT} establish local
boundedness, weak Harnack, and Harnack inequalities in terms of coarse
ellipticity constants. These constants are formed from variational
matrices on simplices and discounted sums across scales. Their finiteness
on a fixed ambient cube already involves all finer scales within that
cube, but does not provide a uniform bound on the coarse ratios of all
smaller ambient cubes. Uniform use of the estimates in an oscillation or
boundary comparison argument requires control of the constants as both
the center and radius vary. This is the role of our uniform hypothesis.

Fix parameters
\begin{equation}\label{eq:admissible}
 1<p,q<\infty,\qquad s,t>0,\qquad
 \theta=1-s-t-\frac{n-1}{2}\left(\frac1p+\frac1q\right)>0.
\end{equation}
The restriction $n\geq3$ is inherited from the coarse interior estimates
of \cite[Theorems A and C]{AAKT}. We retain this range throughout,
including in the comparison of coefficient classes. In particular, the
inverse-square singularities in the reverse example below are locally
integrable in this range.

For an axis-parallel cube $Q$, let $\Lambda_{s,1,p}(Q)$ and
$\lambda_{t,1,q}(Q)$ be the rescaled coarse constants of \cite{AAKT},
defined precisely in Subsection~\ref{sec:coarse} below. Set
\begin{equation}\label{eq:ratio}
 \Theta_a(Q)=\frac{\Lambda_{s,1,p}(Q)}{\lambda_{t,1,q}(Q)}.
\end{equation}
The uniform coarse ellipticity assumption that we impose is
\begingroup
\renewcommand*{\theHequation}{UC}
\begin{equation}\tag{UC}\label{eq:UC}
 0<\lambda_{t,1,q}(Q),\quad
 \Lambda_{s,1,p}(Q)<\infty
 \quad\text{for every axis-parallel cube }Q,
 \qquad
 \sup_Q\Theta_a(Q)\leq K<\infty.
\end{equation}
\endgroup
All axis-parallel cubes, with arbitrary centers and side lengths, occur
in the supremum. We write
\begin{equation}\label{eq:data}
 \mathbf d=(n,p,q,s,t,K).
\end{equation}
The global formulation keeps the notation simple. Each proof only uses
the assumption on the ambient cubes involved in that proof, as discussed
in Subsection~\ref{sec:localization}.
The bound in \eqref{eq:UC} controls an effective ellipticity ratio, rather
than the sizes of $a$ and $a^{-1}$ separately. It nevertheless implies
$|a(x)||a(x)^{-1}|\leq K$ almost everywhere, as
Proposition~\ref{prop:pointwiseanisotropy} below shows. The remaining scalar
variation need not be described by an $A_2$ weight. The explicit
construction in Theorem~\ref{thm:coefficientclasses} makes this
distinction quantitative.

\begin{rem}
A positive weight $\rho$ belongs to $A_2(\R^n)$ if
\begin{equation}\label{eq:A2definition}
 [\rho]_{A_2}
 :=\sup_Q\left(\avgint_Q\rho\d x\right)
          \left(\avgint_Q\rho^{-1}\d x\right)<\infty.
\end{equation}
By $A_2$ weighted ellipticity we mean the existence of such a weight
and fixed constants $0<c\leq C<\infty$ such that
\begin{equation}\label{eq:A2classdefinition}
 c\rho(x)|\xi|^2\leq a(x)\xi\cdot\xi
 \leq C\rho(x)|\xi|^2
 \quad\hbox{for almost every }x\hbox{ and every }\xi\in\R^n.
\end{equation}
Uniform bounds on the products of normalized coefficient moments,
as in \eqref{eq:uniformmoments} below, provide a convenient sufficient
condition for \eqref{eq:UC} in the admissible exponent range. By Proposition~\ref{prop:A2}, these bounds
also imply classical $A_2$ weighted ellipticity. The corresponding
corollary in Subsection~\ref{sec:momentproducts} therefore recovers
a case covered by the classical weighted theory. To obtain coefficients
outside that class, we use the full variational condition
\eqref{eq:UC}, which the explicit construction below verifies directly.
\end{rem}
\subsection{Energy solutions and boundary values}

We consider the bilinear form
\begin{equation}\label{eq:initialform}
 \cE(f,g)=\int_{\R^n}a\nabla f\cdot\nabla g\d x,
 \qquad f,g\in C_0^\infty(\R^n).
\end{equation}
Under \eqref{eq:basic}, this form is closable on
$\mathrm L^2(\R^n,\d x)$. We denote its closure by
$(\cE,\cF)$, retaining the notation $\cE$ for the closed form,
and equip $\cF$ with the norm
\begin{equation}\label{norma}
 \|f\|_{\cF}
 =\bigl(\|f\|_{\mathrm L^2(\R^n)}^2+\cE(f,f)\bigr)^{1/2}.
\end{equation}
The $\mathrm L^2$ term reflects our choice to realize the energy
as a closed Dirichlet form on $\mathrm L^2(\R^n,\d x)$.
This choice requires no inequality controlling the
$\mathrm L^2$ norm by the energy and adds no term to the
weak equation. For an open set $D$, a function $u$ belongs to $\cF_{\loc}(D)$
if, for every open $V\Subset D$, there is an element of $\cF$
agreeing with $u$ almost everywhere on $V$. In particular, membership
implies local $\mathrm L^2$ integrability and local finite energy,
while retaining the approximation requirement encoded by $\cF$.
These two finiteness properties alone are not asserted to characterize
$\cF_{\loc}(D)$. The definition imposes neither global integrability
nor a boundary condition.

A weak solution of $\La u=0$ in $D$ is a function
$u\in\cF_{\loc}(D)$ satisfying
\[
 \int_D a\nabla u\cdot\nabla\phi\d x=0
 \qquad\text{for every }\phi\in C_0^\infty(D).
\]
Such solutions will also be called $\La$-harmonic functions
in $D$. A weak subsolution in $D$ is a function
$u\in\cF_{\loc}(D)$ satisfying
\[
 \int_D a\nabla u\cdot\nabla\phi\d x\leq0
 \qquad\text{for every nonnegative }\phi\in C_0^\infty(D).
\]
Weak supersolutions are defined by reversing this inequality.
Closability, identification of distributional gradients, and
compatibility with the energy spaces of \cite{AAKT} are
established in Subsection~\ref{sec:energy}.
Remark~\ref{rem:energyspaces} summarizes the roles of the
different spaces.

For an open set $\Om$, let $\cF_0(\Om)$ be the closure of
$C_0^\infty(\Om)$ in the ambient form norm \eqref{norma},
where functions in $C_0^\infty(\Om)$ are extended by zero
to $\R^n$ before taking the closure. Let $W\subset\R^n$ be
open and let $u$ be a function on $\Om\cap W$. We say that
$u$ has zero local energy boundary values on
$\partial\Om\cap W$ if the zero extension of $\zeta u$ to
$\R^n$ belongs to $\cF_0(\Om)$ for every
$\zeta\in C_0^\infty(W)$, see Definition~\ref{def:zerotrace}.
This condition prescribes vanishing only on the physical
boundary $\partial\Om\cap W$ and imposes no data on the
artificial boundary $\Om\cap\partial W$. Its testing properties
are established in Subsection~\ref{sec:energy}. Under the
coefficient and geometric hypotheses used below, its equivalence
with continuous vanishing for nonnegative weak solutions is
proved in Subsection~\ref{sec:boundary}.

The basic assumption \eqref{eq:basic} suffices to construct
this energy realization, but alone guarantees neither
$\mathrm L^2$ coercivity nor uniform regularity estimates
across scales. Energy convergence need not imply
$\mathrm L^2$ convergence, even for zero boundary functions,
and multiplication by smooth cutoffs requires care because
the coefficients may be unbounded. Subsection~\ref{sec:energy}
records the standard energy calculus and verifies the approximation
and boundary testing properties needed for this realization and
the direct PDE proof.
The additional hypothesis \eqref{eq:UC} supplies the uniform
interior estimates used in our boundary arguments. It does
not, however, imply an unweighted $\mathrm L^2$ Poincar\'e
inequality, as Proposition~\ref{prop:nogap} shows.

\subsection{Geometric assumptions}\label{intro:geometry}

For a domain $\Om$, $w\in\partial\Omega$, $r>0$, we use the notation
\begin{equation}\label{eq:domainnotation}
 D_r(w)=\Om\cap B(w,r),\qquad
 \Delta_r(w)=\partial\Om\cap B(w,r),\qquad
 \delta(x)=\dist(x,\partial\Om).
\end{equation}
We use the following quantitative formulation of two-sided NTA geometry,
with interior connectivity expressed by uniform curves.

\begin{defn}\label{def:NTA}
A domain $\Om$ has NTA data $(\kappa,M,R_0)$ if
$0<\kappa\leq1/4$, $M\geq1$, $R_0>0$, and the following hold.
\begin{enumerate}
\item For $w\in\partial\Om$ and $0<r<R_0$, there are points
$A_r^+(w)$ and $A_r^-(w)$ such that
\[
 B(A_r^+(w),\kappa r)\subset\Om\cap B(w,r),\qquad
 B(A_r^-(w),\kappa r)\subset\Om^c\cap B(w,r).
\]
\item If $x,y\in\Om$ and $|x-y|<R_0$, a rectifiable curve
$\gamma\subset\Om$ joins $x$ to $y$, has length $\ell\leq M|x-y|$,
and, when parametrized by arc length from $x$, satisfies
\[
 \delta(\gamma(t))\geq M^{-1}\min\{t,\ell-t\},
 \qquad 0\leq t\leq\ell.
\]
\end{enumerate}
We write $A_r(w)=A_r^+(w)$ for an interior corkscrew point.
\end{defn}

\begin{rem}
Definition~\ref{def:NTA} is a standard quantitative formulation of the corkscrew and
interior Harnack chain geometry, with constants and localization scales
adjusted if necessary. The NTA framework originates in \cite{JK}.
In particular, Lipschitz domains have this geometry. Neither a Lipschitz
graph representation nor a condition on boundary surface measure is
assumed in the results below. Constants are allowed to depend on
$\kappa$ and $M$, while $R_0$ specifies the admissible radii.
\end{rem}

\subsection{Main results}\label{intro:results}

The following theorem constructs explicit examples to show that classical
$A_2$ weighted ellipticity and the conditions \eqref{eq:basic}
and \eqref{eq:UC} define incomparable coefficient classes.

\begin{thm}\label{thm:coefficientclasses}
For every $n\geq3$, the following assertions hold.
\begin{enumerate}
\item For every tuple $(p,q,s,t)$ satisfying \eqref{eq:admissible},
there is a scalar coefficient $a=w\Id$, with $w\geq1$ almost everywhere
and $w,w^{-1}\in\mathrm L^1_{\loc}(\R^n)$, satisfying \eqref{eq:UC}
for this prescribed tuple with $K$ depending only on $n$ and $s$,
but satisfying no comparison \eqref{eq:A2classdefinition} with an
$A_2$ weight. For fixed $n,s$, the same coefficient works for every $p,q,t$
for which $(p,q,s,t)$ is admissible.
\item The weight
\begin{equation}\label{eq:classreverseweight}
 w_0(x)={|x|^2}/{|x-e_1|^2},  \qquad x\in\R^n\setminus\{0,e_1\},
\end{equation}
with arbitrary positive values at $0$ and $e_1$, belongs to
$A_2(\R^n)$. For $a_0=w_0\Id$, there is no choice of parameters
satisfying \eqref{eq:admissible} for which both coarse sums are finite
on $Q=(-1,2)^n$.
\end{enumerate}
Consequently, for each fixed admissible tuple, the class satisfying
\eqref{eq:basic} and \eqref{eq:UC} and the class of $A_2$ weighted
elliptic coefficients are incomparable. The same holds for the union
of the coarse classes over all admissible tuples.
\end{thm}

\begin{rem}
The coefficient in the first assertion of
Theorem~\ref{thm:coefficientclasses} has a logarithmically corrected,
locally integrable singularity along a totally disconnected product
Cantor set of dimension $n-s$. For fixed $n$ and $s$, the same
coefficient works for every admissible tuple $(p,q,s,t)$.
No bound for the coarse ellipticity constant uniform as
$s\downarrow0$ is asserted. The variational mechanism is outlined in
Subsection~\ref{intro:proof}. Section~\ref{sec:coefficientclasses} contains the
full construction and the proofs of both assertions of the theorem.
\end{rem}

The following theorem gives the continuous Dirichlet solution and its
harmonic measure representation.

\begin{thm}\label{thm:dirichletmain}
Assume \eqref{eq:basic}, \eqref{eq:admissible}, and \eqref{eq:UC}.
Let $\Om$ be a bounded NTA domain with NTA data $(\kappa,M,R_0)$. Then for every $f\in C(\partial\Om)$
there is a unique
$H_\Om f\in C(\overline\Om)\cap\cF_{\loc}(\Om)$ satisfying
\eqref{eq:equation} and $H_\Om f=f$ on $\partial\Om$.
The extension operator is linear, positive, and contractive in the
supremum norm, and $H_\Om1=1$. For every $x\in\Om$ there is a unique
probability measure $\omega^x_\Om$ on $\partial\Om$ such that
\begin{equation}\label{eq:representationmain}
 H_\Om f(x)=\int_{\partial\Om}f\d\omega^x_\Om.
\end{equation}
Furthermore, weak solutions are locally H\"older continuous, and nonnegative solutions
with zero local energy boundary values satisfy a boundary H\"older
decay estimate with constants depending only on $\mathbf d$ and
the NTA data. H\"older boundary data give a boundary H\"older modulus
as specified in Proposition~\ref{prop:datamodulus} below.
\end{thm}

\begin{rem}
For arbitrary continuous boundary data, the solution is obtained as
a uniform limit of bounded variational solutions with Lipschitz
boundary data. The limit belongs to $\cF_{\loc}(\Om)$ and has finite
energy on every relatively compact subdomain, but the integral
\[
 \int_\Om a\nabla u\cdot\nabla u\d x
\]
need not be finite. The variational construction does not require
an $\mathrm L^2$ Poincar\'e inequality.  Moreover, the theorem remains
valid when the NTA assumption is replaced by boundedness,
connectedness, and the uniform exterior volume density condition
used in Subsection~\ref{sec:boundary}. In this version, the boundary
regularity constants depend on the coefficient data and the density
constant $c_e$ in \eqref{eq:exteriordensity}; $R_0$ specifies the
admissible radii.
\end{rem}

The next theorem establishes boundary comparison and H\"older
continuity of quotients and their logarithms.

\begin{thm}\label{thm:BHImain}
Assume \eqref{eq:basic}, \eqref{eq:admissible}, and \eqref{eq:UC}.
Let $\Om$ be a bounded NTA domain with NTA data $(\kappa,M,R_0)$. Then there are
$L\geq16$, $C\geq1$, and $\alpha\in(0,1)$, depending only on
$\mathbf d,\kappa,M$, with the following property.
Let $w\in\partial\Om$, $0<r<R_0/L$, and let $u,v$ be nonnegative
weak solutions in $D_{Lr}(w)$ with zero local energy boundary values
on $\Delta_{Lr}(w)$. Suppose
$u(A_r(w))>0$ and $v(A_r(w))>0$. Then $u,v>0$ in $D_r(w)$ and
\begin{equation}\label{eq:BHImain}
 C^{-1}\frac{u(A_r(w))}{v(A_r(w))}
 \leq\frac{u(x)}{v(x)}
 \leq C\frac{u(A_r(w))}{v(A_r(w))},\qquad x\in D_r(w).
\end{equation}
Furthermore, the quotient has a positive continuous extension to
$\overline\Om\cap B(w,r)$, and
\begin{equation}\label{eq:quotientmain}
 \left|\frac{u(x)}{v(x)}-\frac{u(y)}{v(y)}\right|
 \leq C\frac{u(A_r(w))}{v(A_r(w))}
       \left(\frac{|x-y|}{r}\right)^\alpha,
 \qquad x,y\in\overline\Om\cap B(w,r).
\end{equation}
Moreover,
\begin{equation}\label{eq:logmain}
 \left|\log\frac{u(x)}{v(x)}-\log\frac{u(y)}{v(y)}\right|
 \leq C\left(\frac{|x-y|}{r}\right)^\alpha,
 \qquad x,y\in\overline\Om\cap B(w,r).
\end{equation}
Continuous vanishing on $\Delta_{Lr}(w)$ may replace the local energy
boundary condition. Boundary values of the quotient refer to its
continuous extension.
\end{thm}

\begin{rem}
The assumptions $u(A_r(w))>0$ and $v(A_r(w))>0$, together with
interior Harnack chains, ensure that both solutions are strictly
positive in $D_r(w)$. The geometry-dependent factor $L$ ensures
that the chains remain within the region where the solutions are
$\La$-harmonic. Consequently, if the solutions are defined only in
$D_{2r}(w)$, the theorem gives comparison and quotient estimates
in $D_{2r/L}(w)$, normalized at $A_{2r/L}(w)$, provided this scale
is admissible and both solutions are positive at that point.
\end{rem}

\begin{rem}
Theorems~\ref{thm:dirichletmain} and \ref{thm:BHImain} apply
to the coefficients in the first assertion of
Theorem~\ref{thm:coefficientclasses}, which lie outside the classical
$A_2$ weighted class. This includes the case where the singular set
intersects the boundary, with no additional restriction on its
position relative to the boundary. The relation between this
construction and earlier fractal examples is discussed below in
Subsection~\ref{intro:literature}.
\end{rem}

\subsection{Relation to the literature}\label{intro:literature}

The coefficient examples of Theorem~\ref{thm:coefficientclasses} clarify how the framework of this paper differs from
classical weighted ellipticity. Coarse ellipticity on a fixed cube already allows coefficients
outside the $A_2$ class. Indeed, the construction in
\cite[Proposition~6.1]{AADFKM}, after adding a positive constant,
gives an integrable coefficient with a finite coarse quantity
on a fixed cube but no $\mathrm L^{1+\varepsilon}$ integrability
near an interior hyperplane, for any $\varepsilon>0$.
The reverse H\"older property of $A_2$ weights therefore rules
out comparability of this coefficient with an $A_2$ weight.  Fractal concentration already appears
in \cite[Proposition~6.2]{AADFKM}, where mollified measures supported on
Ahlfors regular sets yield coarse bounds on a fixed cube that are
uniform in the regularization parameter, despite divergence of higher
coefficient moments.  Our construction establishes uniform coarse control over all
locations and scales for a single locally integrable coefficient. For every prescribed
admissible tuple, we construct a locally integrable scalar weight
comparable to no $A_2$ weight and verify \eqref{eq:UC} on every cube.
The proof estimates the variational matrices directly, using
competitors that are constant near interior Cantor pieces and
controlling the remaining energy near simplex faces.
The ordinary $A_2$ products diverge logarithmically, whereas the coarse
ratios remain uniformly bounded. Conversely, a single explicit $A_2$
weight fails the coarse condition for every admissible tuple.
Together, these examples establish the incomparability stated in
Theorem~\ref{thm:coefficientclasses}.

Boundary regularity for uniformly elliptic equations has been developed
in both Lipschitz domains and more general domains satisfying
quantitative accessibility conditions. For harmonic functions in
Lipschitz domains, the boundary Harnack principle was established
independently by Ancona \cite{Ancona}, Dahlberg \cite{Dahlberg},
and Wu \cite{Wu}. Dahlberg also proved quantitative absolute
continuity of harmonic measure with respect to surface measure
on Lipschitz boundaries.  Caffarelli, Fabes,
Mortola and Salsa \cite{CFMS} proved boundary comparison estimates for
nonnegative solutions of uniformly elliptic divergence form equations.
Jerison and Kenig \cite{JK} introduced NTA domains and developed the
boundary theory of positive harmonic functions in this setting.
The corkscrew and Harnack chain conditions provide a geometric framework
for comparison without requiring local graph representations or
regularity assumptions on boundary surface measure.

For degenerate coefficients comparable to a Muckenhoupt $A_2$ weight,
Fabes, Kenig and Serapioni \cite{FKS} established local regularity.
Fabes, Jerison and Kenig developed the corresponding boundary theory
through a weighted Wiener criterion \cite{FJK} and boundary comparison
estimates \cite{FJKB}. By Proposition~\ref{prop:A2}, the uniform moment
products in \eqref{eq:uniformmoments} place the coefficient in this
classical weighted setting. Corollary~\ref{cor:momentcriterion}
therefore recovers a case covered by that theory. Its exponent
restriction is needed to verify the coarse hypothesis, but is
unnecessary for the classical weighted boundary conclusions.

The analytic framework used here combines developments in
nonuniform ellipticity and coarse-graining. Bella and Sch\"affner
\cite{BS} enlarged the coefficient moment range for local boundedness
and Harnack estimates. Armstrong and Kuusi \cite{AK} introduced
coarse-grained ellipticity in their work on high-contrast
homogenization. Armstrong, Avelin, De Filippis, Kuusi and Mingione
\cite{AADFKM} developed a deterministic local boundedness and Harnack
theory in this framework, with sufficient conditions expressed through
negative Sobolev regularity. Armstrong, Avelin, Kuusi and Turpeinen
\cite{AAKT} established local boundedness, weak Harnack, and Harnack
estimates in the optimal coefficient range, using spatial moments
of variational matrices discounted and summed across scales.
These latter estimates are the analytic inputs of the present paper.
We impose uniform control of the resulting coarse ellipticity ratio
as both the location and scale vary.

Boundary Harnack has also been developed in the abstract setting
of Dirichlet forms. Barlow and Murugan
\cite[Theorem~1.1]{BM} prove boundary comparison on inner uniform
domains under an elliptic Harnack inequality and an additional
Green function assumption. Building on the Green function theory
of Barlow, Chen and Murugan \cite{BCM}, Chen
\cite[Theorem~1.1]{ChenBHP} proves boundary Harnack on uniform
domains for regular strongly local symmetric Dirichlet forms
satisfying elliptic Harnack, on complete metric doubling spaces
equipped with a Radon measure of full support. No separate
spectral gap or Green function hypothesis is imposed in this
later theorem.

In our Euclidean setting, \eqref{eq:UC} supplies the elliptic
Harnack inequality through \cite{AAKT}. The energy and testing results in Subsection~\ref{sec:energy}
identify our weak solutions with harmonic functions for the
Dirichlet form and show that our cutoff boundary condition
provides the local zero-boundary representatives required
by the abstract theory. On uniform domains, boundary comparison therefore
also follows from \cite{ChenBHP}, with H\"older continuity of
quotients obtained by oscillation reduction.

Our direct proof treats the local NTA formulation with the stated
dependence on the geometric and coarse ellipticity data by adapting
the almost-positivity method of De Silva and Savin
\cite[Sections~2 and~3.2]{DSS} to the approximable degenerate
energy class. We justify the truncation, zero-extension, and
boundary testing steps and use interior Harnack and weak Harnack
estimates together with boundary decay. This argument avoids
constructing Green functions.

\subsection{Discussion of the proofs}\label{intro:proof}

For the first assertion of Theorem~\ref{thm:coefficientclasses},
we choose a totally disconnected product Cantor set $E$ of
codimension $s$, the upper coarse discount exponent. Set
$a=w\Id$ and $d_E(x)=\dist(x,E)$, with
\[
 w(x)=\Phi_s(d_E(x))\quad\text{for }x\notin E,\qquad
 \Phi_s(r)={16r^{-s}}{(4+s\log(1/r))^{-2}},
 \qquad 0<r<1,
\]
and $\Phi_s(r)=1$ for $r\geq1$, $w=1$ on $E$. The logarithmic
correction makes this critical power singularity locally integrable,
while the $A_2$ products on balls centered on $E$ grow like
$4+s\log(1/R)$. To prove that $a$ satisfies \eqref{eq:UC}, we note that the upper coarse ellipticity quantity in \eqref{eq:Lambdadef} below is assembled from variational matrices $a_T$ on the simplices
of fixed triangulations at successive scales. By
\eqref{eq:uppermatrix}, $a_Te\cdot e$ is the minimum average
energy among admissible functions with boundary values
$e\cdot x$. Constructing a function with these boundary values
therefore gives an upper bound for this quadratic form. For each such
simplex $T$ of scale $R$, we construct an
admissible competitor $h_e$ by modifying the affine function
$e\cdot x$ to make it constant near interior Cantor pieces. Near $E$, its
remaining gradient is confined to
$\dist(x,\partial T)\leq C d_E(x)$.
Uniform hyperplane concentration estimates yield an additional
volume exponent $\delta=(n-s)/n$, making the energy contributions
summable and giving
\[
 \avgint_T w|\nabla h_e|^2\d x
 \leq C|e|^2\Phi_s\bigl(\dist(T,E)+R\bigr).
\]
This removes the logarithmic loss in ordinary coefficient averages,
uniformly even when simplex faces intersect $E$. For an ambient cube $Q$, put
$F_Q=\Phi_s(\dist(Q,E)+\ell(Q))$.
At generation $k$, the upper variational matrices are bounded
by $CF_Q3^{ks}$ and the inverse lower matrices by $CF_Q^{-1}$.
The upper coarse discount combines with the square-root growth
to give the summable factor $3^{-ks/2}$. Hence
\[
 \Lambda_{s,1,p}(Q)\leq CF_Q,
 \qquad
 (\lambda_{t,1,q}(Q))^{-1}\leq CF_Q^{-1}.
\]
Cancellation of $F_Q$ proves \eqref{eq:UC} uniformly over all
cubes. The coefficient and $C$ depend only on $n,s$, with no
uniform bound asserted as $s\downarrow0$.
Section~\ref{sec:coefficientclasses} gives the construction and
the reverse example proving incomparability with classical
$A_2$ weighted ellipticity.

The proofs of Theorems~\ref{thm:dirichletmain} and
\ref{thm:BHImain} separate the energy construction from the
quantitative estimates. Local integrability of $a$ and $a^{-1}$ gives a closable
energy form and identifies its elements with functions having
distributional gradients. On bounded sets, an energy completion with
zero boundary values controls an $\mathrm L^1$ norm. This suffices for
variational existence without an $\mathrm L^2(\d x)$ Poincar\'e
inequality. Truncation and bounded cutoff approximation then provide
the tests needed for comparison and zero extension.

Uniform coarse ellipticity supplies scale-independent interior Harnack
and weak Harnack constants through \cite{AAKT}. Interior oscillation
decay follows from Harnack. At the boundary, a nonnegative subsolution
with zero energy boundary values extends by zero to an ambient
subsolution. The exterior corkscrew supplies a fixed positive fraction
of its zero set in boundary cubes. Weak Harnack applied to the
complement of this subsolution gives boundary contraction. Variational
solutions with Lipschitz data therefore attain their boundary values
continuously. Uniform approximation and the maximum principle give the
continuous Dirichlet solution and its harmonic measure representation.

Boundary comparison uses the same local estimates together with
interior chains. A growth contradiction first proves the Carleson
estimate. The almost-positivity step is then applied to a signed
linear combination of the two solutions. Its negative part is supported
in a thin boundary layer and decays exponentially across that layer.
A short Harnack chain propagates positivity to the core at the next
scale with a constant independent of the layer thickness. This allows
the constants to be chosen successively and the signed estimate to be
iterated. Recentring at a nearest boundary point yields comparison
throughout the smaller domain. Boundary and interior oscillation
contractions for the quotient give its two-point H\"older estimate.
The positive lower bound then yields the corresponding logarithmic
estimate.

\subsection{Organization of the paper}

Section~\ref{sec:framework} develops the analytic framework. It introduces
the coarse quantities and local estimates, constructs the energy spaces
and variational solutions, and proves interior H\"older continuity.
Section~\ref{sec:coefficientclasses} analyzes the
coefficient assumptions, constructs the Cantor example, and proves
incomparability with $A_2$ weighted ellipticity. Section~\ref{sec:boundarydirichlet} establishes boundary testing and decay,
then solves the continuous Dirichlet problem and constructs harmonic
measure. Section~\ref{sec:comparison} contains the complete boundary
comparison argument, from Harnack chains and the Carleson estimate to
almost positivity, boundary Harnack, and H\"older continuity of quotients
and their logarithms.
Section~\ref{sec:examples} gives further examples and discusses
localization and the geometric assumptions. The interior continuity
criteria under nonuniform control are included in
Subsection~\ref{sec:goodscales}.

\section{Coarse ellipticity, energy solutions and interior regularity}\label{sec:framework}

We develop the analytic framework used in the coefficient estimates
and boundary arguments.
We first define the coarse variational quantities and recall the local
estimates they control. We then construct the approximable energy spaces,
justify truncation and cutoff tests, and establish variational existence.
Finally, uniform coarse ellipticity yields interior oscillation decay
and H\"older continuity. The coarse variational definitions and energy approximation results
will be used in Section~\ref{sec:coefficientclasses}. The testing
results and interior estimates also prepare the boundary arguments
in Sections~\ref{sec:boundarydirichlet} and \ref{sec:comparison}.

\subsection{Coarse ellipticity and the local estimates}\label{sec:coarse}

For the auxiliary energy spaces and variational quantities discussed below, let
$U$ be a bounded connected Lipschitz domain. The boundary results
will be formulated on NTA domains $\Om$, which need not be Lipschitz. Given $U$, let $\Ha(U)$ be the completion
of functions in $C^\infty(U)\cap\mathrm L^1(U)$ with finite energy in the norm
\begin{equation}\label{eq:Hanorm}
 \|h\|_{\Ha(U)}^2
 =\left|\avgint_U h\d x\right|^2+\cE_U(h,h),
 \qquad
 \cE_U(h,h)=\int_U a\nabla h\cdot\nabla h\d x.
\end{equation}
Let $\Haz(U)$ be the closure of $C_0^\infty(U)$ in this space.
For every $h\in C^\infty(U)\cap\mathrm L^1(U)$ of finite energy,
weighted Cauchy-Schwarz and the $W^{1,1}$ Poincar\'e inequality give
\begin{align*}
 \|\nabla h\|_{\mathrm L^1(U)}
 &\leq\left(\int_U\tr a^{-1}\d x\right)^{1/2}
       \cE_U(h,h)^{1/2},
\end{align*}
and
\begin{align*}
 \|h\|_{\mathrm L^1(U)}
 &\leq |U|\left|\avgint_U h\d x\right|
       +C(U)\left(\int_U\tr a^{-1}\d x\right)^{1/2}
       \cE_U(h,h)^{1/2}.
\end{align*}
Thus a sequence Cauchy in $\Ha(U)$ is Cauchy in $W^{1,1}(U)$.
The gradient identification argument in the proof of
Lemma~\ref{lem:closability} below then identifies the completions
with functions in $W^{1,1}(U)$ and their distributional gradients.
This is precisely the energy norm and smooth completion used in
\cite[Section 2.1]{AAKT}. The matrices below are those of
\cite[Section 2.2]{AAKT}, and the triangulations and coarse constants
agree with \cite[Section 3]{AAKT} after translation and dilation.
Lemma~\ref{lem:boundedcompletion} relates these spaces to the ambient
form realization used for boundary values.
An $\Ha$-solution in $U$ means a function $h\in\Ha(U)$ satisfying
the weak equation against every test in $C_0^\infty(U)$.
The terms $\Ha$-subsolution and $\Ha$-supersolution refer to the
corresponding inequalities for nonnegative smooth tests, with the
same sign convention as in the introduction. This terminology
distinguishes the $\Ha$ classes from those defined through
$\cF_{\loc}$. Remark~\ref{rem:energyspaces} proves local
equivalence of the two notions of solution under \eqref{eq:UC}.

For $e\in\R^n$, define symmetric positive definite matrices $a_U$ and
$a_{*,U}$ by
\begin{align}
 a_Ue\cdot e
 &=\inf\left\{\avgint_U a\nabla h\cdot\nabla h\d x
       :h-e\cdot x\in\Haz(U)\right\},\label{eq:uppermatrix}\\
 a_{*,U}^{-1}e\cdot e
 &=\sup_{h\in\Ha(U)}\avgint_U
       \bigl(2e\cdot\nabla h-a\nabla h\cdot\nabla h\bigr)\d x.
       \label{eq:lowermatrix}
\end{align}
The infimum in \eqref{eq:uppermatrix} is attained by the unique
$\Ha$-solution $h$ in $U$ satisfying
$h-e\cdot x\in\Haz(U)$. This is the $\La$-harmonic replacement
of the affine function $x\mapsto e\cdot x$ in $U$. The supremum in \eqref{eq:lowermatrix} is unchanged if it is restricted
to $\Ha$-solutions, as shown in \cite[Propositions 2.1-2.3]{AAKT}.
The variational definitions also give
\begin{equation}\label{eq:matrixaveragebounds}
 a_U\leq\avgint_U a\d x,
 \qquad a_{*,U}^{-1}\leq\avgint_U a^{-1}\d x
\end{equation}
in the sense of quadratic forms. The first inequality follows by
using the affine competitor $h(x)=e\cdot x$, and the second follows
by completing the square in \eqref{eq:lowermatrix}.

Let $\Qunit=(-1/2,1/2)^n$. Divide each triadic subcube of side $3^{-k}$
into its $n!$ coordinate-order simplices, and let $\mathcal T_k$ denote the
resulting triangulation of $\Qunit$. Explicitly, in a cube
$x_0+(0,\ell)^n$, these simplices are
\[
 T_\pi=\{x_0+y:0<y_{\pi(1)}<\cdots<y_{\pi(n)}<\ell\},
\]
where $\pi$ ranges over all permutations of $\{1,\ldots,n\}$. This is
the triangulation of \cite[equations (3.1)-(3.2)]{AAKT}, after translation
and dilation. All simplices in $\mathcal T_k$ have the same volume.
If $Q=x_Q+\ell(Q)\Qunit$, perform the same construction
after the affine map from $\Qunit$ onto $Q$, and denote the triangulation
by $\mathcal T_k(Q)$. Set
\begin{align}
 (\Lambda_{s,1,p}(Q))^{1/2}
 &=(1-3^{-s})\sum_{k=0}^\infty 3^{-ks}
 \left(\frac1{\#\mathcal T_k(Q)}
       \sum_{T\in\mathcal T_k(Q)}|a_T|^p\right)^{1/(2p)},
       \label{eq:Lambdadef}\\
 (\lambda_{t,1,q}(Q))^{-1/2}
 &=(1-3^{-t})\sum_{k=0}^\infty 3^{-kt}
 \left(\frac1{\#\mathcal T_k(Q)}
       \sum_{T\in\mathcal T_k(Q)}|a_{*,T}^{-1}|^q\right)^{1/(2q)}.
       \label{eq:lambdadef}
\end{align}
Here and below matrix norms are operator norms. For a constant matrix $A$,
these definitions give $\Lambda=|A|$ and $\lambda=|A^{-1}|^{-1}$. Recall that $\Theta_a(Q)$ is introduced in \eqref{eq:ratio}. If $\widehat a(z)=a(x_Q+\ell(Q)z)$, change of variables in the
variational definitions gives the same matrices on corresponding
simplices. Thus the coarse constants for $a$ on $Q$ equal those for
$\widehat a$ on $\Qunit$. This covariance concerns simultaneous
transformation of the coefficient and cube. Multiplication of $a$ by a positive scalar
multiplies both $\Lambda$ and $\lambda$ by that scalar and leaves $\Theta$
unchanged.

\begin{rem} The use of simplices in the construction above reflects the piecewise affine interpolation and
cellwise harmonic replacement in \cite{AAKT}. In our coefficient
estimates below, we may bound $a_Te\cdot e$ using any competitor in
$e\cdot x+\Haz(T)$, without requiring a piecewise affine construction.
\end{rem}

The next proposition recalls the local elliptic estimates from \cite{AAKT}
that will be used throughout the paper. We use their coarse variational
formulation and impose no separate Besov regularity assumption on the
coefficients.

\begin{prop}\label{prop:AAKT}
Assume \eqref{eq:basic} on a neighborhood of a cube $Q$,
\eqref{eq:admissible}, and $\Theta_a(Q)<\infty$. Set
\begin{equation}\label{eq:mstar}
 m=\frac{2q}{q+1},\qquad
 m_t^*=\frac{nm}{n-(1-t)m}.
\end{equation}
For every $0<\eta\leq m_t^*/2$ there is a constant $C_\eta$, depending only
on $n,p,q,s,t,\eta$, such that every nonnegative $\Ha$-supersolution
$h\in\Ha(Q)$ satisfies
\begin{equation}\label{eq:WH}
 \left(\avgint_{(5/8)Q}h^\eta\d x\right)^{1/\eta}
 \leq \exp\bigl(C_\eta\sqrt{\Theta_a(Q)}\bigr)
          \essinf_{(1/2)Q}h.
\end{equation}
Furthermore, there is $C_0=C_0(n,p,q,s,t)$ such that every nonnegative
$\Ha$-solution $h\in\Ha(Q)$ satisfies
\begin{equation}\label{eq:EHIcube}
 \esssup_{(1/2)Q}h
 \leq \exp\bigl(C_0\sqrt{\Theta_a(Q)}\bigr)
          \essinf_{(1/2)Q}h.
\end{equation}
Every $\Ha$-subsolution in $Q$ is locally bounded above.
\end{prop}

Proposition~\ref{prop:AAKT} is \cite[Theorems A and C]{AAKT}, after affine rescaling.
It is an external input throughout this paper. The comparison of the
coarse constants in \cite[equation (3.14)]{AAKT} gives
$\Theta_a(Q)\geq1$. Under \eqref{eq:UC}, fix
\begin{equation}\label{eq:Hconstants}
 H=\max\{4,\exp(C_0\sqrt K)\},\qquad
 H_\eta=\max\{2,\exp(C_\eta\sqrt K)\}.
\end{equation}
For the boundary arguments we fix once and for all
$\eta=\min\{1,m_t^*/2\}$. Thus $H_\eta$ and every boundary contraction
constant depend only on the data in \eqref{eq:data} and the indicated
geometric constants. The harmless lower bounds ensure that the oscillation exponents used below
lie in a convenient range. Under \eqref{eq:UC}, the inequalities also apply
to weak solutions, subsolutions, and supersolutions in $\cF_{\loc}(Q)$,
without assuming membership in $\Ha(Q)$.
By Lemma~\ref{lem:boundedcompletion}, their restrictions to strictly
smaller concentric cubes belong to the corresponding $\Ha$ classes.
Apply the inequalities there with the same uniform constants, and
then let those cubes increase to $Q$.

The following lemma bounds the coarse ellipticity ratio by normalized
coefficient moments.

\begin{lem}\label{lem:moments}
If $|a|\in \mathrm L^p(Q)$ and $|a^{-1}|\in \mathrm L^q(Q)$, then
\begin{equation}\label{eq:momentdomination}
 \Theta_a(Q)\leq
 \left(\avgint_Q|a|^p\d x\right)^{1/p}
 \left(\avgint_Q|a^{-1}|^q\d x\right)^{1/q}.
\end{equation}
\end{lem}

\begin{proof}
By \eqref{eq:matrixaveragebounds} and Jensen's inequality,
\[
 \frac1{\#\mathcal T_k(Q)}\sum_{T\in\mathcal T_k(Q)}|a_T|^p
 \leq\avgint_Q|a|^p\d x.
\]
Insert this estimate into \eqref{eq:Lambdadef} and sum the geometric series.
This gives $\Lambda_{s,1,p}(Q)\leq(\avgint_Q|a|^p)^{1/p}$.
The same argument with $a_{*,T}^{-1}$ gives
$\lambda_{t,1,q}(Q)^{-1}\leq(\avgint_Q|a^{-1}|^q)^{1/q}$.
Their product proves the assertion.
\end{proof}

\begin{rem} The normalized integrals in \eqref{eq:momentdomination} are evaluated anew
on every cube. Finite integrals on a fixed large set do not give a uniform
bound for these products on all its subcubes.
\end{rem}

\subsection{Energy spaces, truncation and \texorpdfstring{$\La$}{L_a}-harmonic replacement}\label{sec:energy}

We record standard energy calculus and verify the relations between
the spaces used for variational solutions, the local estimates of
\cite{AAKT}, and boundary testing. The concrete identifications and
approximation arguments are included to make the direct proof
self-contained. The energy construction uses only \eqref{eq:basic}.
Terminology for closed Dirichlet forms follows \cite{FOT}.

The following lemma proves closability and identifies the energy representation
of the closed form. Closability ensures that the completion in the form norm
can be identified with a subspace of $\mathrm L^2(\R^n)$,
with energy independent of the chosen approximating sequence.

\begin{lem}\label{lem:closability}
Under \eqref{eq:basic}, the form \eqref{eq:initialform} is closable on
$\mathrm L^2(\R^n,\d x)$. Its closure $(\cE,\cF)$ is a symmetric,
regular, strongly local Dirichlet form. Every $f\in\cF$ belongs to
$W^{1,1}_{\loc}(\R^n)$ and
\begin{equation}\label{eq:closedenergy}
 \cE(f,f)=\int_{\R^n}a\nabla f\cdot\nabla f\d x.
\end{equation}
\end{lem}

\begin{proof}
By \eqref{eq:basic}, the initial form on $C_0^\infty(\R^n)$
is well defined, symmetric and nonnegative. Its domain
$C_0^\infty(\R^n)$ is dense in $\mathrm L^2(\R^n)$.
We first prove closability and identify the distributional
gradients of functions in the resulting closed domain $\cF$.
We then verify that the closed form has the Markov property
and is regular and strongly local.

To prove closability, let $f_j\in C_0^\infty(\R^n)$ satisfy
$f_j\to0$ in $\mathrm L^2(\R^n)$, and suppose that
$a^{1/2}\nabla f_j$ is Cauchy in $\mathrm L^2(\R^n,\R^n)$.
Denote its limit by $G$. For every compact set $E\subset\R^n$,
Cauchy-Schwarz gives
\begin{equation}\label{eq:gradientLone}
 \|\nabla f_j-a^{-1/2}G\|_{\mathrm L^1(E)}
 \leq\left(\int_E\tr a^{-1}\d x\right)^{1/2}
       \|a^{1/2}\nabla f_j-G\|_{\mathrm L^2(E)}.
\end{equation}
The right-hand side tends to zero by \eqref{eq:basic}. Thus
$\nabla f_j\to a^{-1/2}G$ in $\mathrm L^1_{\loc}(\R^n)$.
On the other hand, $f_j\to0$ in $\mathrm L^2$ implies that
$\nabla f_j\to0$ in the sense of distributions. Uniqueness of the
distributional limit therefore gives $a^{-1/2}G=0$ almost everywhere.
Since $a$ is positive definite almost everywhere, $G=0$.
Consequently $\cE(f_j,f_j)\to0$, which proves closability.

More generally, let $f\in\cF$ and choose
$f_j\in C_0^\infty(\R^n)$ converging to $f$ in the form norm.
Then $f_j\to f$ in $\mathrm L^2(\R^n)$, and
$a^{1/2}\nabla f_j$ converges to some
$G\in\mathrm L^2(\R^n;\R^n)$.
For every bounded open set $V\Subset\R^n$, we have
$f_j\to f$ in $\mathrm L^1(V)$, while
\eqref{eq:gradientLone} gives
\[
 \nabla f_j\longrightarrow a^{-1/2}G
 \quad\text{in }\mathrm L^1(V;\R^n).
\]
Passing to the limit in the integration-by-parts identity yields
\[
 \int_V f\,\partial_i\varphi\d x
 =-\int_V(a^{-1/2}G)_i\,\varphi\d x,
 \qquad \varphi\in C_0^\infty(V),\quad i=1,\ldots,n.
\]
Thus $f\in W^{1,1}(V)$ and its distributional gradient is
$a^{-1/2}G$. Since $V$ was arbitrary,
$f\in W^{1,1}_{\loc}(\R^n)$ and $G=a^{1/2}\nabla f$.
Passage to the limit in the energies now proves
\eqref{eq:closedenergy}.

We next verify the Markov property. This requires that the truncation
$T(t)=\min\{1,\max\{t,0\}\}$ satisfy
\[
 T(f)\in\cF,
 \qquad
 \cE(T(f),T(f))\leq\cE(f,f)
 \qquad\text{for every }f\in\cF.
\]
On the smooth core, composition with a smooth contraction $S$
satisfying $S(0)=0$ and $|S'|\leq1$ preserves the core and does
not increase the energy. Approximate $T$ pointwise by smooth
contractions with these properties. For each smooth compactly
supported $f$, the compositions converge to $T(f)$ in
$\mathrm L^2$ and have bounded form norms. Weak compactness
in $\cF$ and lower semicontinuity of the energy therefore give
membership and the energy bound. Applying the same argument
to core approximants of an arbitrary $f\in\cF$ proves the
Markov property of the closure. Hence the closure is a Dirichlet
form and has the normal contraction property. The gradient
representation and the ordinary $W^{1,1}$ chain rule identify
the gradients of its truncations. By construction, $C_0^\infty(\R^n)$ is dense in $\cF$ in the form
norm. It is also uniformly dense in $C_0(\R^n)$, the space of
continuous compactly supported functions. These two density
properties prove regularity.

Finally, polarizing \eqref{eq:closedenergy} gives
\[
 \cE(f,g)=\int_{\R^n}a\nabla f\cdot\nabla g\d x,
 \qquad f,g\in\cF.
\]
To verify strong locality, let $f,g\in\cF$, suppose that $g$ has
compact support, and assume that $f$ is constant in an open
neighborhood of $\supp g$. Then $\nabla f=0$ almost everywhere
in that neighborhood, while $\nabla g=0$ almost everywhere
outside $\supp g$. The integrand therefore vanishes almost
everywhere, so $\cE(f,g)=0$. This proves strong locality.
\end{proof}

\begin{rem}
The Markov property means that truncation to $[0,1]$ preserves
$\cF$ and does not increase the energy. For a closed symmetric
form, the same holds for every normal contraction, that is,
every Lipschitz map $S:\R\to\R$ with $S(0)=0$ and
$\Lip(S)\leq1$. This includes the positive and negative parts.
More generally, every Lipschitz map $S$ fixing zero preserves
$\cF$ and satisfies
\[
 \cE(S(f),S(f))\leq\Lip(S)^2\cE(f,f),
\]
by rescaling the normal contraction property. This also covers
the scaled truncations used in the testing arguments below.
\end{rem}

For an open set $D$, define
\begin{equation}\label{eq:Fzero}
 \cF_0(D)=\overline{C_0^\infty(D)}^{\|\cdot\|_{\cF}},
\end{equation}
with functions extended by zero to $\R^n$. All zero boundary assertions
below refer to this closure. The same core argument as above shows
that the restricted form on $\cF_0(D)$ is Markovian. In particular,
Lipschitz maps fixing zero preserve $\cF_0(D)$ with the same
energy bound as in the preceding remark. All cutoff and
truncation approximations below are made within this closure.

We record the strong approximation facts used in the boundary tests.

\begin{lem}\label{lem:strongtruncation}
Let $D$ be open and let $f_j,f\in W^{1,1}_{\loc}(D)$ have finite
weighted Dirichlet energy on $D$, that is,
\[
 \cE_D(f,f)<\infty,\qquad
 \cE_D(f_j,f_j)<\infty\quad\text{for every }j.
\]
Suppose that $f_j\to f$ in measure and
\[
 \|a^{1/2}\nabla(f_j-f)\|_{\mathrm L^2(D)}\longrightarrow0.
\]
If $T$ is Lipschitz and piecewise $C^1$ with finitely many breakpoints, then the same energy convergence holds for $T(f_j)-T(f)$.
If $T(0)=0$, $f_j\to f$ in $\mathrm L^r(D)$, and $1\leq r<\infty$,
then the compositions also converge in $\mathrm L^r(D)$. Furthermore, every bounded $f\in\cF_0(D)$ admits a sequence
$f_j\in C_0^\infty(D)$ converging in the form norm, almost everywhere,
and uniformly bounded by $\|f\|_\infty+1$. If $f\geq0$, the
approximants can be chosen nonnegative.
\end{lem}

\begin{proof}
It suffices to prove the first assertion along a subsequence converging
almost everywhere, since every subsequence has such a further
subsequence. Choose a bounded representative of $T'$ at the finitely
many breakpoints. The chain rule gives
\[
 a^{1/2}\nabla\bigl(T(f_j)-T(f)\bigr)
 =T'(f_j)a^{1/2}\nabla(f_j-f)
  +\bigl(T'(f_j)-T'(f)\bigr)a^{1/2}\nabla f.
\]
The first term tends to zero in $\mathrm L^2$. Away from the
breakpoint level sets of $f$, the second tends to zero by dominated
convergence. On those level sets $\nabla f=0$ almost everywhere by
the $W^{1,1}$ chain rule, so the same conclusion holds there.
The $\mathrm L^r$ assertion follows from the Lipschitz bound.

For bounded $f\in\cF_0(D)$, start with a smooth core sequence and
truncate it to $[-M,M]$, where $M=\|f\|_\infty+1$.
The first assertion and $\mathrm L^2$ convergence show that the
truncated sequence still converges in the form norm. For each smooth
core function, approximate this scalar truncation by smooth
contractions fixing zero, with uniformly bounded derivatives and
values in $[-M,M]$. The gradient convergence follows by dominated
convergence, since the gradient of the core function vanishes almost
everywhere on each breakpoint level set. A diagonal choice gives
smooth compactly supported approximants, with supports contained in
those of the corresponding core functions. For $f\geq0$, use the truncation
to $[0,M]$ and smooth nonnegative approximations fixing zero.
Passing to a subsequence gives almost everywhere convergence.
\end{proof}

The next lemma justifies multiplication by cutoffs and the use of local energy
functions as test functions.

\begin{lem}\label{lem:products} The following statements are true.
\begin{enumerate}
\item If $f\in\cF$ is bounded and $\eta\in C_0^\infty(\R^n)$,
then $\eta f\in\cF$ and
\begin{equation}\label{eq:productenergy}
 \cE(\eta f,\eta f)
 \leq2\int_{\R^n}\eta^2a\nabla f\cdot\nabla f\d x
       +2\int_{\R^n}f^2a\nabla\eta\cdot\nabla\eta\d x.
\end{equation}
If $D$ is open and $f\in\cF_0(D)$ is bounded, then
$\eta f\in\cF_0(D)$.

\item Let $G\subset\R^n$ be open and let $u$ be a weak solution
of $\La u=0$ in $G$. Then
\[
 \int_G a\nabla u\cdot\nabla\psi\d x=0
\]
for every bounded $\psi\in\cF_{\loc}(G)$ that vanishes almost
everywhere outside a compact subset of $G$. Thus the weak
equation may be tested with such energy functions, in addition
to functions in $C_0^\infty(G)$.

\item Let $\Om,W\subset\R^n$ be open and let $u$ be a weak
solution of $\La u=0$ in $\Om\cap W$. Assume that
\[
 \int_{\Om\cap W'}a\nabla u\cdot\nabla u\d x<\infty
 \qquad\text{for every open }W'\Subset W.
\]
Then
\[
 \int_{\Om\cap W}a\nabla u\cdot\nabla\psi\d x=0
\]
for every bounded $\psi\in\cF_0(\Om)$ that vanishes almost
everywhere outside a compact subset of $W$. Such a test
function may have support meeting $\partial\Om\cap W$.
Its zero boundary values are encoded by membership in
$\cF_0(\Om)$, and no zero boundary condition on $u$ is
required for this testing assertion.
\end{enumerate}
For weak subsolutions and supersolutions, the identities in
\emph{(ii)} and \emph{(iii)} hold with $\leq0$ and $\geq0$,
respectively, for nonnegative $\psi$ satisfying the stated conditions.
\end{lem}

\begin{proof}
Choose uniformly bounded $f_j\in C_0^\infty(\R^n)$
converging to $f$ in the form norm and almost everywhere,
as supplied by Lemma~\ref{lem:strongtruncation}.
The product rule gives
\[
 a^{1/2}\nabla\bigl(\eta(f_j-f)\bigr)
 =\eta a^{1/2}\nabla(f_j-f)
   +(f_j-f)a^{1/2}\nabla\eta.
\]
The first term tends to zero in $\mathrm L^2$ by energy
convergence. The second does so by dominated convergence,
using the uniform bounds on $f_j$ and the integrability of
$a\nabla\eta\cdot\nabla\eta$. Also, $\eta f_j\to\eta f$
in $\mathrm L^2$. Closedness therefore gives $\eta f\in\cF$
and convergence in the form norm. The product rule yields
\eqref{eq:productenergy}. If $f\in\cF_0(D)$, choose
$f_j\in C_0^\infty(D)$. Then $\eta f_j\in C_0^\infty(D)$,
which proves $\eta f\in\cF_0(D)$.

For the interior testing assertion, let $K\Subset G$ be a
compact set outside which $\psi$ vanishes almost everywhere.
Choose an open set $V$ with $K\Subset V\Subset G$ and an
ambient form function $f\in\cF$ agreeing with $\psi$ almost
everywhere in $V$. Set
\[
 g=\max\{-M,\min\{f,M\}\},
 \qquad M>\|\psi\|_{\mathrm L^\infty(G)}.
\]
By the truncation property, $g\in\cF$, and $g=\psi$ almost
everywhere in $V$. Choose $\chi\in C_0^\infty(V)$ with
$0\leq\chi\leq1$ and $\chi=1$ near $K$. Then $\chi g$
is the zero extension of $\psi$.

Choose uniformly bounded $g_j\in C_0^\infty(\R^n)$
converging to $g$ in the form norm and almost everywhere,
as supplied by Lemma~\ref{lem:strongtruncation}.
The first part of the proof shows that $\chi g_j$ converges
to $\chi g$ in the form norm. Moreover,
$\chi g_j\in C_0^\infty(G)$ and all these functions are
supported in the fixed compact set $\supp\chi$.
Since $u$ has finite energy on a neighborhood of this set,
weighted Cauchy-Schwarz gives
\[
 0=\int_G a\nabla u\cdot\nabla(\chi g_j)\d x
 \longrightarrow
 \int_G a\nabla u\cdot\nabla\psi\d x.
\]
This proves the interior testing assertion.

For the boundary testing assertion, let $K\Subset W$ be a
compact set outside which $\psi$ vanishes almost everywhere.
Choose $\chi\in C_0^\infty(W)$ with $0\leq\chi\leq1$ and
$\chi=1$ near $K$, and choose an open set $W'\Subset W$
containing $\supp\chi$. By Lemma~\ref{lem:strongtruncation},
choose uniformly bounded $\psi_j\in C_0^\infty(\Om)$ converging
to $\psi$ in the form norm and almost everywhere. Then
$\chi\psi_j\in C_0^\infty(\Om\cap W)$, and the first part
of the proof gives convergence to $\chi\psi=\psi$ in the
form norm. Moreover,
\[
 \left|
 \int_{\Om\cap W}
 a\nabla u\cdot\nabla(\chi\psi_j-\psi)\d x
 \right|
 \leq
 \left(\int_{\Om\cap W'}a\nabla u\cdot\nabla u\d x\right)^{1/2}
 \cE(\chi\psi_j-\psi,\chi\psi_j-\psi)^{1/2}
 \longrightarrow0.
\]
The first factor is finite by the assumed energy bound up
to the boundary. Since $\chi\psi_j\in C_0^\infty(\Om\cap W)$,
the weak equation gives
\[
 \int_{\Om\cap W}a\nabla u\cdot\nabla(\chi\psi_j)\d x=0.
\]
Passing to the limit proves
\[
 \int_{\Om\cap W}a\nabla u\cdot\nabla\psi\d x=0,
\]
as asserted.

For subsolutions and supersolutions, let $\psi\geq0$ and
choose nonnegative approximants. In the interior argument,
replace the truncation defining $g$ by
$g=\min\{M,\max\{f,0\}\}$ and choose $g_j\geq0$.
In the boundary argument, choose $\psi_j\geq0$.
The cutoffs are nonnegative, so the same limiting arguments
give the corresponding weak inequalities.
No boundary regularity of $\Om$ is used.
\end{proof}

The following definition specifies the local energy interpretation of zero
boundary values.

\begin{defn}\label{def:zerotrace}
Let $\Om,W\subset\R^n$ be open and let $u$ be a function on
$\Om\cap W$. For $\eta\in C_0^\infty(W)$, define
\[
 \widetilde{\eta u}(x)=
 \begin{cases}
  \eta(x)u(x), & x\in\Om\cap W,\\
  0, & x\in\R^n\setminus(\Om\cap W).
 \end{cases}
\]
We say that $u$ has zero local energy boundary values in $W$
if $\widetilde{\eta u}\in\cF_0(\Om)$ for every
$\eta\in C_0^\infty(W)$. This convention specifies vanishing
on $\partial\Om\cap W$. Since the cutoffs are compactly
supported in $W$, it imposes no condition on the artificial
boundary $\Om\cap\partial W$.
\end{defn}

The following lemma records the truncation properties used for signed
solutions in the boundary comparison argument. For a function $h$, we write
\[
 h^+=\max\{h,0\},
 \qquad h^-=\max\{-h,0\},
\]
so that $h=h^+-h^-$ and $|h|=h^++h^-$.

\begin{lem}\label{lem:truncations}
Let $G\subset\R^n$ be open, and let $h\in\cF_{\loc}(G)$
be a weak solution of $\La h=0$ in $G$.
Then $h^+$ and $h^-$ are
weak subsolutions in $G$. Suppose additionally that $G=\Om\cap W$, where
$\Om,W\subset\R^n$ are open. If $h$ has zero local energy
boundary values in $W$ in the sense of
Definition~\ref{def:zerotrace}, then so do $h^+$ and $h^-$.
Finite linear combinations of functions with zero local
energy boundary values in $W$ have the same boundary condition.
\end{lem}

\begin{proof}
The normal contraction property gives membership of the parts in the
local energy class. For $0\leq\phi\in C_0^\infty(G)$, test the equation
with $\phi S_\epsilon(h)$, where
$S_\epsilon(t)=\min\{t^+/\epsilon,1\}$. The test is bounded and has
finite energy by the chain and product rules. The derivative term has
nonnegative sign, and passage to the limit gives
\[
 \int_G a\nabla h^+\cdot\nabla\phi\d x\leq0.
\]
Here the gradient of $h$ is zero almost everywhere on $\{h=0\}$,
and weighted Cauchy-Schwarz justifies the limit. Apply the same argument
to $-h$ for the negative part. If $\zeta\geq0$ is a smooth cutoff,
then $\zeta h^+=(\zeta h)^+$, so the boundary condition is preserved
by contraction in $\cF_0(\Om)$. Differences of nonnegative smooth
cutoffs handle arbitrary cutoffs. Linearity proves the last assertion.
\end{proof}

For a bounded open set $D$, we introduce the energy completion
used to construct variational solutions without an $\mathrm L^2$
coercivity assumption. We define $\Vzero(D)$ as the completion of $C_0^\infty(D)$
with respect to the norm
\begin{equation}\label{eq:Vzero}
 \|\varphi\|_{\Vzero(D)}
 :=\cE_D(\varphi,\varphi)^{1/2}.
\end{equation}
Positive definiteness of $a$ and the compact support of
$\varphi$ ensure that this is a norm on $C_0^\infty(D)$. For $\varphi\in C_0^\infty(D)$, zero extension and the ordinary
$W^{1,1}$ Sobolev inequality give
\begin{equation}\label{eq:energycoerciveLone}
 \|\varphi\|_{\mathrm L^{n/(n-1)}(D)}
 \leq C_n\|\nabla\varphi\|_{\mathrm L^1(D)}
 \leq C_n\left(\int_D\tr a^{-1}\d x\right)^{1/2}
          \cE_D(\varphi,\varphi)^{1/2}.
\end{equation}
Thus every energy Cauchy sequence converges in
$\mathrm L^{n/(n-1)}(D)$, and hence in $\mathrm L^1(D)$.
Together with the gradient identification argument in
Lemma~\ref{lem:closability}, this identifies the completion with
functions and their distributional gradients. Consequently,
$\Vzero(D)$ is a Hilbert space with inner product $\cE_D$.
On bounded connected Lipschitz domains it agrees with $\Haz(D)$, with
equivalent norms. We will apply Riesz representation in $\Vzero(D)$ to construct
the zero boundary correction to a prescribed Lipschitz function.
This construction requires no $\mathrm L^2$ coercivity inequality.
Once the correction is shown to be bounded, its membership in
$\cF_0(D)$ follows from the following lemma.

\begin{lem}\label{lem:boundedcompletion}
Let $D$ be bounded and open. Every bounded element of $\Vzero(D)$
belongs to $\cF_0(D)$. If $U$ is a bounded connected Lipschitz
domain, $h\in\Ha(U)$ is bounded, and $\eta\in C_0^\infty(U)$,
then the zero extension of $\eta h$ belongs to $\cF$.
If $G$ is open, $u\in\cF_{\loc}(G)$, and $U\Subset G$ is a
bounded connected Lipschitz domain, then $u|_U\in\Ha(U)$.
If $u$ is a weak solution, subsolution, or supersolution in $G$,
its restriction belongs to the corresponding $\Ha$ class in $U$.
\end{lem}

\begin{proof}
An energy core sequence for $f\in\Vzero(D)$ converges in
$\mathrm L^{n/(n-1)}(D)$ by \eqref{eq:energycoerciveLone}, hence
in measure. If $f$ is bounded, truncate the sequence above its
essential bound. Lemma~\ref{lem:strongtruncation} and its scalar
smoothing argument give uniformly bounded smooth core approximants
converging in energy and in $\mathrm L^1$. Their uniform bound
upgrades $\mathrm L^1$ convergence to $\mathrm L^2$ convergence,
which proves membership in $\cF_0(D)$.

For bounded $h\in\Ha(U)$, the $W^{1,1}$ Poincar\'e inequality and
weighted Cauchy-Schwarz imply that an approximating sequence converges
in $\mathrm L^1(U)$ as well as in energy. Truncate and smooth it
as above, pass to an almost everywhere convergent subsequence,
and multiply by $\eta$. Dominated convergence controls the
cutoff-gradient term, proving the second assertion.

For the last assertion, choose an ambient form function agreeing
with $u$ on a neighborhood of $\overline U$. Restrict its smooth
form approximants to $U$. They converge in energy and in the mean,
since $\mathrm L^2(U)$ convergence implies convergence of the mean.
Thus they converge in \eqref{eq:Hanorm}. The solution,
subsolution, or supersolution property in the $\Ha$ class follows
by testing the original equation or inequality on $C_0^\infty(U)$.
\end{proof}

\begin{rem}\label{rem:energyspaces}
All the spaces introduced above use the same weighted energy, but differ in
their integrability requirements and boundary conditions.
On a bounded connected Lipschitz domain $U$, the space $\Ha(U)$
controls the mean and the energy, imposes no boundary condition,
and contains no separate $\mathrm L^2$ requirement. Its zero boundary
subspace $\Haz(U)$ agrees with $\Vzero(U)$, with equivalent norms.
The space $\Vzero(D)$ is defined on any bounded open set $D$ using
energy alone and approximation by $C_0^\infty(D)$. The ambient space
$\cF$ and its zero boundary subspace $\cF_0(D)$ also require
$\mathrm L^2$ integrability. For bounded open $D$, identifying
functions with their zero extensions, we have
\[
 \cF_0(D)=\Vzero(D)\cap\mathrm L^2(D).
\]
The inclusion from left to right follows from the defining core
approximations. Conversely, let $f\in\Vzero(D)\cap\mathrm L^2(D)$
and set $T_M(t)=\max\{-M,\min\{t,M\}\}$. Applying
Lemma~\ref{lem:strongtruncation} and scalar smoothing to an energy
core sequence gives $T_M(f)\in\Vzero(D)$. These bounded truncations
belong to $\cF_0(D)$ by Lemma~\ref{lem:boundedcompletion}.
As $M\to\infty$, they converge to $f$ in $\mathrm L^2(D)$ and
energy by dominated convergence. Their zero extensions are
therefore Cauchy in the form norm, and closedness gives
$f\in\cF_0(D)$.

Thus a bounded variational solution constructed in
$g+\Vzero(D)$, with a bounded Lipschitz extension $g$ of the
boundary data, satisfies $u-g\in\cF_0(D)$.
Membership in $\cF_{\loc}(G)$ means local agreement with
ambient form functions, as defined in the introduction.
In particular, it implies local $\mathrm L^2$ integrability
and imposes neither global integrability nor a boundary condition. Lemma~\ref{lem:boundedcompletion} gives
\[
 u\in\cF_{\loc}(G)
 \ \Longrightarrow\
 u|_U\in\Ha(U)
 \qquad\text{for every bounded connected Lipschitz }U\Subset G.
\]
Consequently, the interior estimates of \cite{AAKT} apply to our
weak solutions by restriction to relatively compact cubes.
Conversely, under \eqref{eq:UC}, applying local boundedness to an
$\Ha$-solution and its negative makes the solution bounded on
relatively compact subdomains. The bounded localization assertion
of Lemma~\ref{lem:boundedcompletion} then places it in
$\cF_{\loc}$. The two local notions of solution therefore agree
under \eqref{eq:UC}. Zero local energy boundary values are an additional requirement:
the zero extension of $\eta u$ must belong to $\cF_0(\Om)$
for every smooth cutoff supported in the prescribed ambient
boundary neighborhood. This is the role of $\cF_0(\Om)$ in the
boundary arguments.

For every open $D$, our space also satisfies
\[
 \cF_0(D)=\overline{\cF\cap C_c(D)}^{\|\cdot\|_{\cF}},
\]
where $C_c(D)$ denotes the continuous functions compactly supported
in $D$, extended by zero. Thus it agrees with the usual part-form
domain. Indeed, for $f\in\cF\cap C_c(D)$, choose
$\chi\in C_0^\infty(D)$ equal to one near $\supp f$.
Lemmas~\ref{lem:strongtruncation} and \ref{lem:products} give
uniformly bounded smooth core approximants $f_j$ such that
$\chi f_j\to f$ in the form norm. Since
$\chi f_j\in C_0^\infty(D)$, the two closures agree.

Lemma~\ref{lem:products} identifies our weak solutions with
harmonic functions for the Dirichlet form. If $u$ has zero local
energy boundary values in $W$, then for every open $V\Subset W$
we may choose $\chi\in C_0^\infty(W)$ equal to one near
$\overline V$. The zero extension of $\chi u$ belongs to
$\cF_0(\Om)$ and agrees with $u$ on $\Om\cap V$.
Our boundary condition therefore supplies the local zero-boundary
representatives required in \cite{ChenBHP} on smaller boundary
neighborhoods.
\end{rem}

The following lemma places Lipschitz functions that vanish on the boundary in
the corresponding zero boundary energy spaces.

\begin{lem}\label{lem:Lipzero}
Let $D\subset\mathbb R^n$ be bounded and open, and let $g$ be
Lipschitz on a neighborhood of $\overline D$, with $g=0$ on
$\partial D$. Then $g|_D\in\Vzero(D)$ and its zero extension
belongs to $\cF_0(D)$.
\end{lem}

\begin{proof}
Write $L=\Lip(g)$ and $\delta_D(x)=\dist(x,\partial D)$ for
$x\in D$. Since $g=0$ on $\partial D$,
$|g(x)|\leq L\delta_D(x)$ for $x\in D$.
This estimate also shows that the zero extension $\widetilde g$
is globally Lipschitz. Its gradient agrees with $\nabla g$ in $D$
and vanishes almost everywhere on $D^c$.
Choose Lipschitz cutoffs $\zeta_\varepsilon$ in $D$ satisfying
\[
 0\leq\zeta_\varepsilon\leq1,\qquad
 \zeta_\varepsilon=0\ \text{if }\delta_D\leq\varepsilon,\qquad
 \zeta_\varepsilon=1\ \text{if }\delta_D\geq2\varepsilon,\qquad
 |\nabla\zeta_\varepsilon|\leq C/\varepsilon.
\]
Let $g_\varepsilon$ be the zero extension of
$\zeta_\varepsilon g$. In $D$,
\[
 \nabla(g-\zeta_\varepsilon g)
 =(1-\zeta_\varepsilon)\nabla g-g\nabla\zeta_\varepsilon.
\]
Both terms vanish outside the interior boundary layer
$A_\varepsilon=\{x\in D:\delta_D(x)<2\varepsilon\}$
and are bounded in magnitude by $CL$. For the second term, this
follows from $|g|\leq2L\varepsilon$ on $A_\varepsilon$.
Consequently,
\[
 \int_{\R^n}
 a\nabla(\widetilde g-g_\varepsilon)
 \cdot\nabla(\widetilde g-g_\varepsilon)\d x
 \leq CL^2\int_{A_\varepsilon}\tr a\d x
 \longrightarrow0.
\]
Indeed, $A_\varepsilon$ decreases to the empty set as
$\varepsilon\downarrow0$, and $\tr a\in\mathrm L^1(D)$.
Moreover,
\[
 \|\widetilde g-g_\varepsilon\|_{\mathrm L^2(\R^n)}^2
 \leq4L^2\varepsilon^2|D|
 \longrightarrow0.
\]

For fixed $\varepsilon$, the function $g_\varepsilon$ is globally
Lipschitz, with gradient bounded by $CL$, and its support is a
compact subset of $D$ at distance at least $\varepsilon$ from
$\partial D$. Let $\rho_\eta$ be a standard mollifier and set
\[
 g_{\varepsilon,\eta}=\rho_\eta*g_\varepsilon,
 \qquad 0<\eta<\varepsilon/2.
\]
Then $g_{\varepsilon,\eta}\in C_0^\infty(D)$ and
$g_{\varepsilon,\eta}\to g_\varepsilon$ in $\mathrm L^2$ as
$\eta\downarrow0$. The gradients converge almost everywhere
and remain bounded by $CL$. Dominated convergence therefore gives
\[
 \int_D
 a\nabla(g_{\varepsilon,\eta}-g_\varepsilon)
 \cdot\nabla(g_{\varepsilon,\eta}-g_\varepsilon)\d x
 \longrightarrow0,
\]
since the integrand is bounded by $CL^2\tr a$.
A diagonal choice of $\varepsilon\downarrow0$ and
$\eta\downarrow0$ thus gives smooth compactly supported
approximants converging to $\widetilde g$ in both
$\mathrm L^2$ and energy. Hence $g|_D\in\Vzero(D)$ and
$\widetilde g\in\cF_0(D)$.
\end{proof}

The following proposition constructs variational solutions for Lipschitz boundary
data and establishes their comparison properties.

\begin{prop}
\label{prop:variational}
Let $D\subset\mathbb R^n$ be bounded and open, and let $F$ be bounded and Lipschitz
on a neighborhood of $\overline D$. Then there is a unique $u\in F+\Vzero(D)$
such that
\begin{equation}\label{eq:variational}
 \cE_D(u,\varphi)=0,\qquad \varphi\in\Vzero(D).
\end{equation}
$u$ minimizes the energy in this affine class. Its value depends only on
$F|_{\partial D}$, and
\begin{equation}\label{eq:maximumvariational}
 \min_{\partial D}F\leq u\leq\max_{\partial D}F
 \quad\hbox{almost everywhere in }D.
\end{equation}
For two such solutions with boundary data $f$ and $g$,
\begin{equation}\label{eq:supcontract}
 \|u_f-u_g\|_{\mathrm L^\infty(D)}\leq\|f-g\|_{C(\partial D)}.
\end{equation}
Moreover $u-F\in\cF_0(D)$ and $u\in\cF_{\loc}(D)$.
\end{prop}

\begin{proof}
The functional $\varphi\mapsto-\cE_D(F,\varphi)$ is bounded on the Hilbert
space $\Vzero(D)$. Riesz representation gives a unique $z$ in that space
with $\cE_D(z,\varphi)=-\cE_D(F,\varphi)$. Set $u=F+z$.
Orthogonality gives, for every $h\in\Vzero(D)$,
\begin{equation}\label{eq:energyorthogonal}
 \cE_D(u+h,u+h)=\cE_D(u,u)+\cE_D(h,h).
\end{equation}
Since $\cE_D^{1/2}$ is a norm on $\Vzero(D)$, this proves
that $u$ is the unique energy minimizer in $F+\Vzero(D)$.
If two Lipschitz extensions have the same boundary values,
their difference belongs to $\Vzero(D)$ by
Lemma~\ref{lem:Lipzero}. They therefore define the same affine
class and the same minimizer. Thus the variational solution
depends only on $F|_{\partial D}$.

Set $m=\min_{\partial D}F$ and $M_0=\max_{\partial D}F$, and define
\[
 \widehat F=\max\{m,\min\{F,M_0\}\}.
\]
Then $\widehat F$ is Lipschitz, takes values in $[m,M_0]$,
and agrees with $F$ on $\partial D$.
By Lemma~\ref{lem:Lipzero},
$(\widehat F-F)|_D\in\Vzero(D)$, so
\[
 \widehat F+\Vzero(D)=F+\Vzero(D).
\]
We may therefore replace $F$ by $\widehat F$ and assume
$m\leq F\leq M_0$ throughout $D$. Approximate $z$ in energy by compactly supported
smooth $z_j$. The functions $(F+z_j-M_0)_+$ have compact support in $D$ and
belong to $\Vzero(D)$ by Lemma~\ref{lem:Lipzero}.
Lemma~\ref{lem:strongtruncation} shows that they
converge in energy to $(u-M_0)_+$. The gradient on the level set $\{u=M_0\}$
is zero almost everywhere, as follows from the $W^{1,1}$ chain rule.
Thus $(u-M_0)_+\in\Vzero(D)$. Testing \eqref{eq:variational} with this function
gives $\cE_D((u-M_0)_+,(u-M_0)_+)=0$, and
\eqref{eq:energycoerciveLone} gives $(u-M_0)_+=0$. The lower bound is identical.

The difference $u_f-u_g$ is the variational solution with boundary datum
$f-g$. Choose a Lipschitz extension of this datum clamped to its boundary
range, and apply \eqref{eq:maximumvariational} to obtain
\eqref{eq:supcontract}. Finally $z=u-F$ is bounded. Lemma~\ref{lem:boundedcompletion}
gives $z\in\cF_0(D)$, and localization of $F$ gives the remaining assertion.
\end{proof}

The following lemma gives the Caccioppoli estimate and proves stability of weak
solutions under locally uniform convergence.

\begin{lem}\label{lem:Caccioppoli}
Let $D\subset\R^n$ be open, let $h$ be a bounded weak solution
of $\La h=0$ in $D$, and let $\eta\in C_0^\infty(D)$. Then
\begin{equation}\label{eq:Caccioppoli}
 \int_D\eta^2a\nabla h\cdot\nabla h\d x
 \leq4\int_Dh^2a\nabla\eta\cdot\nabla\eta\d x.
\end{equation}
Furthermore, if $(h_j)$ is a locally uniformly bounded sequence
of weak solutions in $D$ that converges uniformly on compact
subsets to $h$, then $h\in\cF_{\loc}(D)$ and $h$ is a weak solution.
\end{lem}

\begin{proof}
By Lemma~\ref{lem:products}, we may test with $\eta^2h$.
Cauchy-Schwarz in the $a$ metric gives
$I\leq2I^{1/2}J^{1/2}$, with $I$ and $J$ the two integrals in
\eqref{eq:Caccioppoli}. This proves the estimate.
For the convergence assertion, fix an arbitrary open $V\Subset D$
and choose $\eta\in C_0^\infty(D)$ equal to one on a neighborhood
of $\overline V$. Apply the estimate to $h_j-h_k$ on a relatively
compact neighborhood of $\supp\eta$, where the sequence is
uniformly bounded. The resulting bound is
\[
 \int\eta^2a\nabla(h_j-h_k)\cdot\nabla(h_j-h_k)\d x
 \leq4\|h_j-h_k\|_{\mathrm L^\infty(\supp\eta)}^2
        \int a\nabla\eta\cdot\nabla\eta\d x.
\]
The product estimate then makes $\eta h_j$ Cauchy in the form norm.
Its limit belongs to $\cF$ and agrees with $h$ on $V$.
Since $V\Subset D$ was arbitrary, $h\in\cF_{\loc}(D)$.
The gradients converge in energy on $V$, so Cauchy-Schwarz
justifies passage to the weak formulation for tests supported there.
\end{proof}

\subsection{Interior oscillation and H\"older continuity}\label{sec:interior}

We now use the uniformity in \eqref{eq:UC}. Let
$u\in\cF_{\loc}(D)$ be a weak solution of $\La u=0$
in an open set $D$. For every cube $Q\Subset D$,
Lemma~\ref{lem:boundedcompletion} gives $u|_Q\in\Ha(Q)$
as an $\Ha$-solution. The local boundedness
assertion of Proposition~\ref{prop:AAKT} therefore applies to
$u$ and $-u$, showing that $u\in\mathrm L^\infty_{\loc}(D)$.

The following lemma converts the uniform Harnack inequality induced by \eqref{eq:UC} and Proposition~\ref{prop:AAKT} into oscillation decay
and an interior H\"older estimate.

\begin{lem}\label{lem:interiorosc}
Let $u$ be a bounded weak solution on a cube $Q$. Under \eqref{eq:UC},
\begin{equation}\label{eq:interiorcontraction}
 \essosc_{(1/2)Q}u\leq\kappa_I\essosc_Qu,
 \qquad \kappa_I={(H-1)}/{(H+1)}<1.
\end{equation}
Consequently, with $\alpha_I=-\log\kappa_I/\log2\in(0,1)$,
\begin{equation}\label{eq:interiorpower}
 \essosc_{\rho Q}u\leq C\rho^{\alpha_I}\essosc_Qu,
 \qquad 0<\rho\leq1.
\end{equation}
\end{lem}

\begin{proof}
Write $M_0=\esssup_Qu$, $m=\essinf_Qu$, and write $M_0'$ and $m'$ for the
corresponding quantities on $(1/2)Q$. Both $u-m$ and $M_0-u$ are nonnegative
solutions. Harnack's inequality gives
\[
 M_0'-m\leq H(m'-m),\qquad M_0-m'\leq H(M_0-M_0').
\]
Adding gives $(H+1)(M_0'-m')\leq(H-1)(M_0-m)$. Iteration proves the estimate
for $\rho=2^{-j}$. Monotonicity of oscillation gives all intermediate
scales with a factor at most $2^{\alpha_I}$.
\end{proof}

The nested essential ranges in \eqref{eq:interiorpower} determine a unique
continuous representative. Applying the estimate to cubes centered at
nearby points gives the usual local H\"older estimate. We use these
representatives throughout the rest of the paper.

The following lemma gives Harnack comparison on balls and the strong maximum
principle.

\begin{lem}
\label{lem:ballHarnack}
Under \eqref{eq:UC}, there is $H_b=H_b(\mathbf d)\geq4$ such that
every nonnegative solution of $\La u=0$ in
$B(x,R)$ satisfies
\begin{equation}\label{eq:ballHarnack}
 \sup_{B(x,R/2)}u\leq H_b\inf_{B(x,R/2)}u.
\end{equation}
A nonnegative solution on a connected open set is either identically zero
or strictly positive.
\end{lem}

\begin{proof}
Choose a fixed finite family of cubes of side $c_nR$ whose
half cubes cover $B(x,R/2)$, whose doubles lie in $B(x,R)$,
and whose half cubes can be joined by overlapping chains.
The number of cubes and the maximum chain length depend only on $n$. Repeated use of \eqref{eq:EHIcube} gives \eqref{eq:ballHarnack} with
$H_b\leq H^{N_n}$ after enlarging $N_n$. If $u$ vanishes at an interior
point, \eqref{eq:ballHarnack} makes it vanish in a neighborhood of that
point. Connectedness and overlapping balls propagate this property.
\end{proof}

\begin{rem} The proof gives an explicit dependence of the interior oscillation exponent
on the Harnack constant. In particular, as $H$ becomes large,
\begin{equation}\label{eq:exponentdependence}
 \alpha_I=\frac1{\log2}\log({(H+1)}/{(H-1)})
 \asymp H^{-1}.
\end{equation}
Under \eqref{eq:Hconstants}, this is an exponential dependence on
$\sqrt K$. No polynomial lower bound for this exponent is inferred.
\end{rem}

\subsection{Nonuniform coarse control and interior continuity}
\label{sec:goodscales}

The preceding argument uses a uniform bound on the coarse ellipticity
ratio. We now record the continuity criteria obtained when this ratio
varies along nested scales. These conclusions concern interior
regularity. The boundary comparison arguments below still require
uniform Harnack and weak Harnack constants.

Fix $x$ and $R>0$, and write $Q_j=Q_{2^{-j}R}(x)$.
Suppose that $\Theta_a(Q_j)<\infty$ for every $j\geq0$
and that $u$ is a bounded $\Ha$-solution on $Q_0$. Increase the structural constant $C_0$ so that
$H_j=\exp(C_0\sqrt{\Theta_a(Q_j)})\geq4$. The proof of
Lemma~\ref{lem:interiorosc} gives
\begin{equation}\label{eq:nonuniformiteration}
 \essosc_{Q_N}u\leq
 \prod_{j=0}^{N-1}\frac{H_j-1}{H_j+1}\,\essosc_{Q_0}u
 \leq\exp\left(-2\sum_{j=0}^{N-1}H_j^{-1}\right)\essosc_{Q_0}u.
\end{equation}
The last inequality follows from
$\log((1-z)/(1+z))\leq-2z$ for $0<z<1$.

The following proposition gives continuity and H\"older criteria in terms of the
coarse ellipticity ratios along nested scales.

\begin{prop}\label{prop:goodscales}
Under the assumptions of the preceding paragraph, if
\begin{equation}\label{eq:divergencecriterion}
 \sum_{j=0}^\infty\exp(-C_0\sqrt{\Theta_a(Q_j)})=\infty,
\end{equation}
then the essential oscillation of $u$ at $x$ tends to zero. If this holds
at every point of an open set, $u$ has a continuous representative there.
For a quantitative version, let $G$ be open, let $u$ be bounded
and $\La$-harmonic in $G$, and fix $E\Subset G$ and a common radius $R>0$
such that $Q_R(x)\Subset G$ for every $x\in E$. Assume that the
ratios satisfy $\Theta_a(Q_{2^{-j}R}(x))<\infty$ for every
$x\in E$ and $j\geq0$. If, for every $x\in E$ and integer $N\geq1$,
\begin{equation}\label{eq:gooddensity}
 \#\{0\leq j<N:\Theta_a(Q_{2^{-j}R}(x))\leq K_*\}
 \geq\delta_*N-A_*
\end{equation}
with fixed $1\leq K_*<\infty$, $0<\delta_*\leq1$ and $A_*\geq0$, then
\begin{equation}\label{eq:goodholder}
 |u(x)-u(y)|\leq C\|u\|_{\mathrm L^\infty(G)}
       \left(\frac{|x-y|}{R}\right)^{\alpha_*},
 \qquad x,y\in E,
\end{equation}
where the values at points of $E$ are defined by the limiting essential
ranges. The exponent is
\begin{equation}\label{eq:goodexponent}
 \alpha_*=-\frac{\delta_*}{\log2}
  \log\left(\frac{H_*-1}{H_*+1}\right),
 \qquad H_*=\exp(C_0\sqrt{K_*}).
\end{equation}
Here $C$ depends only on $n,p,q,s,t,K_*,\delta_*,A_*$.
For the first assertion, the finiteness and $\La$-harmonicity assumptions
are understood at each center where the criterion is applied.
\end{prop}

\begin{proof}
Condition \eqref{eq:divergencecriterion} makes the last factor in
\eqref{eq:nonuniformiteration} tend to zero. Nested essential ranges then
define the value at $x$. If the condition holds at every point, a sufficiently
small cube about a nearby point lies in a fixed small cube about $x$.
Its limiting value belongs to that cube's essential range. This proves
continuity of the resulting representative. At every scale counted in \eqref{eq:gooddensity}, the contraction factor
is at most $\kappa_*=(H_*-1)/(H_*+1)$. At all other scales it is at most
one. Therefore
\[
 \essosc_{Q_N}u
 \leq\kappa_*^{\delta_*N-A_*}\essosc_{Q_0}u
 =\kappa_*^{-A_*}2^{-N\alpha_*}\essosc_{Q_0}u.
\]
The density assumption implies the divergence criterion at every
$x\in E$. If $0<|x-y|<R/4$, choose $N$ with
$2|x-y|<2^{-N}R\leq4|x-y|$. The point $y$ lies strictly inside
$Q_{2^{-N}R}(x)$, so its representative, defined by shrinking
cubes about $y$, belongs to that cube's closed essential range.
The oscillation estimate therefore proves \eqref{eq:goodholder}.
For $|x-y|\geq R/4$, the boundedness of $u$ gives the same estimate
after increasing $C$.
\end{proof}

\section{Uniform coarse ellipticity and \texorpdfstring{$A_2$}{A2} weighted ellipticity}\label{sec:coefficientclasses}

We now prove Theorem~\ref{thm:coefficientclasses} and explain how the
coefficient assumptions differ from classical $A_2$ weighted ellipticity.
We first record the pointwise consequence of \eqref{eq:UC} and the
stronger consequence of uniform moment products. We then construct a
scalar coefficient outside $A_2$, prove the variational estimate on
every simplex, and sum the coarse quantities uniformly over all cubes.
The last example proves the reverse non-inclusion. Throughout this
section, $a_T$ and $a_{*,T}$ denote precisely the variational matrices
defined in Subsection~\ref{sec:coarse} in the approximable energy class.

\subsection{Pointwise anisotropy and scalar variation}

Although \eqref{eq:UC} is formulated through effective energies, its
validity down to arbitrarily small scales controls the ratio of the
largest and smallest eigenvalues at almost every point. The following
proposition makes this consequence precise.

\begin{prop}\label{prop:pointwiseanisotropy}
Assume \eqref{eq:basic} and \eqref{eq:UC}. Then
\begin{equation}\label{eq:pointwiseanisotropy}
 |a(x)||a(x)^{-1}|\leq K
 \quad\hbox{for almost every }x\in\R^n.
\end{equation}
Consequently, with $w(x)=|a(x)|$,
\begin{equation}\label{eq:pointwisescalarcomparison}
 K^{-1}w(x)|\xi|^2\leq a(x)\xi\cdot\xi
 \leq w(x)|\xi|^2,
 \qquad w,w^{-1}\in\mathrm L^1_{\loc}(\R^n).
\end{equation}
\end{prop}

\begin{proof}
For every simplex $T$, the variational definitions give
\begin{align}\label{eq:twosideduppermatrix}
 \left(\avgint_Ta^{-1}\d x\right)^{-1}
 &\leq a_T\leq\avgint_Ta\d x,\qquad
 \left(\avgint_Ta\d x\right)^{-1}
 \leq a_{*,T}^{-1}\leq\avgint_Ta^{-1}\d x,
\end{align}
where the inequalities are understood in the sense of quadratic
forms. Write
\[
 M_T=\avgint_Ta^{-1}\d x,
 \qquad N_T=\avgint_Ta\d x.
\]
Both matrices are finite and positive definite. To prove the first lower bound in \eqref{eq:twosideduppermatrix}, fix $e\in\R^n$ and an admissible
competitor $h=e\cdot x+v$, with $v\in\Haz(T)$.
Choose $v_j\in C_0^\infty(T)$ converging to $v$ in $\Ha(T)$.
Weighted Cauchy-Schwarz gives
\[
 \int_T|\nabla v_j-\nabla v|\d x
 \leq
 \left(\int_T|a^{-1}|\d x\right)^{1/2}
 \cE_T(v_j-v,v_j-v)^{1/2}
 \longrightarrow0.
\]
Since $\int_T\nabla v_j\d x=0$ by the divergence theorem, it follows that
$\avgint_T\nabla h\d x=e$. Consequently, for every
$\zeta\in\R^n$,
\begin{align*}
 (\zeta\cdot e)^2
 &=\left(\avgint_T
       (a^{-1/2}\zeta)\cdot(a^{1/2}\nabla h)\d x\right)^2\leq
 \left(\avgint_Ta\nabla h\cdot\nabla h\d x\right)
 (M_T\zeta\cdot\zeta).
\end{align*}
For $e\neq0$, choosing $\zeta=M_T^{-1}e$ and dividing by
$M_T^{-1}e\cdot e$ yields
\[
 \avgint_Ta\nabla h\cdot\nabla h\d x
 \geq M_T^{-1}e\cdot e.
\]
The case $e=0$ is immediate. Taking the infimum over admissible
$h$ proves $M_T^{-1}\leq a_T$. Taking the affine competitor
$h(x)=e\cdot x$ gives the upper bound $a_T\leq N_T$. For the second lower bound in \eqref{eq:twosideduppermatrix}, restrict the supremum in
\eqref{eq:lowermatrix} to affine functions $h(x)=\xi\cdot x$.
Then
\[
 a_{*,T}^{-1}e\cdot e
 \geq\sup_{\xi\in\R^n}
       \bigl(2e\cdot\xi-N_T\xi\cdot\xi\bigr)
 =N_T^{-1}e\cdot e,
\]
where the maximum is attained at $\xi=N_T^{-1}e$.
Finally, completing the square gives, almost everywhere,
\[
 2e\cdot\nabla h-a\nabla h\cdot\nabla h
 =
 a^{-1}e\cdot e
 -a(\nabla h-a^{-1}e)\cdot(\nabla h-a^{-1}e)
 \leq a^{-1}e\cdot e.
\]
Averaging and taking the supremum over $h\in\Ha(T)$ proves
$a_{*,T}^{-1}\leq M_T$, completing the proof of
\eqref{eq:twosideduppermatrix}.

Let $x$ be a common Lebesgue point of $a$ and $a^{-1}$, and set
$Q_\rho=x+\rho\Qunit$. For each fixed generation $k$, every simplex
$T\in\mathcal T_k(Q_\rho)$ has volume a fixed positive fraction of
$|Q_\rho|$. Thus its averages of $a$ and $a^{-1}$ tend to $a(x)$
and $a(x)^{-1}$, respectively, as $\rho\downarrow0$.
The inequalities in \eqref{eq:twosideduppermatrix} imply convergence of the corresponding
variational matrices to these same limits.

Fix $N\geq0$ and retain only the terms $0\leq k\leq N$ in
\eqref{eq:Lambdadef} and \eqref{eq:lambdadef}. Their product is bounded
above by $(\Theta_a(Q_\rho))^{1/2}\leq\sqrt K$. Passing to the limit gives
\[
 \sqrt K\geq
 (1-3^{-(N+1)s})(1-3^{-(N+1)t})
 |a(x)|^{1/2}|a(x)^{-1}|^{1/2}.
\]
Letting $N\to\infty$ proves \eqref{eq:pointwiseanisotropy}.
The eigenvalue bounds yield \eqref{eq:pointwisescalarcomparison}.
Finally, $w\leq\tr a$ and $w^{-1}\leq|a^{-1}|\leq\tr a^{-1}$
give local integrability.
\end{proof}

\begin{rem}
The scalar weight $w(x)=|a(x)|$ in
\eqref{eq:pointwisescalarcomparison} need not belong to $A_2$,
as the Cantor construction in Theorem~\ref{a2:thm:example} shows.
Thus \eqref{eq:UC} gives pointwise comparability to a scalar weight
without ensuring classical $A_2$ weighted ellipticity. By contrast,
the next proposition shows that uniform bounds on coefficient
moment products force the scalar weight to belong to $A_2$.
\end{rem}
\subsection{Uniform moment products and classical weighted ellipticity}\label{sec:momentproducts}

The uniformity of the moment bound over all cubes forces classical
weighted ellipticity. The next proposition proves this reduction for
arbitrary $p,q>1$, without any relation between these exponents and the
dimension.

\begin{prop}\label{prop:A2}
Suppose that for fixed $p,q>1$,
\begin{equation}\label{eq:uniformmoments}
 \sup_Q
 \left(\avgint_Q|a|^p\d x\right)^{1/p}
 \left(\avgint_Q|a^{-1}|^q\d x\right)^{1/q}\leq K_0<\infty,
\end{equation}
where both moments are finite on every cube. Then $w=|a|$ is an $A_2$
weight with $[w]_{A_2}\leq K_0$, and
\begin{equation}\label{eq:weightedellipticity}
 K_0^{-1}w(x)|\xi|^2\leq a(x)\xi\cdot\xi
 \leq w(x)|\xi|^2
\end{equation}
for almost every $x$ and every $\xi\in\R^n$.
\end{prop}

\begin{proof}
Lebesgue differentiation in \eqref{eq:uniformmoments} gives
$|a(x)||a(x)^{-1}|\leq K_0$ almost everywhere. The smallest eigenvalue
is $|a^{-1}|^{-1}$, so the lower bound in
\eqref{eq:weightedellipticity} follows. The upper bound is the definition
of $w$. Also $w^{-1}\leq|a^{-1}|$, so H\"older's inequality yields
the required $A_2$ bound
\[
 \left(\avgint_Qw\d x\right)\left(\avgint_Qw^{-1}\d x\right)
 \leq
 \left(\avgint_Q|a|^p\d x\right)^{1/p}
 \left(\avgint_Q|a^{-1}|^q\d x\right)^{1/q}\leq K_0.\qedhere
\]
\end{proof}

We next record the moment criterion as a direct verification of the
coarse hypothesis. Its boundary regularity conclusions recover a
classical weighted case.

\begin{cor}\label{cor:momentcriterion}
Suppose that \eqref{eq:uniformmoments} holds for some $p,q>1$ satisfying
\begin{equation}\label{eq:momentcriterion}
 1/p+1/q<2/{(n-1)}.
\end{equation}
Then there are $s,t>0$ satisfying \eqref{eq:admissible} for which
\eqref{eq:basic} and \eqref{eq:UC} hold, with $K\leq K_0$.
In particular, the conclusions of Theorems~\ref{thm:dirichletmain} and
\ref{thm:BHImain} hold, with quantitative constants depending only on
$n,p,q,K_0$ and the NTA data.
\end{cor}

\begin{proof}
Choose $s,t>0$ depending only on $n,p,q$ so that
\eqref{eq:admissible} holds. Lemma~\ref{lem:moments} gives
\eqref{eq:UC} with $K\leq K_0$. Local integrability in
\eqref{eq:basic} follows from the finite moments and positivity.
\end{proof}

\begin{rem} The additional restriction \eqref{eq:momentcriterion} belongs to this
verification of the coarse hypothesis. The reduction in
Proposition~\ref{prop:A2} applies for every $p,q>1$. Consequently, under
\eqref{eq:uniformmoments}, continuous Dirichlet solvability, boundary
Harnack, and H\"older continuity of quotients in the stated NTA domains
are already consequences of the classical weighted theory
\cite{FKS,FJK,FJKB}, without \eqref{eq:momentcriterion}.
The bounds $[w]_{A_2}\leq K_0$ and
\eqref{eq:weightedellipticity} control the structural constants in terms
of $n,K_0$ and the NTA data. The corollary therefore supplies a convenient
sufficient condition for applying the coarse estimates, and no new
boundary regularity result is claimed in this moment regime.
\end{rem}

\begin{rem}
To show that Theorems~\ref{thm:dirichletmain} and
\ref{thm:BHImain} apply beyond classical $A_2$ weighted ellipticity,
we below construct a scalar coefficient satisfying the full variational
condition \eqref{eq:UC} but not elliptic with respect to any $A_2$
weight.
\end{rem}

\subsection{An explicit coefficient outside the weighted class}\label{a2:sec:cantorexample}

The following theorem constructs an explicit coefficient outside the
weighted $A_2$ class for every admissible tuple. Its singular set has
codimension $s$, which allows the same upper coarse discount exponent
to control all the fine scales.

\begin{thm}\label{a2:thm:example}
Let $n\geq3$ and let $(p,q,s,t)$ satisfy \eqref{eq:admissible}. Set
\begin{equation}\label{a2:eq:Cantorparameters}
 \beta=s,\qquad \delta={(n-\beta)}/{n},\qquad
 b=2^{-1/\delta}=2^{-n/(n-\beta)}.
\end{equation}
Let $C_b\subset[0,1]$ be the Cantor set generated by the two maps
$x\mapsto bx$ and $x\mapsto1-b+bx$. Put $E=C_b^n$ and
$d_E(x)=\dist(x,E)$. Define
\begin{equation}\label{a2:eq:Phi}
 \begin{aligned}
 \Phi_\beta(r)&=
 \begin{cases}
 \displaystyle\frac{16r^{-\beta}}{(4+\beta\log(1/r))^2},&0<r<1,\\[4pt]
 1,&r\geq1,
 \end{cases}\\
 w(x)&=\Phi_\beta(d_E(x))\quad(x\notin E),\qquad w|_E=1.
 \end{aligned}
\end{equation}
Then $w\geq1$ almost everywhere, $w,w^{-1}\in\mathrm L^1_{\loc}(\R^n)$,
and $a=w\Id$ satisfies \eqref{eq:UC} for every admissible tuple
with this fixed $s$, with $K$ depending only on $n$ and $s$. However, the weight $w$ does not belong to
$A_2$, and $a$ is not elliptic with respect to any $A_2$ weight with
fixed comparison constants.
\end{thm}

\begin{rem}
Although $w(x)\to\infty$ as $d_E(x)\downarrow0$ with $x\notin E$,
the assignment $w|_E=1$ specifies a finite, positive representative.
Since $E$ has Lebesgue measure zero, any other finite, positive
assignment there gives the same energy form and coarse quantities.
\end{rem}

The proof of Theorem~\ref{a2:thm:example} is given in
Subsections~\ref{a2:sec:geometry}-\ref{a2:sec:proof}.
Proposition~\ref{a2:prop:converse} then establishes the opposite
non-inclusion. Together they prove Theorem~\ref{thm:coefficientclasses}.

\subsection{Geometry of the product Cantor set}\label{a2:sec:geometry}

Let $\nu$ be the probability measure on $C_b$ which gives mass $2^{-k}$
to every construction interval of generation $k$, and put
$\mu=\nu^{\otimes n}$.  By self-similarity and separation, $\nu$ is the normalized
$\delta$-dimensional Hausdorff measure on $C_b$, and $\mu$
is the normalized $d$-dimensional Hausdorff measure on
$E=C_b^n$, where $d=n-\beta=n\delta$. Admissibility implies $0<\beta=s<1$, so $0<\delta<1$ and $b<1/2$.
Unless otherwise indicated, all constants in the construction depend
only on $n,\beta$. The construction cubes have side $b^k$ and mass
$2^{-nk}=(b^k)^d$. Different children of a construction cube of side
$L$ are separated by at least $(1-2b)L$ in at least one coordinate.

The following lemma gives both the dimension estimate and a quantitative
bound near every affine hyperplane.

\begin{lem}\label{a2:lem:geometry}
There are constants $c,C>0$ depending only on $n,\beta$ such that
\begin{equation}\label{a2:eq:ahlfors}
 cR^d\leq\mu(B(z,R))\leq CR^d,
 \qquad z\in E,\quad 0<R\leq1.
\end{equation}
For every $(n-1)$-dimensional affine hyperplane $H\subset\R^n$,
every $z\in\R^n$, and $0<r\leq R\leq1$,
\begin{equation}\label{a2:eq:hyperplane}
 \mu\bigl(B(z,R)\cap\{\dist(\cdot,H)<r\}\bigr)
 \leq C R^{(n-1)\delta}r^\delta.
\end{equation}
The upper estimate in \eqref{a2:eq:ahlfors} holds for arbitrary centers.
\end{lem}

\begin{proof}
At the generation for which $b^k\leq R<b^{k-1}$, an interval of length
$2R$ meets at most a bounded number of construction intervals, since
their separation is comparable to $b^k$. Its $\nu$ measure is therefore
at most $CR^\delta$. If its center is in $C_b$, a construction interval
containing that center and having length comparable to a sufficiently
small multiple of $R$ lies inside it. This gives the corresponding
lower estimate. Taking products, and comparing balls with cubes, proves
\eqref{a2:eq:ahlfors}. To prove \eqref{a2:eq:hyperplane}, we use the estimate $\nu(I)\leq C|I|^\delta$ which holds for every interval $I\subset\R$. Write
$H=\{x:v\cdot x=c_0\}$ with $|v|=1$.
Since $\sum_{j=1}^n|v_j|^2=1$, there is an index $i$ such that
$|v_i|\geq n^{-1/2}$. Fix the remaining coordinates $x'=(x_j)_{j\neq i}$.
Because $v$ is a unit normal, the condition
$\dist(x,H)<r$ is equivalent to
\[
 \biggl|v_i x_i+\sum_{j\neq i}v_jx_j-c_0\biggr|<r.
\]
Thus the admissible values of $x_i$ form an interval $J(x')$
of length
\[
 |J(x')|={2r}/{|v_i|}\leq2\sqrt n\,r.
\]
The interval estimate therefore gives
$\nu(J(x'))\leq Cr^\delta$, uniformly in $x'$ and in the
position and orientation of $H$. Each other coordinate in $B(z,R)$ lies in an interval of length $2R$,
whose measure is at most $CR^\delta$. Fubini's theorem gives
$C R^{(n-1)\delta}r^\delta$ in \eqref{a2:eq:hyperplane}.
\end{proof}

Set $E_r=\{x:d_E(x)<r\}$. The next lemma converts the measure estimates
into the volume estimates needed for the energy calculation.

\begin{lem}\label{a2:lem:tubes}
For $0<r\leq R\leq1$ and arbitrary $z$,
\begin{equation}\label{a2:eq:tubeupper}
 |E_r\cap B(z,R)|\leq C R^{n-\beta}r^\beta.
\end{equation}
For $z\in E$ and $0<r<cR\leq c$,
\begin{equation}\label{a2:eq:tubelower}
 |E_r\cap B(z,R)|\geq cR^{n-\beta}r^\beta.
\end{equation}
For every fixed $A\geq1$, every $(n-1)$-dimensional affine hyperplane
$H\subset\R^n$, and $0<r\leq R\leq1$,
\begin{equation}\label{a2:eq:stripvolume}
 |E_r\cap B(z,R)\cap\{\dist(\cdot,H)<Ar\}|
 \leq C_A R^{(n-1)\delta}r^{\beta+\delta}.
\end{equation}
The global versions are $|E_r|\leq Cr^\beta$ and
$|E_r\cap\{\dist(\cdot,H)<Ar\}|\leq C_A r^{\beta+\delta}$ for $0<r\leq1$.
\end{lem}

\begin{proof}
To prove \eqref{a2:eq:tubeupper}, fix $0<r\leq R\leq1$ and
choose a maximal family of points
$z_1,\ldots,z_N\in E\cap B(z,2R)$ whose pairwise distances
are at least $r$. This family is finite because it is
$r$-separated and contained in a bounded set. By maximality,
\[
 E\cap B(z,2R)\subset\bigcup_{j=1}^N B(z_j,r),
\]
while the smaller balls $B(z_j,r/3)$ are pairwise disjoint
and contained in $B(z,3R)$. Since their centers belong to $E$,
the lower bound in \eqref{a2:eq:ahlfors} gives
\[
 cNr^d
 \leq\sum_{j=1}^N\mu(B(z_j,r/3))
 \leq\mu(B(z,3R))
 \leq CR^d.
\]
For the last inequality, if $3R\leq1$, we use the upper
bound in \eqref{a2:eq:ahlfors}, which holds for arbitrary
centers. If $3R>1$, we instead use
\[
 \mu(B(z,3R))\leq\mu(E)=1\leq3^dR^d.
\]
Thus $N\leq C(R/r)^d$. We next show that the balls $B(z_j,2r)$ cover
$E_r\cap B(z,R)$. If $x$ belongs to this set, choose
$y\in E$ with $|x-y|<r$. Then $|y-z|\leq|y-x|+|x-z|<r+R\leq2R$, so $y\in E\cap B(z,2R)$. The covering above gives an index
$j$ with $|y-z_j|<r$, and hence $|x-z_j|<2r$. Consequently,
\[
 |E_r\cap B(z,R)|
 \leq\sum_{j=1}^N|B(z_j,2r)|
 \leq CNr^n
 \leq C(R/r)^d r^n
 =CR^{n-\beta}r^\beta,
\]
where we used $d=n-\beta$. This proves
\eqref{a2:eq:tubeupper} throughout the stated range of radii.  For \eqref{a2:eq:tubelower}, choose a maximal $r$-separated
family in $E\cap B(z,R/2)$. Its $r$-balls cover this set, so
\eqref{a2:eq:ahlfors} gives at least $c(R/r)^d$ centers.
Their disjoint $r/3$-balls lie in $E_r\cap B(z,R)$ and give
the lower bound. For \eqref{a2:eq:stripvolume}, the case $r\geq R/(A+2)$
follows from \eqref{a2:eq:tubeupper}. Otherwise, choose a
maximal $r$-separated family in
\[
 E\cap B(z,R+r)\cap\{\dist(\cdot,H)<(A+1)r\}.
\]
Every point counted in \eqref{a2:eq:stripvolume} is within
distance $r$ of this set, so the $2r$-balls cover it.
The disjoint $r/3$-balls lie in
$B(z,3R)\cap\{\dist(\cdot,H)<(A+2)r\}$.
By \eqref{a2:eq:ahlfors} and \eqref{a2:eq:hyperplane},
the number of centers is at most $C_A(R/r)^{d-\delta}$.
Multiplication by $Cr^n$ gives \eqref{a2:eq:stripvolume}.
If $3R>1$, apply the hyperplane estimate using a fixed
finite covering of $E$ by unit balls, noting that
$R^{d-\delta}\asymp1$. The global bounds follow by covering
the bounded set $E_1$ with finitely many unit balls.
\end{proof}

For a simplex $T$ similar to one of the reference simplices, write $R$
for its scale, so $|T|=R^n/n!$ and $\diam(T)\leq\sqrt n R$.
Let $R_*=\min\{R,1\}$ and $\delta_T(x)=\dist(x,\partial T)$.
The boundary $\partial T$ lies in the union of $n+1$ hyperplanes. Lemma \ref{a2:lem:tubes} consequently gives, for fixed $A$,
\begin{equation}\label{a2:eq:simplexstrip}
 |\{x\in T:d_E(x)<r,\ \delta_T(x)<Ar\}|
 \leq C_A |T|\left(\frac r{R_*}\right)^{\beta+\delta},
 \qquad 0<r\leq R_*.
\end{equation}
For $C_nR\leq1$, use \eqref{a2:eq:stripvolume} in a ball of
radius $C_nR$ containing $T$. If $R\leq1<C_nR$, then $R$ is
bounded below by a dimensional constant. The global strip bound
therefore gives the same estimate, since
$|T|R^{-\beta-\delta}=R^{(n-1)\delta}/n!$ is bounded below
on this range. For $R\geq1$, use the global strip bound and
$|T|\geq1/n!$.

\subsection{Integrability and failure of the \texorpdfstring{$A_2$}{A2} condition}\label{a2:sec:weight}

Write $L_\beta(r)=4+\beta\log(1/r)$ for $0<r\leq1$.
For $0<r<1$, direct differentiation gives
\begin{equation}\label{a2:eq:phiderivative}
 \begin{aligned}
 -\Phi_\beta'(r)
 &=\frac{16\beta}{r^{\beta+1}L_\beta(r)^2}
       \left(1-\frac2{L_\beta(r)}\right),\qquad \frac{8\beta}{r^{\beta+1}L_\beta(r)^2}
 \leq-\Phi_\beta'(r)
       \leq\frac{16\beta}{r^{\beta+1}L_\beta(r)^2}.
 \end{aligned}
\end{equation}
Thus $\Phi_\beta$ is nonincreasing and at least one. Its logarithmic slope lies
between $-\beta$ and $-\beta/2$ below scale one. In particular,
\begin{equation}\label{a2:eq:phiratio}
 1\leq\frac{\Phi_\beta(r)}{\Phi_\beta(R)}\leq\left(\frac Rr\right)^\beta,
 \qquad 0<r\leq R,
\end{equation}
where the extension $\Phi_\beta=1$ above scale one is understood.
Comparing arguments by a fixed multiplicative factor $A\geq1$ therefore
changes $\Phi_\beta$ by at most $A^\beta$.

The following proposition verifies the qualitative hypotheses and the
failure of all comparable $A_2$ descriptions.

\begin{prop}\label{a2:prop:notA2}
The weight in \eqref{a2:eq:Phi} satisfies $w\geq1$ almost everywhere and
$w,w^{-1}\in\mathrm L^1_{\loc}(\R^n)$. Furthermore, there are constants $c,C>0$
such that, for $z\in E$ and $0<R<c$,
\begin{equation}\label{a2:eq:A2divergence}
 \left(\fint_{B(z,R)}w\d x\right)
 \left(\fint_{B(z,R)}w^{-1}\d x\right)
 \geq cL_\beta(R)=c\bigl(4+\beta\log(1/R)\bigr).
\end{equation}
Consequently $w\notin A_2$, and no $A_2$ weight is comparable to $w$
with fixed positive comparison constants.
\end{prop}

\begin{proof}
Since $E$ has Lebesgue measure zero, its assigned point values do not
affect the integrals. The layer-cake identity gives
\[
 \int_{\R^n}(w-1)\d x
 =\int_0^1|E_r|\,(-\Phi_\beta'(r))\d r
 \leq C\beta\int_0^1\frac{\d r}{rL_\beta(r)^2}<\infty.
\]
Also $w^{-1}\leq1$. This proves the integrability assertions and
\eqref{eq:basic} for $a=w\Id$. Applying the lower tube bound inside $B(z,R)$ gives
\begin{align*}
 \int_{B(z,R)}w\d x
 &\geq\int_0^{cR}|E_r\cap B(z,R)|\,(-\Phi_\beta'(r))\d r\\
 &\geq c\beta R^{n-\beta}\int_0^{cR}\frac{\d r}{rL_\beta(r)^2}
 =\frac{cR^{n-\beta}}{L_\beta(cR)}
 \geq\frac{cR^{n-\beta}}{L_\beta(R)}.
\end{align*}
Choose a fixed $\varepsilon>0$ so small that \eqref{a2:eq:tubeupper}
gives $|E_{\varepsilon R}\cap B(z,R)|\leq\tfrac12|B(z,R)|$.
On the remaining half of the ball $d_E(x)\geq \varepsilon R$ and
$w^{-1}\geq(\Phi_\beta(\varepsilon R))^{-1}\geq c(\Phi_\beta(R))^{-1}$ by
\eqref{a2:eq:phiratio}. Hence
\[
 \fint_{B(z,R)}w\d x\geq\frac{cR^{-\beta}}{L_\beta(R)},\qquad
 \fint_{B(z,R)}w^{-1}\d x\geq cR^\beta L_\beta(R)^2.
\]
Their product proves \eqref{a2:eq:A2divergence}. The definition of $A_2$
using balls is equivalent to that using cubes, by inscribed and
circumscribed cubes with comparable volumes. The reverse bounds for the two means also hold. Indeed,
\[
 \int_{B(z,R)}w\d x
 \leq |B(z,R)|\Phi_\beta(R)
       +C\beta R^{n-\beta}\int_0^R\frac{\d r}{rL_\beta(r)^2}
 \leq\frac{CR^{n-\beta}}{L_\beta(R)},
\]
by the upper tube bound, while $d_E(x)\leq R$ throughout the ball
gives $w^{-1}\leq\Phi_\beta(R)^{-1}$ almost everywhere. Thus the $A_2$
product on these balls is, in fact, comparable to $L_\beta(R)$. If $c_0\rho\leq w\leq C_0\rho$ and $\rho\in A_2$, then the
$A_2$ product of $w$ on every ball is at most
$C_n(C_0/c_0)[\rho]_{A_2}$. This contradicts
\eqref{a2:eq:A2divergence}. In particular, ellipticity of $w\Id$ with
respect to an $A_2$ weight is impossible.
\end{proof}

\subsection{The variational estimate on every simplex}\label{a2:sec:variational}

It remains to prove that $a=w\Id$ satisfies \eqref{eq:UC}
for the prescribed admissible tuple, uniformly over all cubes.
The main step is to estimate the minimum energy among functions
with prescribed affine boundary values on a simplex. Using the
affine function itself gives a bound in terms of the integral
of $w$, which is too large for the required coarse estimates. We therefore modify the affine function to make it constant near
the Cantor pieces lying sufficiently far inside the simplex.
Its gradient vanishes near these pieces, so the energy at very
small distances from $E$ is confined to a region near the simplex
boundary. Figure~\ref{fig:cantorcompetitor} illustrates the product
Cantor construction and the regions where the competitor is
made constant.

\begin{figure}[htbp]
 \centering
\includegraphics[width=\textwidth]{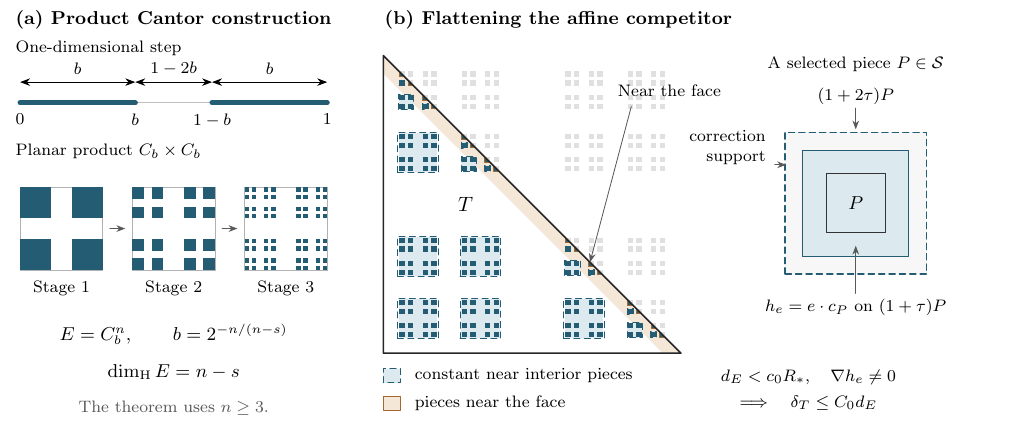}
 \caption{The Cantor construction and the affine competitor.
 Left: three stages of the planar product $C_b\times C_b$,
 illustrating $E=C_b^n$.
 Right: selected construction cubes inside a simplex $T$.
 The competitor is constant on $(1+\tau)P$, and the corresponding
 correction is supported in $(1+2\tau)P$.
 The enlargement widths in the inset are exaggerated for visibility.}
 \label{fig:cantorcompetitor}
\end{figure}

The following lemma constructs the competitor and proves that
$h_e-e\cdot x\in\Haz(T)$, making $h_e$ admissible in the
variational definition \eqref{eq:uppermatrix}. Recall that $\delta_T(x)$ denotes the distance from $x$ to $\partial T$.

\begin{samepage}
\begin{lem}\label{a2:lem:competitor}
Let $T$ be a reference-shape simplex of scale $R$, and set
$R_*=\min\{R,1\}$. For every $e\in\R^n$, there is a globally Lipschitz
function $h_e$ with $h_e=e\cdot x$ outside $T$,
$h_e-e\cdot x\in\Haz(T)$, and
\begin{equation}\label{a2:eq:gradientbound}
 |\nabla h_e|\leq C|e|\quad\hbox{almost everywhere}.
\end{equation}
Furthermore, there are $c_0,C_0>0$, depending only on $n,\beta$, such that
\begin{equation}\label{a2:eq:gradientsupport}
 \nabla h_e(x)=0
 \quad\hbox{for almost every }x\in T\hbox{ with }
 d_E(x)<c_0R_*\hbox{ and }
 \delta_T(x)>C_0d_E(x).
\end{equation}
\end{lem}
\end{samepage}

\begin{proof}
All construction cubes are closed and axis-parallel. Choose
$\tau>0$ smaller than $\min\{1/8,(1-2b)/(16b)\}$.
If two construction cubes are incomparable under inclusion, their
$(1+2\tau)$-dilates are disjoint. To see this, take their last common
ancestor of side $L$. They lie in different children, separated in
one coordinate by $(1-2b)L$, while both their side lengths are at
most $bL$. The added margins are smaller than this separation by
the choice of $\tau$.

Let $\cS$ consist of the maximal construction cubes $P$ satisfying
\begin{equation}\label{a2:eq:stopping}
 \ell(P)\leq R_*/4,\qquad (1+4\tau)P\subset T.
\end{equation}
The inclusion is in the open simplex, so each such enlarged cube is
a compact subset of $T$. Since $T$ is open, every $x\in E\cap T$ satisfies
$\delta_T(x)=\dist(x,\partial T)>0$. The construction cubes
containing $x$ have diameters tending to zero, so sufficiently
small ones satisfy \eqref{a2:eq:stopping}. Hence $x$ belongs
to a selected cube of $\cS$. The cubes in $\cS$ are incomparable.

Choose $\eta_P\in C_0^\infty(\operatorname{int}((1+2\tau)P))$ with
$0\leq\eta_P\leq1$, $\eta_P=1$ on $(1+\tau)P$, and
$|\nabla\eta_P|\leq C/\ell(P)$. If $c_P$ is the center of $P$, put
\begin{equation}\label{a2:eq:competitorformula}
 h_e(x)=e\cdot x+
       \sum_{P\in\cS}\eta_P(x)\,e\cdot(c_P-x).
\end{equation}
The family $\cS$ is at most countable but need not be finite,
since maximal cubes may occur at arbitrarily fine generations
near $\partial T$. Indeed, if points of $E\cap T$ approach
$\partial T$, then $\cS$ must be infinite, because a finite
union of selected cubes is compactly contained in $T$ and
cannot cover such points. However, for every $\varepsilon>0$, only finitely many cubes
in $\cS$ have side length at least $\varepsilon$, since only
finitely many generations contribute and each contains finitely
many construction cubes. To justify the series in \eqref{a2:eq:competitorformula} and
establish the admissibility of $h_e$, we consider finite partial
sums. Let $\cS_N$ be increasing finite subcollections exhausting
$\cS$, and define
\[
 h_e^{(N)}(x)=e\cdot x+
       \sum_{P\in\cS_N}\eta_P(x)\,e\cdot(c_P-x).
\]
Then $h_e^{(N)}-e\cdot x\in C_0^\infty(T)$. The supports of the corrections are disjoint. Each correction has
gradient bounded by $C|e|$ and amplitude bounded by $C|e|\ell(P)$.
Since only finitely many selected cubes have side length above any
fixed positive threshold, $h_e^{(N)}\to h_e$ uniformly.
Moreover, the disjointness of the supports gives a Lipschitz bound
$C|e|$ for every $h_e^{(N)}$. Passing to the uniform limit proves
\eqref{a2:eq:gradientbound}. On $(1+\tau)P$, the function $h_e$
is the constant $e\cdot c_P$.

We next verify convergence in the energy space. The omitted
corrections are supported in
\[
 U_N=\bigcup_{P\in\cS\setminus\cS_N}(1+2\tau)P\subset T,
\]
where $(1+2\tau)P$ is the cube concentric with $P$ and of side
length $(1+2\tau)\ell(P)$. The gradient bounds and disjointness
of the supports give
\[
 \cE_T(h_e-h_e^{(N)},h_e-h_e^{(N)})
 \leq C|e|^2\int_{U_N}w\d x
 =C|e|^2\sum_{P\in\cS\setminus\cS_N}
       \int_{(1+2\tau)P}w\d x
 \longrightarrow0.
\]
Indeed, the enlarged cubes are pairwise disjoint subsets of $T$,
so the full series of integrals is bounded by $\int_Tw\d x<\infty$.
Uniform convergence also gives convergence of the means.
Thus $h_e^{(N)}-e\cdot x$ converges to $h_e-e\cdot x$ in
$\Ha(T)$, proving that $h_e-e\cdot x\in\Haz(T)$.

It remains to prove \eqref{a2:eq:gradientsupport}. At a point where the
gradient is nonzero, let $r=d_E(x)>0$ and choose a nearest point
$y\in E$. If $y\notin T$, then $\delta_T(x)\leq r$. Otherwise let
$P\in\cS$ contain $y$. Since the gradient is zero on $(1+\tau)P$,
we may assume that $x$ lies outside this enlarged cube, and hence
\begin{equation}\label{a2:eq:stoppingsize}
 r\geq\tau\ell(P)/2.
\end{equation}
The cube $P$ has a parent $\widehat P$, since its side is at most
$1/4$. If the parent violates the size restriction in
\eqref{a2:eq:stopping}, then $\ell(P)>bR_*/4$, which contradicts
$r<c_0R_*$ provided $c_0<\tau b/8$.
Otherwise maximality implies that $(1+4\tau)\widehat P$ is not a
subset of $T$. Since it contains $y$, this gives
$\delta_T(y)\leq C\ell(\widehat P)=C\ell(P)/b$.
Together with \eqref{a2:eq:stoppingsize},
\[
 \delta_T(x)\leq|x-y|+\delta_T(y)\leq C_0r.
\]
This proves the asserted support property. Sets where gradients are
undefined, and $E$ itself, have Lebesgue measure zero.
\end{proof}

The next lemma turns the competitor into a uniform bound for the upper
variational matrix.

\begin{lem}\label{a2:lem:cell}
For every reference-shape simplex $T$ of scale $R$,
\begin{equation}\label{a2:eq:cell}
 |a_T|\leq C\Phi_\beta\bigl(\dist(T,E)+R\bigr).
\end{equation}
The constant is independent of the location and scale of $T$.
\end{lem}

\begin{proof} Note that the simplex $T$ may have arbitrary location and scale.
In particular, it need not intersect $E$, since the coarse
condition must be verified on every ambient cube. Fix $e\in\R^n$ and let $h_e$ be the competitor from
Lemma~\ref{a2:lem:competitor}. We first prove
$|a_T|\leq C\Phi_\beta(R_*)$, where $R_*=\min\{R,1\}$.
Decreasing $c_0$ if necessary, assume $0<c_0\leq1$. On the region where $d_E\geq c_0R_*$, monotonicity and
\eqref{a2:eq:phiratio} give
\[
 w\leq\Phi_\beta(c_0R_*)
 \leq c_0^{-\beta}\Phi_\beta(R_*).
\]
Together with $|\nabla h_e|\leq C|e|$, this bounds the energy
on that region by $C|e|^2|T|\Phi_\beta(R_*)$. To estimate the remaining energy, set
\[
 r_j=c_0R_*2^{-j},\qquad
 S_j=\{x\in T:r_{j+1}\leq d_E(x)<r_j\},
 \qquad j\geq0.
\]
These shells cover $\{x\in T:0<d_E(x)<c_0R_*\}$, and
$E$ has Lebesgue measure zero. On the part of $S_j$ where
$\nabla h_e\neq0$, \eqref{a2:eq:gradientsupport} gives
\[
 \delta_T(x)\leq C_0d_E(x)<C_0r_j.
\]
Thus \eqref{a2:eq:simplexstrip} implies
\[
 |S_j\cap\{\nabla h_e\neq0\}|
 \leq C|T|\left(\frac{r_j}{R_*}\right)^{\beta+\delta}.
\]
On $S_j$, the profile bound gives
\[
 w\leq\Phi_\beta(r_{j+1})
 \leq\left(\frac{R_*}{r_{j+1}}\right)^\beta
       \Phi_\beta(R_*).
\]
Combining these estimates with \eqref{a2:eq:gradientbound},
and using $r_{j+1}=r_j/2$, we obtain
\begin{align*}
 \int_{S_j}w|\nabla h_e|^2\d x
 &\leq C|e|^2|T|\Phi_\beta(R_*)
       \left(\frac{R_*}{r_{j+1}}\right)^\beta
       \left(\frac{r_j}{R_*}\right)^{\beta+\delta}\leq C|e|^2|T|\Phi_\beta(R_*)\,2^{-j\delta}.
\end{align*}
Since $\delta>0$, summing over the shells and adding the
contribution from $d_E\geq c_0R_*$ yields
\begin{equation}\label{a2:eq:cellenergy}
 \int_Tw|\nabla h_e|^2\d x
 \leq C|e|^2|T|\Phi_\beta(R_*)
       \biggl(1+\sum_{j\geq0}2^{-j\delta}\biggr)
 \leq C|e|^2|T|\Phi_\beta(R_*).
\end{equation}
\eqref{a2:eq:cellenergy}, the admissibility of $h_e$, and the variational definition give
\[
 a_Te\cdot e
 \leq\avgint_Tw|\nabla h_e|^2\d x
 \leq C\Phi_\beta(R_*)|e|^2.
\]
Since $e$ is arbitrary, this proves the preliminary bound. To incorporate the distance from $T$ to $E$, put
$D=\dist(T,E)$. If $D\leq R$, then $R\leq D+R\leq2R$.
Since $\Phi_\beta(R_*)=\Phi_\beta(R)$,
\eqref{a2:eq:phiratio} gives
\[
 |a_T|\leq C\Phi_\beta(R)
 \leq C2^\beta\Phi_\beta(D+R).
\]
If $D>R$, then $d_E(x)\geq D$ throughout $T$, so the
affine competitor $x\mapsto e\cdot x$ gives
$|a_T|\leq\Phi_\beta(D)$. Applying \eqref{a2:eq:phiratio}
once more,
\[
 |a_T|\leq\Phi_\beta(D)
 \leq\left(\frac{D+R}{D}\right)^\beta\Phi_\beta(D+R)
 \leq2^\beta\Phi_\beta(D+R).
\]
This proves \eqref{a2:eq:cell} in both cases.
\end{proof}

\subsection{Uniformity of the coarse estimates on all cubes}\label{a2:sec:UC}

We now combine the simplex estimates to verify \eqref{eq:UC} uniformly
over all cubes. For each cube, we choose a common scalar quantity
against which to compare the variational matrices of all its
simplices at every finer scale. These bounds control both discounted
sums, with constants independent of the spatial moment exponents.
After proving the theorem from these cellwise bounds, we record a
sharper spatial moment estimate using the tube-volume bounds.

\begin{lem}\label{a2:lem:spatialmoment}
Let $Q$ have side $R$, and put
\begin{equation}\label{a2:eq:cubescale}
 D=\dist(Q,E),\qquad F_Q=\Phi_\beta(D+R),\qquad h_k=3^{-k}R.
\end{equation}
For every $k\geq0$,
\begin{equation}\label{a2:eq:levelbounds}
 \max_{T\in\cT_k(Q)}|a_T|\leq CF_Q\,3^{k\beta},\qquad
 \max_{T\in\cT_k(Q)}|a_{*,T}^{-1}|\leq CF_Q^{-1}.
\end{equation}
In particular, for every $p>1$,
\begin{equation}\label{a2:eq:uniformspatialmoment}
 \left(\frac1{\#\cT_k(Q)}
       \sum_{T\in\cT_k(Q)}|a_T|^p\right)^{1/(2p)}
 \leq CF_Q^{1/2}\,3^{k\beta/2},
\end{equation}
where these constants depend only on $n,\beta$.
\end{lem}

\begin{proof}
For $x\in Q$, we have $d_E(x)\leq D+\sqrt nR$. Monotonicity and
\eqref{a2:eq:phiratio} give
\[
 w(x)\geq\Phi_\beta(D+\sqrt nR)
       \geq n^{-\beta/2}F_Q
 \quad\hbox{for almost every }x\in Q.
\]
Using this estimate and \eqref{eq:matrixaveragebounds}, we see that every simplex $T\subset Q$ satisfies
\begin{equation}\label{a2:eq:inversecell}
 |a_{*,T}^{-1}|\leq\fint_Tw^{-1}\d x
                    \leq n^{\beta/2}F_Q^{-1}.
\end{equation}
For $T\in\cT_k(Q)$, Lemma~\ref{a2:lem:cell} and
$\dist(T,E)\geq \dist(Q,E)=D$ yield
\[
 \begin{aligned}
 |a_T|&\leq C\Phi_\beta(D+h_k)\leq CF_Q\left(\frac{D+R}{D+h_k}\right)^\beta
       \leq CF_Q\left(\frac R{h_k}\right)^\beta
       =CF_Q\,3^{k\beta}.
 \end{aligned}
\]
This proves \eqref{a2:eq:levelbounds} for all locations and scales.
Taking the spatial moment gives \eqref{a2:eq:uniformspatialmoment}. \end{proof}

As we next prove, the bounds for individual cells suffice to verify the full coarse
condition for the prescribed tuple. In particular, the resulting coarse
bound does not depend on $p,q,t$.

\subsection{Proof of Theorem~\ref{a2:thm:example}}\label{a2:sec:proof}

Proposition~\ref{a2:prop:notA2} proves the qualitative assumptions and
the failure of $A_2$ weighted ellipticity. Fix an arbitrary cube $Q$,
and use the quantity $F_Q$ in \eqref{a2:eq:cubescale}.
By \eqref{a2:eq:levelbounds} and the normalized geometric sum in
\eqref{eq:lambdadef},
\begin{equation}\label{a2:eq:lambdafinal}
 (\lambda_{t,1,q}(Q))^{-1}\leq CF_Q^{-1}.
\end{equation}
This holds for every $q>1$ and $t>0$, with $C$ depending only on
$n,\beta$. Equation~\eqref{a2:eq:uniformspatialmoment} likewise gives
\begin{equation}\label{a2:eq:Lambdafinal}
 (\Lambda_{s,1,p}(Q))^{1/2}
 \leq CF_Q^{1/2}(1-3^{-s})
       \sum_{k\geq0}3^{-k(s-\beta/2)}.
\end{equation}
Our choice $\beta=s$ implies
\begin{equation}\label{a2:eq:scondition}
 s-\frac\beta2=\frac s2>0,\qquad
 \frac{1-3^{-s}}{1-3^{-(s-\beta/2)}}
 =1+3^{-s/2}<2.
\end{equation}
Thus $\Lambda_{s,1,p}(Q)\leq CF_Q$. Combining this with
\eqref{a2:eq:lambdafinal} proves
\[
 \Theta_a(Q)=\Lambda_{s,1,p}(Q)(\lambda_{t,1,q}(Q))^{-1}
 \leq C(n,s)
\]
uniformly over all cubes. Both defining sums are finite and strictly
positive, by the displayed bounds and the positive definiteness of
the variational matrices. This verifies every part of \eqref{eq:UC}
for the prescribed admissible tuple.\qed

\begin{rem} The distinction between a coefficient moment and an effective energy
is visible quantitatively. For balls centered on $E$, the ordinary
mean of $w$ has an additional factor comparable to $L_\beta(R)$ relative
to $\Phi_\beta(R)$. The affine-boundary competitor removes that factor
from the variational upper estimate. The hyperplane estimate
\eqref{a2:eq:hyperplane} makes this removal uniform even when the boundary
of a simplex intersects $E$. Choosing the codimension $\beta=s$ then
ensures summability for every prescribed admissible tuple.
\end{rem}

The cellwise bounds above suffice for the theorem and give a
constant independent of $p,q,t$. For completeness, the upper
tube-volume estimate \eqref{a2:eq:tubeupper} also yields the
sharper spatial moment bound below. It quantifies the gain
from the small volume occupied by cells close to $E$,
although this gain is not needed for \eqref{eq:UC}.

\begin{lem}\label{a2:lem:sharpermoment}
Use the notation of Lemma~\ref{a2:lem:spatialmoment}.
For every $p>1$ and $k\geq0$, if $R\leq1$ and $D\leq R$, the estimate
\begin{equation}\label{a2:eq:spatialmoment}
 \left(\frac1{\#\cT_k(Q)}
       \sum_{T\in\cT_k(Q)}|a_T|^p\right)^{1/(2p)}
 \leq C_p(\Phi_\beta(R))^{1/2}\,3^{k\beta(1-1/p)/2}
\end{equation}
holds. If $R\geq1$, this last estimate holds with $\Phi_\beta(R)=1$.
The constant $C_p$ may also depend on $p$.
\end{lem}

\begin{proof}
If $x\in T\in\cT_k(Q)$, the numbers
$\dist(T,E)+h_k$ and $d_E(x)+h_k$ differ by a bounded factor.
Lemma~\ref{a2:lem:cell}, \eqref{a2:eq:phiratio}, and the equality of simplex
volumes therefore give
\begin{equation}\label{a2:eq:momentintegral}
 \frac1{\#\cT_k(Q)}\sum_T|a_T|^p
 \leq C_p\fint_Q(\Phi_\beta(d_E(x)+h_k))^p\d x.
\end{equation}
Suppose first that $R\leq1$ and $D\leq R$.
The local tube bound gives
\[
 |Q\cap E_r|\leq CR^{n-\beta}r^\beta,\qquad 0<r\leq R.
\]
For $0<h\leq R$, integration of this distribution estimate yields
\begin{equation}\label{a2:eq:truncatedmoment}
 \int_{Q\cap\{d_E<R\}}(\max\{d_E,h\})^{-\beta p}\d x
 \leq C_pR^{n-\beta}h^{-\beta(p-1)}.
\end{equation}
Indeed, split off $d_E<h$ and sum the layers
$2^jh\leq d_E<2^{j+1}h\leq R$. Their contributions form a geometric
series with ratio $2^{-\beta(p-1)}<1$. The last partial layer satisfies
the same bound. On $d_E+h\leq R$, \eqref{a2:eq:phiratio} gives
\[
 \Phi_\beta(d_E+h)
 \leq\Phi_\beta(R)\left(\frac R{\max\{d_E,h\}}\right)^\beta,
\]
and on the complementary set it gives
$\Phi_\beta(d_E+h)\leq\Phi_\beta(R)$.
Thus \eqref{a2:eq:truncatedmoment} implies
\[
 \fint_Q(\Phi_\beta(d_E+h))^p\d x
 \leq C_p(\Phi_\beta(R))^p(R/h)^{\beta(p-1)}.
\]
Insert $h=h_k$ into \eqref{a2:eq:momentintegral} to obtain
\eqref{a2:eq:spatialmoment}. If $R\geq1$ and $h_k\geq1$, the right side of
\eqref{a2:eq:momentintegral} is bounded by a constant. If $h_k<1$, use
$\Phi_\beta(r)\leq r^{-\beta}$ for $r<1$, $\Phi_\beta(r)=1$ for $r\geq1$,
and the global tube bound $|E_r|\leq Cr^\beta$.
The same distribution calculation gives
\[
 \fint_Q(\Phi_\beta(d_E+h_k))^p\d x
 \leq C_p\bigl(1+R^{-n}h_k^{-\beta(p-1)}\bigr)
 \leq C_p\,3^{k\beta(p-1)}.
\]
This proves the asserted large-cube estimate.
\end{proof}

\subsection{The reverse inclusion also fails}\label{a2:sec:converse}

The Cantor construction gives a coefficient satisfying uniform coarse
ellipticity but lying outside the classical $A_2$ weighted class.
To prove incomparability, we now establish the opposite non-inclusion.
The following proposition exhibits a single $A_2$ weight whose scalar
coefficient fails the coarse hypothesis for every admissible choice
of parameters.

\begin{prop}\label{a2:prop:converse}
Let $n\geq3$ and, for $x\notin\{0,e_1\}$, define
\begin{equation}\label{a2:eq:reverseweight}
 w_0(x)={|x|^2}/{|x-e_1|^2},\qquad a_0(x)=w_0(x)\Id,
\end{equation}
with arbitrary positive values at the two exceptional points.
Then $w_0\in A_2(\R^n)$ and $a_0$ satisfies \eqref{eq:basic}.
However, there is no choice of $p,q,s,t$ satisfying \eqref{eq:admissible}
for which both coarse sums are finite on $Q=(-1,2)^n$.
\end{prop}

\begin{proof}
Near $0$, the weight is comparable to $|x|^2$, and near $e_1$ it
is comparable to $|x-e_1|^{-2}$. Since $n>2$, each power and its
inverse is locally integrable. The elementary power-weight calculation
gives uniformly bounded $A_2$ products on sufficiently small balls
near either point. Small balls away from the two points either have
distances to both points larger than a fixed multiple of their radius,
so $w_0$ is comparable to a constant, or are contained in one of the
preceding power-weight neighborhoods. It remains to consider balls whose radii are at least the
fixed positive threshold used above. The inequalities
\[
 w_0(x)\leq2+2/{|x-e_1|^2},\qquad
 w_0(x)^{-1}\leq2+2/{|x|^2}
\]
show that each normalized integral is uniformly bounded on such balls.
Indeed, the integral of $|x-z|^{-2}$ over a ball of radius $R$ is at
most $C_nR^{n-2}$. This proves $w_0\in A_2$ globally.

The points $0$ and $e_1$ are vertices of the triadic grid of
$Q=(-1,2)^n$ at every generation $k\geq1$.  Write $h_k=3^{1-k}$. For all sufficiently large
$k$, there is a simplex $T_k^-$ with vertex $e_1$ and scale $h_k$.
On this simplex $w_0\geq c h_k^{-2}$. Fix $e\in\mathbb R^n$. Every admissible competitor $h$ with
$h-e\cdot x\in H^1_{a_0,0}(T_k^-)$ satisfies
\[
 \fint_{T_k^-}\nabla h\d x=e.
\]
Again this follows by approximation of the zero boundary part. The coefficient lower bound and Cauchy-Schwarz therefore give
\[
 (a_0)_{T_k^-}e\cdot e
 \geq c h_k^{-2}
       \inf_{h-e\cdot x\in H^1_{a_0,0}(T_k^-)}
       \fint_{T_k^-}|\nabla h|^2\d x
 \geq c h_k^{-2}|e|^2.
\]
Near $e_1$, the coefficient lower bound also ensures
the unweighted gradient integrability needed for this calculation. There is likewise a simplex $T_k^+$ with vertex $0$ on which
$w_0\leq C h_k^2$. In \eqref{eq:lowermatrix}, restricting to affine
functions and optimizing their gradients gives
\[
 (a_0)_{*,T_k^+}^{-1}
 \geq\left(\fint_{T_k^+}a_0\d x\right)^{-1}
 \geq c h_k^{-2}\Id.
\]
Since $\#\cT_k(Q)=n!\,3^{kn}$, the $k$th terms of the two
discounted sums are bounded below, up to fixed positive factors, by
\[
 3^{-k\left(s-1+n/(2p)\right)},\qquad
 3^{-k\left(t-1+n/(2q)\right)},
\]
respectively. Finiteness of both sums thus requires
\begin{equation}\label{a2:eq:necessary}
 2s>2- n/p,\qquad 2t>2- n/q.
\end{equation}
Set $A=1/p+1/q<2$. Adding \eqref{a2:eq:necessary} gives
\[
 2(s+t)+(n-1)A>4-A>2.
\]
This contradicts \eqref{eq:admissible}, which requires the same
left side to be strictly smaller than two.
\end{proof}

We can now conclude the coefficient-class theorem stated in the introduction.

\begin{proof}[Proof of Theorem~\ref{thm:coefficientclasses}]
Fix an arbitrary tuple $(p,q,s,t)$ satisfying \eqref{eq:admissible}.
Theorem~\ref{a2:thm:example}, with Cantor codimension $\beta=s$, proves
the first assertion for this tuple. Proposition~\ref{a2:prop:converse}
proves the second assertion. The first coefficient belongs to the
prescribed coarse class and to no $A_2$ weighted elliptic class, while
the second belongs to the latter class and admits no admissible coarse
parameters. This proves incomparability for each fixed tuple and also
for the union over all admissible tuples.
\end{proof}

\begin{rem} For the Cantor coefficient, Theorems~\ref{thm:dirichletmain} and
\ref{thm:BHImain} hold in the stated NTA geometries, with quantitative
constants depending only on $n,p,q,s,t$ and the geometric data, since
the coarse bound is controlled by $n,s$. This supplies an example
outside the weighted class for every admissible parameter choice.
No restriction
on the position of $E$ relative to $\partial\Om$ is needed. In particular,
the boundary may pass through $E$. Proposition~\ref{prop:A2} also shows
that this coefficient cannot satisfy the uniform moment criterion for
any $p,q>1$. The example therefore proves that the full coarse criterion
strictly extends that sufficient moment criterion, while
Theorem~\ref{thm:coefficientclasses} specifies its different relation to
the entire $A_2$ weighted class.
\end{rem}

\section{Boundary regularity and the continuous Dirichlet problem}\label{sec:boundarydirichlet}

Having established the scope of the coefficient condition, we now prove
the boundary regularity and continuous Dirichlet assertions of
Theorem~\ref{thm:dirichletmain}. We assume \eqref{eq:basic}, \eqref{eq:admissible}, and \eqref{eq:UC}
throughout this section, except where a weaker hypothesis is stated
explicitly.

\subsection{Zero extension and boundary decay}\label{sec:boundary}

The boundary argument uses weak Harnack on ambient cubes.
Recall that $Q_R(w)=w+(-R,R)^n$, and let $\omega_n$ denote the
volume of the unit ball. The exterior corkscrew condition gives
\begin{equation}\label{eq:exteriordensity}
 \bigl|(5/8)Q_R(w)\setminus\Om\bigr|
 \geq c_e\bigl|(5/8)Q_R(w)\bigr|,
 \qquad c_e=\omega_n(\kappa/4)^n,
 \qquad w\in\partial\Om,\quad 0<R<R_0.
\end{equation}
Indeed, an exterior corkscrew at scale $R/2$ gives a ball of radius
$\kappa R/2$ inside $B(w,R/2)\subset(5/8)Q_R(w)$. Its relative volume
is $\omega_n(2\kappa/5)^n\geq c_e$. The smaller fraction $c_e$ also
works for the cubes used in the thin-layer argument below.

The next lemma shows that zero extension preserves the subsolution property
under the local energy boundary condition.

\begin{lem}\label{lem:zeroextension}
Let $\Om,W\subset\R^n$ be open, and let $u\geq0$ be a weak subsolution
in $\Om\cap W$, with zero local energy boundary values on
$\partial\Om\cap W$ in the sense of Definition~\ref{def:zerotrace}.
Its zero extension $\widetilde u$ is a weak subsolution in $W$.
\end{lem}

\begin{proof}
Local membership of $\widetilde u$ in the ambient form domain follows from
the definition of the boundary condition. Let $\phi\in C_0^\infty(W)$ be
nonnegative and let $T_\varepsilon(z)=\min\{z/\varepsilon,1\}$ for $z\geq0$.
Choose $0\leq\zeta\in C_0^\infty(W)$ equal to one near
$\supp\phi$. By the boundary condition, the zero extension of
$\zeta u$ belongs to $\cF_0(\Om)$. Its composition with
$T_\varepsilon$ is bounded and belongs to the same space, so
$\phi T_\varepsilon(\zeta u)\in\cF_0(\Om)$ by
Lemma~\ref{lem:products}. On $\Om$ this is
$\phi T_\varepsilon(u)$. The boundary condition also gives finite
energy of $u$ on the intersection of $\Om$ with every compact
subset of $W$. Thus the boundary testing assertion of
Lemma~\ref{lem:products} applies. Explicitly, bounded nonnegative
core approximants to $T_\varepsilon(\zeta u)$, multiplied by
$\phi$, give nonnegative tests in $C_0^\infty(\Om\cap W)$
converging in energy. The subsolution inequality therefore gives
\begin{equation}\label{eq:zerotest}
 \int_{\Om}T_\varepsilon(u)a\nabla u\cdot\nabla\phi\d x
 +\frac1\varepsilon\int_{\{0<u<\varepsilon\}}
       \phi a\nabla u\cdot\nabla u\d x\leq0.
\end{equation}
The second term is nonnegative. The first converges by energy
Cauchy-Schwarz and dominated convergence. The gradient of $u$ is zero
almost everywhere on $\{u=0\}$. Hence
\[
 \int_W a\nabla\widetilde u\cdot\nabla\phi\d x
 =\int_{\Om}a\nabla u\cdot\nabla\phi\d x\leq0.
\]
This is the asserted subsolution inequality.
\end{proof}

The following lemma is the precise consequence of weak Harnack that will
be iterated. Its measure is ordinary Lebesgue measure.

\begin{lem}\label{lem:zeroset}
Suppose $f$ is an ambient subsolution in a cube $Q$, $0\leq f\leq M_0$,
and
\[
 |\{f=0\}\cap(5/8)Q|\geq\mu |(5/8)Q|,
 \qquad 0<\mu\leq1.
\]
Then
\begin{equation}\label{eq:zeroset}
 \esssup_{(1/2)Q}f\leq\vartheta M_0,
 \qquad \vartheta=1-\frac{\mu^{1/\eta}}{H_\eta}\in(0,1).
\end{equation}
\end{lem}

\begin{proof}
The function $M_0-f$ is a nonnegative supersolution. By
\eqref{eq:WH},
\[
 \mu^{1/\eta}M_0
 \leq\left(\avgint_{(5/8)Q}(M_0-f)^\eta\d x\right)^{1/\eta}
 \leq H_\eta\left(M_0-\esssup_{(1/2)Q}f\right).
\]
Rearrangement gives the assertion, including the trivial case $M_0=0$.
\end{proof}

The following proposition combines zero extension with exterior density to
obtain boundary decay.

\begin{prop}\label{prop:boundarycontraction}
Assume \eqref{eq:UC} on the cubes involved and the exterior
density condition \eqref{eq:exteriordensity} for $\Om$. Let
$w\in\partial\Om$ and $0<R<R_0$. Suppose that $u\geq0$
is a bounded weak subsolution of $\La u=0$ in
$\Om\cap Q_R(w)$, with zero local energy boundary values
on $\partial\Om\cap Q_R(w)$.  With the fixed exponent
$\eta$ chosen after \eqref{eq:Hconstants}, set
\begin{equation}\label{eq:boundaryconstants}
 \varepsilon_B={c_e^{1/\eta}}/{H_\eta},\qquad
 \kappa_B=1-\varepsilon_B,\qquad
 \beta_B=-{\log\kappa_B}/{\log2}>0.
\end{equation}
Then
\begin{align}
 \esssup_{\Om\cap Q_{R/2}(w)}u
 &\leq\kappa_B\esssup_{\Om\cap Q_R(w)}u,\quad\mbox{and}
       \label{eq:boundarycontraction}\\
 \esssup_{\Om\cap Q_\rho(w)}u
 &\leq C\left(\frac\rho R\right)^{\beta_B}
                 \esssup_{\Om\cap Q_R(w)}u,
 \qquad 0<\rho\leq R.\label{eq:boundarypower}
\end{align}
\end{prop}

\begin{proof}
The zero extension $\widetilde u$ is an ambient subsolution by
Lemma~\ref{lem:zeroextension}. It vanishes on the exterior set in
\eqref{eq:exteriordensity}. Apply Lemma~\ref{lem:zeroset} with
$\mu=c_e$ and $M_0=\esssup_{\Om\cap Q_R(w)}u$ to obtain
\eqref{eq:boundarycontraction}. Repeating on the concentric cubes
$Q_{2^{-j}R}(w)$ gives \eqref{eq:boundarypower}, with
$C=2^{\beta_B}$ for intermediate radii.
\end{proof}

The next corollary records the same estimate with Euclidean balls,
which will be convenient in the comparison proof.

\begin{cor}\label{cor:boundaryballs}
Assume \eqref{eq:exteriordensity} and \eqref{eq:UC} on the
cubes involved. Let $w\in\partial\Om$, $0<R<R_0$, and let $h\geq0$ be a bounded
subsolution in $D_R(w)$ with zero local energy boundary values on
$\partial\Om\cap B(w,R)$. Then
\begin{equation}\label{eq:boundaryballpower}
 \esssup_{D_t(w)}h\leq C_B(t/R)^{\beta_B}
                      \esssup_{D_R(w)}h,
 \qquad 0<t\leq R,
\end{equation}
where $C_B$ depends only on $n,\kappa$ and $\mathbf d$.
\end{cor}

\begin{proof}
Set $a_R=R/(2\sqrt n)$. For $t\leq a_R$, use
$D_t(w)\subset\Om\cap Q_t(w)$ and
$Q_{a_R}(w)\subset B(w,R)$ in
\eqref{eq:boundarypower}. This gives the estimate with
$C_B=(4\sqrt n)^{\beta_B}$. For $a_R<t\leq R$, the same bound follows
from monotonicity of the supremum and this choice of $C_B$.
\end{proof}

The boundedness assumption in Proposition~\ref{prop:boundarycontraction}
and Corollary~\ref{cor:boundaryballs} is automatic on smaller
neighborhoods compactly contained in a region where the zero local
energy boundary condition holds. Indeed, the zero extension is an
ambient subsolution by Lemma~\ref{lem:zeroextension} and is therefore
locally bounded by Proposition~\ref{prop:AAKT}. The boundary decay
estimate then shows that every nonnegative weak solution with zero
local energy boundary values extends continuously to the corresponding
physical boundary portion, with value zero there. This proves the
zero boundary decay assertion in Theorem~\ref{thm:dirichletmain}.

The next proposition establishes the converse: for nonnegative weak
solutions, continuous vanishing on a boundary portion implies zero
local energy boundary values there.

\begin{prop}
\label{prop:continuouszero}
Let $\Om,W\subset\R^n$ be open, and let $u\geq0$ be a weak solution
in $\Om\cap W$. Suppose that $u$ is continuous on
$(\Om\cup\partial\Om)\cap W$ and equals zero on
$\partial\Om\cap W$. Then $u$ has zero local energy boundary values in
$W$. This implication only uses \eqref{eq:basic} and the weak formulation in the local form domain.
\end{prop}

\begin{proof}
Fix $\eta\in C_0^\infty(W)$. Choose a second cutoff equal to one near
$\supp\eta$ to keep all subsequent supports in a compact subset of $W$.
Continuity and vanishing on the physical boundary show that the support of
$\eta(u-\varepsilon)_+$ stays a positive distance from $\partial\Om$ for
each $\varepsilon>0$. The function $u$ is bounded on this compact boundary
neighborhood. Test the equation with $\eta^2(u-\varepsilon)_+$ to obtain
\begin{equation}\label{eq:levelenergy}
 \int_{\{u>\varepsilon\}}\eta^2a\nabla u\cdot\nabla u\d x
 \leq4\int_{\Om}(u-\varepsilon)_+^2
                   a\nabla\eta\cdot\nabla\eta\d x.
\end{equation}
The right side is bounded independently of $\varepsilon$. Applying the
estimate also to the larger cutoff establishes finite energy of $u$ up
to the physical boundary on $\supp\eta$. Letting $\varepsilon$ decrease to
zero proves this claim by monotone convergence and the fact that
$\nabla u=0$ almost everywhere on $\{u=0\}$.

Each zero extension of $\eta(u-\varepsilon)_+$ belongs to $\cF_0(\Om)$.
These functions converge to the zero extension of $\eta u$ in $\mathrm L^2$.
The gradient error is the sum of a term supported on
$\{0<u\leq\varepsilon\}$ and a term bounded by
$\varepsilon|\nabla\eta|$. The first has energy tending to zero by the
finite energy just established, and the second is controlled by
$\varepsilon^2\int a\nabla\eta\cdot\nabla\eta$. Closedness proves the
asserted membership in $\cF_0(\Om)$.
\end{proof}

\begin{rem}\label{rem:traceequivalence}
Note the equivalence of the two boundary
conditions. Indeed, under the standing coefficient assumptions and the
exterior density condition \eqref{eq:exteriordensity}, nonnegative
$\La$-harmonic functions with zero local energy boundary values vanish continuously by
Proposition~\ref{prop:boundarycontraction}. Proposition~\ref{prop:continuouszero}
gives the converse. Hence the two boundary
conditions are equivalent. Thus Theorem~\ref{thm:BHImain} also applies directly to
positive weak solutions that vanish continuously on the boundary patch,
as in the usual formulation of the boundary Harnack principle.
\end{rem}

\subsection{The continuous Dirichlet problem}\label{sec:dirichlet}

Let $\Om$ be a bounded NTA domain. We construct solutions with
continuous boundary data by uniformly approximating the data by
Lipschitz functions and passing to the limit in the corresponding
variational solutions. To prove that these solutions attain their
boundary values continuously, we first verify that their level
truncations satisfy the local energy boundary condition required
for the boundary decay estimates. The following lemma establishes
this property using the approximation that defines $\cF_0(\Om)$.

\begin{lem}\label{lem:boundarylevels}
Let $F$ be bounded and Lipschitz on a neighborhood of $\overline\Om$,
let $z\in\cF_0(\Om)$ be bounded, and write $u=F+z$ in $\Om$.
If $W$ is open and $F\leq b$ on $\partial\Om\cap W$, then
$(u-b)^+$ has zero local energy boundary values in $W$.
If $F\geq b$ there, the same holds for $(b-u)^+$.
\end{lem}

\begin{proof}
Choose uniformly bounded $z_j\in C_0^\infty(\Om)$ converging
to $z$ in the form norm and almost everywhere, as supplied by
Lemma~\ref{lem:strongtruncation}. Fix a nonnegative
$\zeta\in C_0^\infty(W)$. The functions
$g_j=\zeta(F+z_j-b)^+$ are Lipschitz on a neighborhood of
$\overline\Om$ and vanish on all of $\partial\Om$. Indeed,
$z_j$ is zero near the boundary, $F\leq b$ on the boundary portion in
$W$, and $\zeta$ vanishes outside $W$.
Lemma~\ref{lem:Lipzero} puts their zero extensions in $\cF_0(\Om)$.

Lemma~\ref{lem:strongtruncation} gives convergence of the gradients
of $(F+z_j-b)^+$ to that of $(u-b)^+$ in the $a$ energy. For the
cutoff term, uniform boundedness and almost everywhere convergence give
dominated convergence against $a\nabla\zeta\cdot\nabla\zeta$.
The functions also converge in $\mathrm L^2$. Closedness therefore
places the zero extension of $\zeta(u-b)^+$ in $\cF_0(\Om)$.
Differences of nonnegative cutoffs give the full boundary condition.
Apply the same argument to $-u$ for the second assertion.
\end{proof}

We now establish boundary continuity for the bounded variational solutions
from Proposition~\ref{prop:variational}. Let $f$ be Lipschitz on
$\partial\Om$ and let $u$ be its variational extension. Fix
$w\in\partial\Om$ and $0<R<R_0$. Put
\begin{equation}\label{eq:boundarylevels}
 b_R^+=\sup_{\partial\Om\cap Q_R(w)}f,\qquad
 b_R^-=\inf_{\partial\Om\cap Q_R(w)}f.
\end{equation}
The functions $(u-b_R^+)_+$ and $(b_R^--u)_+$ have zero local energy
boundary values in $Q_R(w)$ by Lemma~\ref{lem:boundarylevels}. The two truncations are subsolutions. Applying
Proposition~\ref{prop:boundarycontraction} gives, for $0<\rho\leq R$,
\begin{equation}\label{eq:boundarymodulus}
 \sup_{\Om\cap Q_\rho(w)}|u-f(w)|
 \leq\osc_{\partial\Om\cap Q_R(w)}f
       +C\left(\frac\rho R\right)^{\beta_B}\|f\|_{C(\partial\Om)}.
\end{equation}
Interior continuity has already been proved. Letting first $\rho$ and then
$R$ tend to zero shows that $u$ extends continuously to the boundary with
value $f$.

\subsection{Proof of Theorem~\ref{thm:dirichletmain}}

We first give an explicit approximation and extension of the boundary
data. Set $m=\min_{\partial\Om}f$, $M_0=\max_{\partial\Om}f$, and
\[
 F_j(x)=\min\left\{M_0,\inf_{z\in\partial\Om}
                  \bigl(f(z)+j|x-z|\bigr)\right\},
 \qquad f_j=F_j|_{\partial\Om}.
\]
Then $F_j$ is globally Lipschitz with constant at most $j$, and
$m\leq F_j\leq M_0$. Uniform continuity of $f$ on the compact boundary
shows that $f_j\to f$ uniformly. Indeed, points $z$ close to a given boundary
point give an error controlled by the modulus of $f$, while points
farther away incur the term $j|x-z|$. Let $u_j$ be the variational
solutions furnished by Proposition~\ref{prop:variational}.
The preceding argument gives $u_j\in C(\overline\Om)$ with boundary values
$f_j$. The contraction estimate gives
\begin{equation}\label{eq:uniformCauchy}
 \|u_j-u_k\|_{C(\overline\Om)}\leq\|f_j-f_k\|_{C(\partial\Om)}.
\end{equation}
Consequently $u_j$ converges uniformly on $\overline\Om$ to a continuous
function $u$ with boundary values $f$. To verify that the limit is a weak
solution, fix a compactly supported cutoff $\eta$ in $\Om$ and apply
Lemma~\ref{lem:Caccioppoli} to $u_j-u_k$. It shows convergence in local
energy as well as uniform convergence. Hence $u\in\cF_{\loc}(\Om)$ and
$\int_\Om a\nabla u\cdot\nabla\phi\d x=0$ for every
$\phi\in C_0^\infty(\Om)$.

To prove uniqueness in the stated class, let $h$ be the difference of two
solutions with the same boundary values. It is continuous on
$\overline\Om$ and vanishes on $\partial\Om$. For $\varepsilon>0$,
$(h-\varepsilon)_+$ has compact support in $\Om$ and is an admissible
test function. Testing gives
$\cE((h-\varepsilon)_+,(h-\varepsilon)_+)=0$. The zero boundary coercivity
\eqref{eq:energycoerciveLone} makes this function zero. Apply the same
argument to $-h$ and then let $\varepsilon\downarrow0$.

Linearity, positivity and the contraction estimate pass from the
variational solutions to their uniform limits. The constant function one
is the solution with boundary value one. For each fixed $x\in\Om$, the
map $f\mapsto H_\Om f(x)$ is a positive linear functional of norm one on
$C(\partial\Om)$. The Riesz representation theorem gives the unique
probability measure in \eqref{eq:representationmain}. The interior and
zero boundary regularity assertions follow from
Subsections~\ref{sec:interior} and \ref{sec:boundary}.\qed

The next proposition quantifies the boundary modulus of the solution and yields
H\"older estimates for H\"older boundary data.

\begin{prop}\label{prop:datamodulus}
Let $u=H_\Om f$, $w\in\partial\Om$, and $0<R<R_0$.
For $R_j=2^{-j}R$ set
\[
 A_j=\sup_{\Om\cap Q_{R_j}(w)}|u-f(w)|,\qquad
 \omega_j=\sup_{\partial\Om\cap Q_{R_j}(w)}|f-f(w)|.
\]
Then
\begin{equation}\label{eq:datarecursion}
 A_{j+1}\leq\kappa_B A_j+\omega_j,
 \qquad
 A_k\leq\kappa_B^k A_0+
       \sum_{j=0}^{k-1}\kappa_B^{k-1-j}\omega_j.
\end{equation}
In particular, suppose
$|f(z)-f(w)|\leq F|z-w|^\gamma$ for
$z\in\partial\Om\cap Q_R(w)$, where $F\geq0$ and $0<\gamma\leq1$.
Then, for every
$0<\beta<\min\{\beta_B,\gamma\}$,
\begin{equation}\label{eq:holderdata}
 |u(x)-f(w)|\leq C_\beta
 \left(\|f\|_\infty+FR^\gamma\right)
 \left(\frac{|x-w|}{R}\right)^\beta
\end{equation}
for $x\in\Om\cap B(w,R/2)$, where $C_\beta$ depends only on
$n,\mathbf d,\kappa,\gamma,\beta$.
\end{prop}

\begin{proof}
First use Lipschitz data. Apply the boundary contraction to
$(u-f(w)-\omega_j)_+$ and $(f(w)-u-\omega_j)_+$ on $Q_{R_j}(w)$.
Each has supremum at most $A_j$, and both have zero energy boundary values
on that cube. This gives the first inequality in \eqref{eq:datarecursion}.
Uniform approximation of continuous data passes the same inequality to
$u=H_\Om f$, since each boundary supremum converges uniformly with the data.
Iteration gives the second inequality. Under the H\"older assumption,
$\omega_j\leq C_nFR^\gamma2^{-j\gamma}$. The resulting geometric sum is
bounded by $C_\beta2^{-k\beta}$ for the indicated strict range of $\beta$.
Comparing $|x-w|$ with the dyadic cube radii proves \eqref{eq:holderdata}.
At coinciding endpoint exponents the convolution can contain a logarithmic
factor, which is why the strict inequality is convenient.
\end{proof}

The following corollary compares harmonic measures with different interior
poles.

\begin{cor}
\label{cor:measurepoles}
For any $x,y\in\Om$, there is $C_{x,y}<\infty$ such that
\begin{equation}\label{eq:polecomparison}
 C_{x,y}^{-1}\omega^y_\Om(E)
 \leq\omega^x_\Om(E)\leq C_{x,y}\omega^y_\Om(E)
\end{equation}
for every Borel set $E\subset\partial\Om$.
\end{cor}

\begin{proof}
A finite chain of interior balls connects $x$ and $y$. Apply Harnack along
the chain to $H_\Om f$ for nonnegative continuous $f$. The resulting
inequalities for positive linear functionals imply the corresponding
inequalities for their representing regular Borel measures.
\end{proof}

\begin{rem}
Under the present hypotheses, bounded-domain Green kernels are
available from the abstract theory in \cite{BCM,ChenBHP}, using
the energy-space identifications above and the elliptic Harnack
inequality. The direct argument in the next section avoids their
construction and uses local estimates and boundary testing.
Separate questions concern global fundamental solutions under
suitable additional hypotheses and quantitative estimates for suitably
normalized kernels in terms of the coarse ellipticity data. We do not pursue these
questions here.
\end{rem}

\section{Boundary Harnack and H\"older continuity of quotients}\label{sec:comparison}

Throughout this section, assume
\eqref{eq:basic}, \eqref{eq:admissible}, and \eqref{eq:UC}, and let
$\Om$ have NTA data $(\kappa,M,R_0)$. All constants depend only on
\begin{equation}\label{eq:comparisondata}
 \mathbf D=(\mathbf d,\kappa,M),
\end{equation}
unless a more restricted dependence is stated. The fixed Harnack
constants $H_b,H_\eta$ and the exponent $\eta$ have already been chosen.
The boundary comparison proof follows the method of
\cite[Sections 2 and 3.2]{DSS}, with the energy tests supplied by
Subsections~\ref{sec:energy} and \ref{sec:boundary}.

We choose the localization constants in the following order. First,
$L_0\geq100(1+M)$ is fixed large enough for the interior chains
in Lemma~\ref{lem:chains}. $L_0$ is then kept fixed in the subsequent lemmas.
Next, the boundary contraction and short-chain constants determine
the Carleson and almost-positivity constants. For the main theorem
we finally take $L=8L_0$. Thus $r<R_0/L$ implies
$4r<R_0/L_0$, $2r<R_0$, and $D_{4L_0r}(w)\subset D_{Lr}(w)$.
Restriction to a smaller ambient neighborhood preserves the zero
local energy boundary condition, since its cutoffs are also admissible
in the original neighborhood. These observations will be used when
changing centers and scales.

\subsection{Harnack chains and the Carleson estimate}\label{sec:Carleson}

We first record the two interior chain consequences separately. The
second one has a constant independent of the boundary layer thickness.
This independence is essential in the almost-positivity argument.

\begin{lem}\label{lem:chains}
There are $L_0\geq16$, $C_g\geq1$, and $\gamma>0$, depending only on
$n,\kappa,M,H_b$, with the following properties.
\begin{enumerate}
\item If $w\in\partial\Om$, $0<r<R_0/L_0$ and $h\geq0$ is $\La$-harmonic in
$D_{L_0r}(w)$, then
\begin{equation}\label{eq:chain-growth}
 h(x)\leq C_g\left(\frac r{\delta(x)}\right)^\gamma h(A_r(w)),
 \qquad x\in D_{6r}(w).
\end{equation}
For each fixed $\varepsilon>0$ there is $c_\varepsilon>0$, depending
only on the same data and $\varepsilon$, such that
\begin{equation}\label{eq:core-lower}
 h(x)\geq c_\varepsilon h(A_r(w)),
 \qquad x\in D_{4r}(w),\quad\delta(x)>2\varepsilon r.
\end{equation}
\item There are $\varepsilon_g>0$ and $c_H\in(0,1/2]$, depending only
on $n,\kappa,M,H_b$, with the following property. If
$z\in\partial\Om$, $0<\varepsilon\leq\varepsilon_g$, $0<R<R_0$, and
\[
 x\in D_{R/2}(z),\qquad \varepsilon R/2<\delta(x)\leq\varepsilon R,
\]
then there is $y\in D_R(z)$ with $\delta(y)>\varepsilon R$ such that
every nonnegative $\La$-harmonic function $h$ in $D_R(z)$ satisfies
\begin{equation}\label{eq:short-chain}
 h(x)\geq c_H h(y).
\end{equation}
The constant $c_H$ does not depend on $\varepsilon$ or $R$.
\end{enumerate}
\end{lem}

\begin{proof}
For (i), join $x$ to $A_r(w)$ by a uniform curve. Its length is at most
$7Mr$, and it stays inside $B(w,(6+7M)r)$. Near the endpoint $x$, split
the curve into arc length portions of size comparable to
$\delta(x),2\delta(x),4\delta(x),\ldots$ up to scale $r$. Perform the
same construction from the other endpoint, whose boundary distance is
at least $\kappa r$. The cigar condition and the fact that $\delta$ is
1-Lipschitz give an interior chain of balls, with doubled balls in $\Om$,
of length at most
\[
 N(x)\leq C\bigl(1+\log^+(r/\delta(x))\bigr).
\]
Here is the local reason for this count. At arc distance $t$ from an
endpoint, once $t$ exceeds a small multiple of that endpoint's boundary
distance and before reaching the middle of the curve, the cigar estimate
gives boundary distance at least $ct$. Each portion with
$t\in[2^jt_0,2^{j+1}t_0]$ is therefore covered by a bounded number of
balls of radius comparable to $2^jt_0$. The initial endpoint portions
are covered using the 1-Lipschitz property of $\delta$.
The balls can be chosen with radii at most a fixed small multiple
of their boundary distances. Their doubled balls then remain in
$\Om$ and in $B(w,C(1+M)r)$, with $C$ dimensional.
Choose $L_0\geq100(1+M)$ larger than this constant. The uniform-curve
hypothesis applies because $|x-A_r(w)|<7r<R_0$.
Repeated use of \eqref{eq:ballHarnack} gives an upper bound
$H_b^{N(x)}h(A_r(w))$ and the reverse lower bound. These imply
\eqref{eq:chain-growth} and \eqref{eq:core-lower}.
For $\delta(x)>r$, increase $C_g$, using $\delta(x)\leq6r$.

For (ii), choose a nearest boundary point $p$ to $x$ and set
$b_0=4/\kappa$. Let $y=A_{b_0\varepsilon R}(p)$. Choose $C_*=10(M+1)(b_0+1)$ and fix
\[
 0<\varepsilon_g<\min\{1/(4b_0),1/(8C_*)\}.
\]
Then $b_0\varepsilon R<R_0$ and
$|x-y|\leq(b_0+1)\varepsilon R<R_0$, so both the corkscrew
and uniform-curve hypotheses apply. In particular,
\[
 \delta(y)\geq4\varepsilon R,\qquad
 |x-y|\leq(b_0+1)\varepsilon R.
\]
A uniform curve from $x$ to $y$ has length at most
$M(b_0+1)\varepsilon R$. On its first short portion, the distance from
the boundary is controlled from below by $\delta(x)>\varepsilon R/2$
and the 1-Lipschitz property. The same argument applies near $y$.
On the remaining portion the cigar estimate gives a lower bound
$c\varepsilon R$. Hence the entire curve admits a chain of a bounded
number $N_*$ of balls of radius comparable to $\varepsilon R$, with
doubled balls in $\Om$. All these balls and their doubles can be chosen within
$C_*\varepsilon R$ of $x$, with the fixed $C_*$ above.
Since $x\in B(z,R/2)$ and $C_*\varepsilon_g<1/8$, they remain
in $B(z,R)$.
Iterated Harnack gives \eqref{eq:short-chain} with, for example,
$c_H=\min\{1/2,H_b^{-N_*}\}$.
\end{proof}

The following proposition establishes the Carleson estimate, which
bounds a nonnegative solution with zero local energy boundary values
by its value at an interior corkscrew point. The proof uses a growth
contradiction and requires no separate integrability estimate for
negative powers of the boundary distance $\delta$.

\begin{prop}\label{prop:Carleson}
For the fixed $L_0$ above, there is a constant $C_C\geq1$,
depending only on \eqref{eq:comparisondata}, such that
\begin{equation}\label{eq:Carleson}
 \sup_{D_{4r}(w)}h\leq C_C h(A_r(w))
\end{equation}
whenever $w\in\partial\Om$, $0<r<R_0/L_0$ and $h\geq0$ is $\La$-harmonic in
$D_{L_0r}(w)$ with zero local energy boundary values on
$\partial\Om\cap B(w,L_0r)$.
\end{prop}

\begin{proof}
If $h(A_r(w))=0$, Harnack chains give $h=0$ on $D_{6r}(w)$,
so the assertion is immediate. Otherwise normalize $h(A_r(w))=1$.
By Lemma~\ref{lem:zeroextension}, the zero extension is an ambient
subsolution in $B(w,L_0r)$. Proposition~\ref{prop:AAKT} therefore makes
it locally bounded there, and hence bounded on $B(w,7r)$.
This bound is finite but is not yet quantitative.
Choose a fixed $K_b>4$ such that
\begin{equation}\label{eq:Kb}
 C_B(2/K_b)^{\beta_B}\leq1/4.
\end{equation}
Suppose $x_0\in D_{4r}(w)$ and $h(x_0)>T$, for a sufficiently large
constant $T$ to be fixed using only the data. If $x_j$ is constructed,
write $\delta_j=\delta(x_j)$ and choose a nearest boundary point $z_j$.
From \eqref{eq:chain-growth}, as long as $x_j\in D_{6r}(w)$,
\begin{equation}\label{eq:small-distance}
 \delta_j\leq r\bigl({C_g}/{h(x_j)}\bigr)^{1/\gamma}.
\end{equation}
Apply Corollary~\ref{cor:boundaryballs} at $z_j$ with outer radius $K_b\delta_j$ and inner
radius $2\delta_j$. Since $x_j\in D_{2\delta_j}(z_j)$, it gives
\[
 h(x_j)\leq\frac14\sup_{D_{K_b\delta_j}(z_j)}h.
\]
We can therefore choose $x_{j+1}$ in that set with
$h(x_{j+1})\geq2h(x_j)$. Its displacement satisfies
\begin{equation}\label{eq:displacement}
 |x_{j+1}-x_j|\leq(K_b+1)\delta_j
 \leq(K_b+1)r\left(\frac{C_g}{2^j h(x_0)}\right)^{1/\gamma}.
\end{equation}
Choose $T$ so large that
\[
 (K_b+1)(C_g/T)^{1/\gamma}
       \sum_{j=0}^\infty2^{-j/\gamma}<1/2.
\]
In particular, $(K_b+1)\delta_j<r/2$ at every stage for which
\eqref{eq:small-distance} holds. This ensures that the boundary
balls used in the induction have radii below $r$ and stay in the
fixed neighborhood of $\La$-harmonicity.
Inductively, every $x_j$ lies in $D_{5r}(w)$, and every ball
$B(z_j,K_b\delta_j)$ is contained in $B(w,6r)$. All its radii are
less than $r<R_0$, so every preceding estimate remains applicable.
But $h(x_j)\geq2^j h(x_0)\to\infty$, contradicting the finite bound on
$B(w,7r)$. Therefore $h\leq T$ on $D_{4r}(w)$, proving the assertion.
\end{proof}

\subsection{Thin boundary layers and almost positivity}\label{sec:almost}

For $z\in\partial\Om$, define the interior core
\begin{equation}\label{eq:core-definition}
 \mathcal A_R^\varepsilon(z)
 =\{x\in D_R(z):\delta(x)>\varepsilon R\}.
\end{equation}
The next lemma quantifies how fast a nonnegative ambient subsolution
decays when its support inside $\Om$ lies in a thin boundary layer.

\begin{lem}\label{lem:thin}
Set $b=4\sqrt n$. There is $c_*>0$, depending only on
$n,\kappa,H_\eta,\eta$, such that the following holds for
$z\in\partial\Om$, $0<\varepsilon\leq1/(4b)$, and
$0<R<R_0$. Suppose $f$ is an ambient
subsolution in $B(z,R)$, $0\leq f\leq M_0$, and $f=0$ almost everywhere
on both $B(z,R)\setminus\Om$ and
$\mathcal A_R^\varepsilon(z)$. Then
\begin{equation}\label{eq:thin}
 \esssup_{B(z,R/2)}f\leq M_0\exp(-c_*/\varepsilon).
\end{equation}
One may take $c_*=-\log\kappa_B/(4b)$, with $\kappa_B$ as in
\eqref{eq:boundaryconstants}.
\end{lem}

\begin{proof}
Let $F(t)=\esssup_{B(z,t)}f$. We claim that
\begin{equation}\label{eq:layer-recursion}
 F(t)\leq\kappa_B F(t+b\varepsilon R),
 \qquad t+b\varepsilon R\leq R.
\end{equation}
Consider a point $x\in D_t(z)$ with $\delta(x)\leq\varepsilon R$ and
a nearest boundary point $p$. Take
\[
 Q=p+(-3\varepsilon R,3\varepsilon R)^n.
\]
Then $x\in(1/2)Q$, and
$Q\subset B(x,b\varepsilon R)\subset B(z,t+b\varepsilon R)$.
The cube $(5/8)Q$ contains $B(p,\varepsilon R)$, which contains an
exterior corkscrew ball of radius $\kappa\varepsilon R$. Its relative
volume in $(5/8)Q$ is $\omega_n(4\kappa/15)^n\geq c_e$.
Lemma~\ref{lem:zeroset}, with upper bound
$F(t+b\varepsilon R)$, bounds $f$ by the right side of
\eqref{eq:layer-recursion} almost everywhere near $x$.
The same bound holds at every Lebesgue point in this layer. Elsewhere
$f=0$. Taking the essential supremum proves the claim. Next, iterate the claim starting at $t=R/2$ for
$k=\lfloor1/(2b\varepsilon)\rfloor$ steps. The restriction on
$\varepsilon$ gives $k\geq1/(4b\varepsilon)$, and hence
\[
 F(R/2)\leq\kappa_B^k F(R/2+kb\varepsilon R)
 \leq M_0\exp\bigl({\log\kappa_B}/{(4b\varepsilon)}\bigr).
\]
This is \eqref{eq:thin}.
\end{proof}

The following lemma is the signed comparison step. It applies to a
linear combination of the original solutions, so no auxiliary Dirichlet
problem on a boundary layer is required.

\begin{lem}\label{lem:almost}
There are constants $M_*>1$, $\varepsilon_*\in(0,1/4)$, and
$a_*\in(0,1)$, depending only on \eqref{eq:comparisondata}, such that if
$z\in\partial\Om$, $0<R<R_0$, and $h$ is a $\La$-harmonic function
in $D_R(z)$, with zero local energy boundary values on
$\partial\Om\cap B(z,R)$, and
\begin{equation}\label{eq:almost-assumptions}
 h\geq-1\quad\hbox{in }D_R(z),\qquad
 h\geq M_*\quad\hbox{in }\mathcal A_R^{\varepsilon_*}(z),
\end{equation}
then
\begin{equation}\label{eq:almost-conclusion}
 h\geq-a_*\quad\hbox{in }D_{R/2}(z),\qquad
 h\geq a_*M_*\quad\hbox{in }\mathcal A_{R/2}^{\varepsilon_*}(z).
\end{equation}
\end{lem}

\begin{proof}
First take $c_H$ and $\varepsilon_g$ from
Lemma~\ref{lem:chains}(ii), and set
\begin{equation}\label{eq:choice-aM}
 a_*=c_H/2,\qquad M_*=\max\{2,2/c_H\}.
\end{equation}
Choose $\varepsilon_*>0$ only after these constants have been fixed,
so that
\begin{equation}\label{eq:choice-eps}
 \varepsilon_*<\min\{\varepsilon_g,1/4,1/(4b)\},\qquad
 \exp(-c_*/\varepsilon_*)\leq a_*.
\end{equation}
These choices are successive, since $c_H$, $M_*$ and $a_*$ are independent of
the layer thickness. For $x\in\mathcal A_{R/2}^{\varepsilon_*}(z)$ with
$\delta(x)>\varepsilon_*R$, the desired positive lower bound follows
already from \eqref{eq:almost-assumptions}. Otherwise
$\varepsilon_*R/2<\delta(x)\leq\varepsilon_*R$. Apply the short chain
estimate to the nonnegative $\La$-harmonic function $h+1$, connecting $x$
to a point $y\in\mathcal A_R^{\varepsilon_*}(z)$. It gives
\[
 h(x)\geq c_H(h(y)+1)-1
 \geq c_H(M_*+1)-1\geq a_*M_*.
\]
This proves the second conclusion. For the first conclusion, $h^-$ is a subsolution with zero local
energy boundary values by Lemma~\ref{lem:truncations}. Extend it
by zero. Lemma~\ref{lem:zeroextension} makes this an ambient
subsolution. It is bounded by 1 and vanishes on
$\mathcal A_R^{\varepsilon_*}(z)$. Lemma~\ref{lem:thin} and
\eqref{eq:choice-eps} give
$h^-\leq\exp(-c_*/\varepsilon_*)\leq a_*$ in $D_{R/2}(z)$.
Interior continuity converts the almost everywhere inequalities to
pointwise inequalities.
\end{proof}

The next corollary states exactly what repeated almost positivity gives
at a fixed center. Its restriction to a non-tangential region will be
removed by recentering in Subsection~\ref{sec:BHI}.

\begin{cor}\label{cor:iteration}
Under \eqref{eq:almost-assumptions}, for every integer $j\geq0$,
\begin{equation}\label{eq:almost-iteration}
 h\geq-a_*^j\ \hbox{in }D_{2^{-j}R}(z),\qquad
 h\geq M_*a_*^j\ \hbox{in }\mathcal A_{2^{-j}R}^{\varepsilon_*}(z).
\end{equation}
In particular, $h(x)>0$ for every $x\in\Om$ satisfying
\begin{equation}\label{eq:non-tangential}
 |x-z|<R,\qquad \delta(x)>2\varepsilon_*|x-z|.
\end{equation}
\end{cor}

\begin{proof}
After the first step, apply Lemma~\ref{lem:almost} to $h/a_*$
at radius $R/2$. Continue inductively to obtain
\eqref{eq:almost-iteration}. For a point satisfying
\eqref{eq:non-tangential}, choose a dyadic radius
$\rho=2^{-j}R$ with $|x-z|<\rho\leq2|x-z|$. Then
$\delta(x)>\varepsilon_*\rho$, so $x$ belongs to the positive core
at this radius.
\end{proof}

\subsection{Boundary Harnack in NTA domains}\label{sec:BHI}

We now apply the signed lemma to a normalized difference and use a
different boundary center for each point. This is the step that turns
positivity in the regions \eqref{eq:non-tangential} into a full
boundary Harnack inequality.

\begin{prop}\label{prop:comparison}
For the fixed $L_0$ above, there is $T\geq2$ depending
only on \eqref{eq:comparisondata} such that, for $w\in\partial\Om$ and
$0<r<R_0/L_0$, nonnegative $\La$-harmonic functions
$u,v$ in $D_{L_0r}(w)$ with zero local energy boundary values on
$\partial\Om\cap B(w,L_0r)$ and $u(A_r(w)),v(A_r(w))>0$ are positive in $D_r(w)$ and satisfy
\begin{equation}\label{eq:comparison}
 T^{-1}\frac{u(A_r(w))}{v(A_r(w))}
 \leq\frac{u(x)}{v(x)}
 \leq T\frac{u(A_r(w))}{v(A_r(w))},\qquad x\in D_r(w).
\end{equation}
\end{prop}

\begin{proof}
Normalize $u(A_r(w))=v(A_r(w))=1$. Interior Harnack chains
give positivity of both functions on $D_{6r}(w)$, and the
Carleson estimate gives $v\leq C_C$ in $D_{4r}(w)$.
Let $M_*$ and $\varepsilon_*$ be the constants from
Lemma~\ref{lem:almost}, chosen in \eqref{eq:choice-aM}
and \eqref{eq:choice-eps}. The core lower bound
\eqref{eq:core-lower} gives
\[
 u\geq c_{\varepsilon_*}>0
 \quad\hbox{on }D_{4r}(w)\cap\{\delta>2\varepsilon_*r\}.
\]
Set
\begin{equation}\label{eq:comparison-T}
 T=\max\left\{2,\frac{C_C(M_*+1)}{c_{\varepsilon_*}}\right\},
 \qquad h=\frac{Tu-v}{C_C}.
\end{equation}
Then $h$ is a signed $\La$-harmonic function in $D_{4r}(w)$
with zero local energy boundary values on
$\partial\Om\cap B(w,4r)$. Since $u\geq0$ and $v\leq C_C$
there, the choice of $T$ yields
\begin{equation}\label{eq:W-start}
 \begin{aligned}
 h&\geq-1
 &&\hbox{in }D_{4r}(w),\\
 h&\geq\frac{T c_{\varepsilon_*}-C_C}{C_C}\geq M_*
 &&\hbox{in }D_{4r}(w)\cap\{\delta>2\varepsilon_*r\}.
 \end{aligned}
\end{equation}
For every $z\in\Delta_{2r}(w)$, we have
$D_{2r}(z)\subset D_{4r}(w)$ and
\[
 \mathcal A_{2r}^{\varepsilon_*}(z)
 \subset D_{4r}(w)\cap\{\delta>2\varepsilon_*r\}.
\]
Thus the hypotheses of Lemma~\ref{lem:almost} hold at
center $z$ and radius $R=2r$, with the same constants
for every such $z$. Now fix $x\in D_r(w)$ and choose a nearest boundary point $z$.
Since $w\in\partial\Om$,
\[
 |x-z|=\delta(x)\leq|x-w|<r,\qquad
 |z-w|\leq|z-x|+|x-w|<2r.
\]
Choose a dyadic radius $\rho=2^{-j}(2r)$ satisfying
$\delta(x)<\rho\leq2\delta(x)$. Since $\varepsilon_*<1/2$,
we have $\delta(x)>\varepsilon_*\rho$, and hence
$x\in\mathcal A_\rho^{\varepsilon_*}(z)$.
Iterating Lemma~\ref{lem:almost} gives
\[
 h(x)\geq M_*a_*^j>0.
\]
Consequently, $v(x)\leq Tu(x)$ throughout $D_r(w)$.
Interchanging $u$ and $v$ gives $u(x)\leq Tv(x)$.
Undoing the normalization proves \eqref{eq:comparison}.
\end{proof}

\begin{rem}\label{rem:comparison-inputs}
Proposition~\ref{prop:comparison} compares solutions that are already
given. Its proof uses local boundedness, interior Harnack, weak Harnack,
the zero-extension test, and NTA geometry. It does not require solving
the continuous Dirichlet problem first. The constants depend on the
uniform local estimates and the geometry, without a separate
$\mathrm L^2$ coercivity or energy Sobolev constant.
\end{rem}

\subsection{H\"older continuity of quotients and logarithms}\label{sec:quotient}

The final step is an oscillation calculation. We include it because a
two-sided bound for a quotient, without its scale iteration, does not
itself assert existence of a boundary limit. The following lemma converts boundary comparison into a contraction of
the quotient oscillation on a smaller boundary ball.

\begin{lem}\label{lem:quotient-contraction}
Let $z\in\partial\Om$, $0<R<R_0$, and let $q=u/v$ for positive $\La$-harmonic
functions $u,v$ in $D_R(z)$ with zero local energy boundary values on
$\partial\Om\cap B(z,R)$. Suppose $q$ is bounded in $D_R(z)$. Put
\[
 C_Q=T^2,\qquad \sigma={(C_Q-1)}/{(C_Q+1)}\in(0,1),
\]
where $T$ is the constant in Proposition~\ref{prop:comparison}.
Then
\begin{equation}\label{eq:quotient-contraction}
 \osc_{D_{R/L_0}(z)}q\leq\sigma\osc_{D_R(z)}q.
\end{equation}
Consequently, with $\alpha_B=-\log\sigma/\log L_0>0$,
\begin{equation}\label{eq:quotient-boundary-power}
 \osc_{D_\rho(z)}q
 \leq C(\rho/R)^{\alpha_B}\osc_{D_R(z)}q,
 \qquad 0<\rho\leq R.
\end{equation}
\end{lem}

\begin{proof}
Write
\[
 m=\inf_{D_R(z)}q,\quad M_0=\sup_{D_R(z)}q,\quad
 m'=\inf_{D_{R/L_0}(z)}q,\quad M_0'=\sup_{D_{R/L_0}(z)}q.
\]
Apply Proposition~\ref{prop:comparison} with radius $R/L_0$.
Its outer region is exactly $D_R(z)$ and its scale condition follows
from $R<R_0$. The functions $u-mv$ and $M_0v-u$ are nonnegative $\La$-harmonic functions
with the same zero local energy boundary values. To apply comparison
with strictly positive numerators, add $\epsilon v$ and then let
$\epsilon\downarrow0$. Comparing their quotients with $v$ at two
arbitrary points of $D_{R/L_0}(z)$ gives
\begin{equation}\label{eq:two-inequalities}
 M_0'-m\leq C_Q(m'-m),\qquad
 M_0-m'\leq C_Q(M_0-M_0').
\end{equation}
Indeed, each normalized comparison in
Proposition~\ref{prop:comparison} has factor $T$, so the
supremum-to-infimum comparison has factor $T^2=C_Q$.
Adding \eqref{eq:two-inequalities} and rearranging gives
\[
 (C_Q+1)(M_0'-m')\leq(C_Q-1)(M_0-m),
\]
which proves \eqref{eq:quotient-contraction}. Iteration at radii
$L_0^{-j}R$ and monotonicity prove
\eqref{eq:quotient-boundary-power}.
\end{proof}

\subsection{Proof of Theorem~\ref{thm:BHImain}}

Take $L=8L_0$, as fixed at the beginning of the section.
Then $4r<R_0/L_0$, and all applications of comparison at radius
$4r$ take place in $D_{4L_0r}(w)\subset D_{Lr}(w)$. Interior chains and the positive values at $A_r(w)$
give $u,v>0$ throughout $D_{6r}(w)$, including at $A_{4r}(w)$. Put
\[
 q=\frac uv,\qquad q_0=\frac{u(A_r(w))}{v(A_r(w))},\qquad \widehat q=\frac q{q_0}.
\]
Apply Proposition~\ref{prop:comparison} with radius $4r$.
Its comparison point is $A_{4r}(w)$, and $A_r(w)\in D_{4r}(w)$.
Comparing both $x$ and $A_r(w)$ with this point gives
\begin{equation}\label{eq:normalized-bounds}
 T^{-2}\leq \widehat q\leq T^2\quad\hbox{in }D_{4r}(w).
\end{equation}
For each $z\in\Delta_{2r}(w)$, $D_r(z)\subset D_{3r}(w)$, so
Lemma~\ref{lem:quotient-contraction} and
\eqref{eq:normalized-bounds} imply
\begin{equation}\label{eq:uniform-boundary-osc}
 \osc_{D_\rho(z)}\widehat q\leq C(\rho/r)^{\alpha_B},
 \qquad 0<\rho\leq r.
\end{equation}

There is an analogous interior estimate. Suppose that $u$ and $v$
are positive and $\La$-harmonic in $B(x,R)$ and that $\widehat q$
is bounded there. Put
\[
 m_B=\inf_{B(x,R)}\widehat q,\qquad
 M_B=\sup_{B(x,R)}\widehat q.
\]
Apply interior Harnack to the nonnegative harmonic functions
$u/q_0-m_Bv$ and $M_Bv-u/q_0$, and to $v$.
On $B(x,R/2)$ their quotients by $v$ have supremum-to-infimum
ratio at most $H_b^2$. Repeating the algebra
in \eqref{eq:two-inequalities} gives an oscillation contraction
with factor
\[
 \sigma_J={(H_b^2-1)}/{(H_b^2+1)}<1.
\]
Iteration yields, with $\alpha_J=-\log\sigma_J/\log2>0$,
\begin{equation}\label{eq:interior-quotient}
 \osc_{B(x,s)}\widehat q\leq C(s/R)^{\alpha_J}\osc_{B(x,R)}\widehat q,
 \qquad 0<s\leq R.
\end{equation}
An arbitrarily small addition of $v$ handles zero nonnegative numerators
as before.

Take $x,y\in D_r(w)$, set $d=|x-y|>0$, and choose a nearest boundary
point $z$ to $x$. If $d\geq r/16$, the desired bound follows at once
from \eqref{eq:normalized-bounds}. Suppose $d<r/16$.
If $\delta(x)\leq4d$, both points lie in $D_{6d}(z)$, and
$z\in\Delta_{2r}(w)$. Equation~\eqref{eq:uniform-boundary-osc}
therefore gives $|\widehat q(x)-\widehat q(y)|\leq C(d/r)^{\alpha_B}$. If $\delta(x)>4d$, the ball $B(x,\delta(x)/2)$ lies in $\Om$ and
in $D_{3\delta(x)/2}(z)$. By
\eqref{eq:uniform-boundary-osc}, or by
\eqref{eq:normalized-bounds} when $3\delta(x)/2>r$,
\[
 \osc_{B(x,\delta(x)/2)}\widehat q\leq C(\delta(x)/r)^{\alpha_B}.
\]
Apply \eqref{eq:interior-quotient} with $s=2d$ to obtain
\[
 |\widehat q(x)-\widehat q(y)|
 \leq C\left(\frac d{\delta(x)}\right)^{\alpha_J}
          \left(\frac{\delta(x)}r\right)^{\alpha_B}
 \leq C(d/r)^\alpha,
 \qquad \alpha=\min\{\alpha_B,\alpha_J,1/2\}.
\]
The last inequality uses $d<\delta(x)\leq r$ and is valid whichever
of $\alpha_B$ and $\alpha_J$ is smaller.

We have proved a uniform H\"older estimate for $\widehat q$ on $D_r(w)$.
It gives a unique extension to $\overline\Om\cap B(w,r)$ by taking
limits of interior sequences. The lower bound in
\eqref{eq:normalized-bounds} makes the extension positive.
Restoring $q_0$ proves \eqref{eq:quotientmain}, and comparison follows
from Proposition~\ref{prop:comparison}. Since the logarithm is
Lipschitz on $[T^{-2},T^2]$, the same estimate gives
\eqref{eq:logmain}. Finally, continuous vanishing implies the
energy boundary condition by Proposition~\ref{prop:continuouszero}.\qed

\section{Further examples and localization}\label{sec:examples}

The following examples clarify the role and scope of the coefficient
hypotheses. Local moment bounds
$|a|\in\mathrm L^p_{\loc}$ and $|a^{-1}|\in\mathrm L^q_{\loc}$,
even with $1/p+1/q<2/(n-1)$, need not ensure continuity of weak
solutions when the normalized moment products in
\eqref{eq:uniformmoments} are not uniformly bounded over cubes.
We then show that uniform coarse ellipticity can hold even when
the energy fails to control the $\mathrm L^2(\d x)$ norm on a
bounded ball. Finally, we specify the local coefficient and
geometric assumptions needed for the boundary arguments.

\subsection{A discontinuous solution with admissible coefficient moments}

Finite coefficient moments give finite coarse constants on individual
cubes, but do not provide uniform control across scales. The next example
shows that this distinction can obstruct even interior continuity.

\begin{prop}\label{prop:discontinuous}
For every $n\geq3$ there is a symmetric positive definite coefficient
field with $|a|\leq1$ and $|a^{-1}|\in \mathrm L^q_{\loc}(\R^n)$ for every
$1\leq q<n$, and a bounded weak solution in $B(0,1)$ that is discontinuous at
the origin. Its boundary values on $\partial B(0,1)$ are smooth.
In particular the moment condition
\begin{equation}\label{eq:optimalmoments}
 1/p+1/q<2/{(n-1)}
\end{equation}
does not, without a uniform scale condition, imply continuity of all weak
solutions or solvability of the continuous Dirichlet problem in the stated
continuous solution class.
\end{prop}

\begin{proof}
For $x\neq0$, put $r=|x|$ and let $\Pi(x)={(x\otimes x)}/{r^2}$ be the orthogonal projection onto the radial direction.
Set
\begin{equation}\label{eq:countercoefficient}
 a(x)=\Pi(x)+\frac r{1+r}(\Id-\Pi(x)),\qquad
 u(x)=x_1+{x_1}/{|x|}.
\end{equation}
The value of $a$ or $u$ at the origin is immaterial. The radial eigenvalue
of $a$ is one and its tangential eigenvalues are $r/(1+r)$. Hence
$|a|=1$ and $|a^{-1}|=1+r^{-1}$. The asserted integrability follows.
Direct differentiation gives
\begin{equation}\label{eq:counterflux}
 \nabla u=\Pi(x)e_1+(1+r^{-1})(\Id-\Pi(x))e_1,
 \qquad a\nabla u=e_1.
\end{equation}
Since $n\geq3$, $\nabla u\in \mathrm L^2(B(0,1))$. Smooth approximation in
$H^1$ is also approximation in the $a$ energy because $a\leq\Id$.
Thus $u$ is an approximable weak solution, and its constant flux in
\eqref{eq:counterflux} verifies the equation distributionally across the
origin. Along the two directions $e_1$ and $-e_1$, its limits at the
origin are $1$ and $-1$, respectively. On the unit sphere, $u=2x_1$.

One can choose $p=q=n-1/2$, which satisfies \eqref{eq:optimalmoments},
and then choose $s,t>0$ small enough to satisfy \eqref{eq:admissible}.
Lemma~\ref{lem:moments} gives finite coarse constants on every fixed cube.
Nevertheless they cannot satisfy \eqref{eq:UC}.
To see the scale obstruction directly, let $0<R<1$.
The function $h_R=u+1+R$ is nonnegative and $\La$-harmonic in $B(0,R)$, and
\begin{equation}\label{eq:counterHarnack}
 \frac{\esssup_{B(0,R/2)}h_R}{\essinf_{B(0,R/2)}h_R}
 =\frac{2+3R/2}{R/2}=\frac4R+3.
\end{equation}
The Harnack constant diverges as $R\downarrow0$.

The function $u-2x_1$ belongs to $H^1_0(B(0,1))$, and therefore to the
zero boundary energy space. Thus $u$ is the variational solution with
boundary datum $2x_1$. If there were another solution $v$ continuous on the closed
ball with these boundary values, $v-u$ would be a bounded local form
solution tending uniformly to zero near the sphere. For every
$\varepsilon>0$, its positive and negative truncations above level
$\varepsilon$ have compact support in the ball and can be used as tests.
They have zero energy and hence vanish. Thus $v=u$ almost everywhere,
which contradicts the existence of a continuous representative.
\end{proof}

\subsection{Failure of coercivity under uniform coarse ellipticity}

The following proposition shows that the coefficient hypotheses of
Theorems~\ref{thm:dirichletmain} and \ref{thm:BHImain} allow the bottom
of the form spectrum on $\mathrm L^2(\d x)$ to be zero even in a
bounded ball.

\begin{prop}\label{prop:nogap}
For every $n\geq3$, there exist parameters $(p,q,s,t)$ satisfying
\eqref{eq:admissible} and a scalar coefficient field satisfying
\eqref{eq:basic} and \eqref{eq:UC} such that
\begin{equation}\label{eq:zerogap}
 \lambda_1(B(0,1))
 :=\inf_{0\neq\phi\in C_0^\infty(B(0,1))}
       \frac{\cE(\phi,\phi)}{\|\phi\|_{\mathrm L^2}^2}=0.
\end{equation}
In particular, for this coefficient there are no $\chi>1$ and $S<\infty$
such that
\begin{equation}\label{eq:excludedSobolev}
 \|\phi\|_{\mathrm L^{2\chi}(B(0,1))}^2
 \leq S\cE(\phi,\phi),\qquad \phi\in C_0^\infty(B(0,1)).
\end{equation}
Nevertheless, Theorems~\ref{thm:dirichletmain} and \ref{thm:BHImain}
apply to this coefficient.
\end{prop}

\begin{proof}
Choose
\begin{equation}\label{eq:powerchoice}
 2<\gamma<{2n}/{(n-1)},\qquad a(x)=|x|^\gamma\Id.
\end{equation}
The value at the origin may be chosen arbitrarily positive definite.
Since $2n/(n-1)\leq n$ for $n\geq3$, the inverse coefficient is locally
integrable. Choose
\[
 \max\{1,(n-1)/2\}<q<n/\gamma
\]
and then choose $p>1$ sufficiently large that
$1/p+1/q<2/(n-1)$. These choices are possible by
\eqref{eq:powerchoice}. The normalized $\mathrm L^p$ moment of
$|x|^\gamma$ times the normalized $\mathrm L^q$ moment of
$|x|^{-\gamma}$ is uniformly bounded over cubes.
Indeed, on cubes far from the origin the weight is comparable to a
constant. On cubes at distance at most a fixed dimensional multiple of
their side length from the origin, dilation and $\gamma q<n$ bound
the two factors by $C\ell(Q)^\gamma$ and $C\ell(Q)^{-\gamma}$.
Choose small positive $s,t$ satisfying \eqref{eq:admissible}.
Lemma~\ref{lem:moments} proves \eqref{eq:UC}.

Take $0\neq\phi\in C_0^\infty(B(0,1))$ and set
$\phi_\epsilon(x)=\phi(x/\epsilon)$. Then
\begin{equation}\label{eq:gapcounter}
 \frac{\cE(\phi_\epsilon,\phi_\epsilon)}
      {\|\phi_\epsilon\|_{\mathrm L^2}^2}
 =\epsilon^{\gamma-2}
   \frac{\int |y|^\gamma|\nabla\phi(y)|^2\d y}
        {\int|\phi(y)|^2\d y}\longrightarrow0.
\end{equation}
This proves \eqref{eq:zerogap}. If \eqref{eq:excludedSobolev} held,
H\"older's inequality on $B=B(0,1)$ would give
\[
 \|\phi\|_{\mathrm L^2(B)}^2
 \leq |B|^{1-1/\chi}\|\phi\|_{\mathrm L^{2\chi}(B)}^2
 \leq |B|^{1-1/\chi}S\cE(\phi,\phi),
\]
contradicting \eqref{eq:zerogap}. The coefficient hypotheses of these
boundary theorems have already been checked.
\end{proof}

\begin{rem}
The weight in Proposition~\ref{prop:nogap} is classical and
belongs to $A_2$. The example shows that an
$\mathrm L^2(\d x)$ Poincar\'e inequality is not needed for
variational existence in the energy completion, for the
continuous Dirichlet construction, or for our direct boundary
comparison proof. The conclusions also apply to NTA domains
whose boundary passes through the origin. The vanishing spectral gap prevents a direct application of
\cite[Assumption~4.9 and Theorem~1.1]{BM} with reference measure
$\d x$. It does not exclude the existence of a Green kernel
or obstruct application of the later boundary Harnack theorem
\cite{ChenBHP}. The Green function theory developed in
\cite{BCM} does not require such a spectral gap: a Green kernel
need not define a bounded inverse on $\mathrm L^2(\d x)$.
\end{rem}

\subsection{Localization and the geometric assumptions}\label{sec:localization}

For the continuous Dirichlet theorem it is enough that the basic
integrability assumptions hold near $\overline\Om$ and that the uniform
coarse estimates hold on the ambient cubes used there. One may first
restrict to a relatively compact neighborhood of $\overline\Om$ and
extend the coefficient farther away, since no estimate there enters
the proof. For boundary comparison near a fixed boundary point, all
cubes and chains are contained in a fixed geometry-dependent enlargement
of the stated neighborhood. Thus uniformity on these cubes suffices.
An extension across the physical boundary with the required local
properties remains part of the hypotheses when only an interior
coefficient is prescribed.

The continuous Dirichlet argument uses only boundedness, connectedness,
and exterior volume density with a fixed positive fraction on the
scales considered. Its boundary regularity constants depend on that
fraction and the coefficient data. The comparison proof additionally uses
the interior corkscrews and Harnack chains. Its use of the exterior
corkscrew is confined to a positive Lebesgue measure fraction in
boundary cubes. Accordingly, the same proof works with the interior
NTA geometry and a quantitative exterior volume density assumption.
No boundary surface measure estimate is used. The exterior volume density condition is an input of this direct
proof. The abstract boundary Harnack theorem
\cite[Theorem~1.1]{ChenBHP} applies to uniform domains without
this exterior assumption.

The moment bound and exponent restriction in
Corollary~\ref{cor:momentcriterion} provide an explicit sufficient
criterion for \eqref{eq:UC}, recovering a classical $A_2$ weighted case.
The construction in Section~\ref{sec:coefficientclasses} verifies the
full coarse condition directly for a Cantor coefficient outside that
class. These verifications concern the ambient coefficient alone and
impose no additional
compatibility between its singular set and the NTA boundary. For other
coefficient models, uniform control of the variational quantities
remains the analytic hypothesis needed to apply the boundary argument.

\medskip
\noindent
{\bf Conflict of interest.}
The author states that there is no conflict of interest.

\medskip
\noindent
{\bf Data availability.}
Data sharing is not applicable to this article as no datasets
were generated or analysed during the current study.

\medskip
\noindent
{\bf Declaration of use of AI.}
The author used GPT-6 Astra, developed by OpenAI, during the preparation and revision of this manuscript.  The tool was
used for organization and for assistance with checking mathematical
arguments, clarity, and internal consistency.  The mathematical arguments, results,
and conclusions were developed and verified by the author, who also
checked the references and takes full responsibility for the content of
the manuscript.

\end{document}